\documentclass[10pt]{amsart}
\usepackage{fullpage, amsmath, amsthm,amsfonts,url,amssymb,stmaryrd, mathrsfs}
\usepackage{mathrsfs}
\usepackage{amssymb,amscd}
\usepackage{color}
\usepackage{extarrows}
\usepackage[colorlinks,linkcolor=red,anchorcolor=green,citecolor=blue]{hyperref}
\usepackage{enumitem}

\newtheorem{theorem}{Theorem}[section]
\newtheorem{corollary}[theorem]{Corollary}
\newtheorem{lemma}[theorem]{Lemma}
\newtheorem{proposition}[theorem]{Proposition}

\newtheorem{definition}[theorem]{Definition}
\newtheorem{hypothesis}[theorem]{Hypothesis}
\newtheorem{remark}[theorem]{Remark}
\newtheorem{notation}[theorem]{Notation}

\newcommand{\hooklongrightarrow}{\lhook\joinrel\longrightarrow}

\newcommand{\us}{\upsilon}
\newcommand{\ra}{\rightarrow}
\newcommand{\lra}{\longrightarrow}

\newcommand{\ul}{\underline}

\newcommand{\bA}{\mathbb A}
\newcommand{\bC}{\mathbb C}
\newcommand{\F}{\mathbb F}
\newcommand{\bG}{\mathbb G}
\newcommand{\bH}{\mathbb H}

\newcommand{\Q}{\mathbb Q}
\newcommand{\R}{\mathbb R}

\newcommand{\Z}{\mathbb Z}
\newcommand{\bS}{\mathbb S}

\newcommand{\cL}{\mathcal L}
\newcommand{\co}{\mathcal O}

\newcommand{\cR}{\mathcal R}
\newcommand{\cH}{\mathcal H}
\newcommand{\cC}{\mathcal C}
\newcommand{\cS}{\mathcal S}

\newcommand{\cW}{\mathcal W}
\newcommand{\cT}{\mathcal T}
\newcommand{\cM}{\mathcal M}

\newcommand{\cV}{\mathcal V}

\newcommand{\cU}{\mathcal U}
\newcommand{\cZ}{\mathcal Z}

\newcommand{\fh}{\mathfrak h}
\newcommand{\fm}{\mathfrak{m}}
\newcommand{\ub}{\mathfrak b}
\newcommand{\fl}{\mathfrak l}

\newcommand{\ug}{\mathfrak g}

\newcommand{\fx}{\mathfrak x}

\newcommand{\fs}{\mathfrak s}
\newcommand{\ft}{\mathfrak t}
\newcommand{\fa}{\mathfrak a}
\newcommand{\fz}{\mathfrak z}

\newcommand{\sN}{\mathscr N}

\DeclareMathOperator{\tr}{\mathrm tr}
\DeclareMathOperator{\GL}{\mathrm GL}

\DeclareMathOperator{\et}{\text{\'et}}
\DeclareMathOperator{\Fil}{\mathrm Fil}
\DeclareMathOperator{\Res}{\mathrm Res}

\DeclareMathOperator{\Sym}{\mathrm Sym}
\DeclareMathOperator{\Gal}{\mathrm Gal}
\DeclareMathOperator{\Hom}{\mathrm Hom}

\DeclareMathOperator{\End}{\mathrm End}
\DeclareMathOperator{\cris}{\mathrm cris}
\DeclareMathOperator{\rig}{\mathrm rig}
\DeclareMathOperator{\an}{\mathrm an}
\DeclareMathOperator{\Spec}{\mathrm Spec}

\DeclareMathOperator{\dR}{\mathrm dR}
\DeclareMathOperator{\Frob}{\mathrm Frob}
\DeclareMathOperator{\Ind}{\mathrm Ind}
\DeclareMathOperator{\unr}{\mathrm unr}

\DeclareMathOperator{\Ker}{\mathrm Ker}
\DeclareMathOperator{\pr}{\mathrm pr}

\DeclareMathOperator{\Ext}{\mathrm Ext}

\DeclareMathOperator{\Spm}{\mathrm Spm}

\DeclareMathOperator{\Ima}{\mathrm Im}
\DeclareMathOperator{\SL}{\mathrm SL}
\DeclareMathOperator{\lalg}{\mathrm lalg}

\DeclareMathOperator{\dett}{\mathrm det}

\DeclareMathOperator{\sss}{\mathrm ss}

\DeclareMathOperator{\red}{\mathrm red}

\DeclareMathOperator{\st}{\mathrm st}
\DeclareMathOperator{\St}{\mathrm St}
\DeclareMathOperator{\PGL}{\mathrm PGL}
\DeclareMathOperator{\Art}{\mathrm Art}
\DeclareMathOperator{\ab}{\mathrm ab}
\DeclareMathOperator{\WD}{\mathrm WD}
\DeclareMathOperator{\HT}{\mathrm HT}

\begin{document}
\title{$\cL$-invariants and local-global compatibility for the group $\GL_2/F$}
\author{Yiwen Ding}
\address{Department of Mathematics, Imperial College London}
\email{y.ding@imperial.ac.uk}
\maketitle
\begin{abstract}Let $F$ be a totally real number field, $\wp$ a place of $F$ above $p$. Let $\rho$ be a $2$-dimensional $p$-adic representation of $\Gal(\overline{F}/F)$ which appears in the \'etale cohomology of quaternion Shimura curves (thus $\rho$ is associated to Hilbert eigenforms). When the restriction  $\rho_{\wp}:=\rho|_{D_{\wp}}$ at the decomposition group of $\wp$ is semi-stable non-crystalline, one can associate to $\rho_{\wp}$ the so-called Fontaine-Mazur $\cL$-invariants, which are however invisible in the classical local Langlands correspondence. In this paper, we prove one can find these $\cL$-invariants in the completed cohomology group of quaternion Shimura curves, which generalizes some of Breuil's results \cite{Br10} in the $\GL_2/\Q$-case.
\end{abstract}
\tableofcontents
\section*{Introduction}
Let $F$ be a totally real number field, $B$ a quaternion algebra of center $F$ such that there exists only one real place of $F$ where $B$ is split. One can associate to $B$ a system of quaternion Shimura curves $\{M_K\}_{K}$,  proper and smooth over $F$, indexed by open compact subgroups $K$ of $(B\otimes_{\Q} \bA^{\infty})^{\times}$. We fix a prime number $p$, and suppose that there exists only one prime $\wp$ of $F$ above $p$. Suppose $B$ is split at $\wp$, i.e.  $(B\otimes_{\Q} \Q_p)^{\times} \cong \GL_2(F_{\wp})$ (where $F_{\wp}$ denotes the completion of $F$ at $\wp$). Let $E$ be a finite extension of $\Q_p$ sufficiently large with $\co_E$ its ring of integers and $\varpi_E$ a uniformizer of $\co_E$.

Let $\rho$ be a $2$-dimensional continuous representation of $\Gal(\overline{F}/F)$ over $E$ such that $\rho$ appears in the \'etale cohomology of $M_K$ for $K$ sufficiently small (so $\rho$ is associated to Hilbert eigenforms). By the theory of completed cohomology of Emerton (\cite{Em1}), one can associate to $\rho$ a unitary admissible Banach representation $\widehat{\Pi}(\rho)$ of $\GL_2(F_{\wp})$ as follows: put
\begin{equation*}
  \widehat{H}^1(K^p,E):=\Big(\varprojlim_{n}\varinjlim_{K_p'} H^1_{\et}\big(M_{K^pK_p'} \times_{F} \overline{F}, \co_E/\varpi_E^n\big)\Big)\otimes_{\co_E} E
\end{equation*}
where $K^p$ denotes the component of $K$ outside $p$, and $K_p'$ runs over open compact subgroups of $\GL_2(F_{\wp})$. This is an $E$-Banach space equipped with a continuous action of $\GL_2(F_{\wp}) \times \Gal(\overline{F}/F) \times \cH^{p}$ where $\cH^{p}$ denotes the $E$-algebra of Hecke operators outside $p$. Put
\begin{equation*}
  \widehat{\Pi}(\rho):=\Hom_{\Gal(\overline{F}/F)}\big(\rho, \widehat{H}^1(K^{p},E)\big).
\end{equation*}
The representation $\widehat{\Pi}(\rho)$  is supposed to be (a finite direct sum of) the right representation of $\GL_2(F_{\wp})$ corresponding to $\rho_{\wp}:=\rho|_{\Gal(\overline{F_{\wp}}/F_{\wp})}$ in the $p$-adic Langlands program (cf. \cite{Br0}). In nowadays, we know quite little about $\widehat{\Pi}(\rho)$, e.g. we don't know wether  it depends only on the local Galois representation $\rho_{\wp}$. By local-global compatibility of the classical local Langlands correspondence for $\GL_2/F$ (for $\ell=p$), one can indeed describe the locally algebraic vectors of $\widehat{\Pi}(\rho)$ in terms of the Weil-Deligne representation $\WD(\rho_{\wp})$ associated to $\rho_{\wp}$ and the Hodge-Tate weights $\HT(\rho_{\wp})$ of $\rho_{\wp}$ via the local Langlands correspondence. However, in general, (unlike the $\ell\neq p$ case), when passing to $\big(\WD(\rho_{\wp}), \HT(\rho_{\wp})\big)$, a lot of information about $\rho_{\wp}$ is lost. Finding the lost information in $\widehat{\Pi}(\rho)$ is thus one of the key problems in $p$-adic Langlands program (this is in fact the starting point of Breuil's initial work on $p$-adic Langlands program, cf. \cite{Br080}).

In this paper, we consider the case where $\rho_{\wp}$ is semi-stable non-crystalline and non-critical (i.e. $\rho_{\wp}$ satisfies the hypothesis \ref{hyp: clin-aq0}). In this case, the missing data, when passing from $\rho_{\wp}$ to $\big(\WD(\rho_{\wp}), \HT(\rho_{\wp})\big)$, can be explicitly described by the so-called \emph{Fontaine-Mazur $\cL$-invariants} $\ul{\cL}_{\Sigma_{\wp}}=(\cL_{\sigma})_{\sigma\in \Sigma_{\wp}}\in E^d$ associated to $\rho_{\wp}$ (e.g. see \S \ref{sec: clin-ene}), where $\Sigma_{\wp}$ denotes the set of $\Q_p$-embeddings of $F_{\wp}$ in $\overline{\Q_p}$. Using these $\cL$-invariants, Schraen has associated to $\rho_{\wp}$ a locally $\Q_p$-analytic representation $\Sigma\big(\WD(\rho_{\wp}), \HT(\rho_{\wp}), \ul{\cL}_{\Sigma_{\wp}}\big)$ of $\GL_2(F_{\wp})$ over $E$ (cf. \cite[\S 4.2]{Sch10}, see also \S \ref{sec: clin-4.2}), which generalizes Breuil's theory \cite{Br04} in $\GL_2(\Q_p)$-case. Note that one can indeed recover $\rho_{\wp}$ from $\Sigma\big(\WD(\rho_{\wp}), \HT(\rho_{\wp}), \ul{\cL}_{\Sigma_{\wp}}\big)$. The main result of this paper is the
\begin{theorem}[cf. $\text{Thm.\ref{thm: clin-sio}}$]\label{thm: clin0}
Keep the above notation and suppose that $\rho$ is absolutely irreducible modulo $\varpi_E$, there exists a continuous injection of $\GL_2(F_{\wp})$-representations
\begin{equation*}
  \Sigma\big(\WD(\rho_{\wp}), \HT(\rho_{\wp}), \ul{\cL}_{\Sigma_{\wp}}\big) \hooklongrightarrow \widehat{\Pi}(\rho)_{\Q_p-\an},
\end{equation*}
where $\widehat{\Pi}(\rho)_{\Q_p-\an}$ denotes the locally $\Q_p$-analytic vectors of $\widehat{\Pi}(\rho)$.
\end{theorem}
Such a result is called local-global compatibility, since the left side of this injection depends only on the local representation $\rho_{\wp}$ while the right side is globally constructed. Moreover, one can prove the ``uniqueness" (in the sense of Cor.\ref{cor: clin-ape}) of $\Sigma\big(\WD(\rho_{\wp}), \HT(\rho_{\wp}), \ul{\cL}_{\Sigma_{\wp}}\big)$ as subrepresentation of $\widehat{\Pi}(\rho)_{\Q_p-\an}$. As a result, we see the local Galois representation $\rho_{\wp}$ is determined by $\widehat{\Pi}(\rho)$. Such a  result in the $\Q_p$-case,  proved by Breuil (\cite{Br10}),  was the first discovered local-global compatibility in the $p$-adic local Langlands correspondence. In fact, the $\cL$-invariants appearing in the automorphic representation side are often referred to as \emph{Breuil's $\cL$-invariants}. The theorem \ref{thm: clin0} thus shows the equality of Fontaine-Mazur $\cL$-invariants and Breuil's $\cL$-invariants. Our approach is by using some $p$-adic family arguments on both $\GL_2$-side and Galois side, thus different from that of Breuil (by using modular symbols).

In the following (of the introduction), we sketch how we manage to ``find" $\{\cL_{\sigma}\}_{\sigma\in \Sigma_{\wp}}$ in $\widehat{\Pi}(\rho)$. For simplicity, suppose $\rho_{\wp}$ is of Hodge-Tate weights $(-1,0)_{\Sigma_{\wp}}$ (thus $\rho_{\wp}$ is associated to Hilbert eigenforms of weights $(2,\cdots, 2; 0)$ in the notation of \cite{Ca2}). Let $\tau\in \Sigma_{\wp}$, it's enough to find $\cL_{\tau}$ in $\widehat{\Pi}(\rho)_{\tau-\an}$ \big(the maximal  locally $\tau$-analytic subrepresentation of $\widehat{\Pi}(\rho)$\big) in the sense of (\ref{equ: clin-utau}) below:

Denote by $Z_1:=\bigg\{\begin{pmatrix}
  a & 0 \\ 0 & a
\end{pmatrix}\ \Big|\ a\in 1+2\varpi \co_{\wp}\bigg\}$ (where $\co_{\wp}$ denotes the ring of integers of $F_{\wp}$ and $\varpi$ is a uniformizer of $\co_{\wp}$), consider $\widehat{H}^1(K^p,E)_{\tau-\an}^{Z_1}$ \big(where ``$(\cdot)^{Z_1}$" signifies the vectors fixed by $Z_1$, and ``$\tau-\an$" signifies the locally $\tau$-analytic subrepresentation\big). By applying Jacquet-Emerton functor, one gets an essentially admissible locally $\tau$-analytic representation of $T(F_{\wp})$: $J_B(\widehat{H}^1(K^p,E)_{\tau-\an}^{Z_1})$, which is moreover equipped with an action of $\cH^p$ commuting with that of $T(F_{\wp})$. Following Emerton, one can construct  an eigenvariety $\cV_{\tau}$ from $J_B(\widehat{H}^1(K^p,E)_{\tau-\an}^{Z_1})$, which is in particular a rigid space finite over $\widehat{T}_{\tau}$, the rigid space parameterizing locally $\tau$-analytic characters of $T(F_{\wp})$ (cf. Thm.\ref{thm: clin-cjw}). A closed point of $\cV_{\tau}$ can be written as $(\chi,\lambda)$ where $\chi$ is a locally $\tau$-analytic character of $T(F_{\wp})$ and $\lambda$ is a system of Hecke eigenvalues (for $\cH^p$).

One can associate to $\rho$ an $E$-point $z_{\rho}=(\chi_{\rho}, \lambda_{\rho})$ of $\cV_{\tau}$, where $\chi_{\rho}=\unr(\alpha/q)\otimes \unr(q\alpha)$ \big($\unr(a)$ denotes the unramified character of $F_{\wp}^{\times}$ sending $\varpi$ to $a$\big), $\lambda_{\rho}$ denotes the system of eigenvalues of $\cH^p$ associated to $\rho$ (via the Eichler-Shimura relations), $\{\alpha,q\alpha\}$ are the eigenvalues of $\varphi^{d_0}$ on $D_{\st}(\rho_{\wp})$ (where $d_0$ is the degree of the maximal unramified extension of $\Q_p$ in $F_{\wp}$, $q:=p^{d_0}$). Moreover, by multiplicity one result on automorphic representations of $(B\otimes_{\Q} \bA)^{\times}$, one can prove as in \cite[\S 4.4]{Che11} that $\cV_{\tau}$ is smooth at $z_{\rho}$ (cf. Thm.\ref{thm: clin-elt}, note that by the hypothesis \ref{hyp: clin-aq0}, $z_{\rho}$ is in fact a \emph{non-critical} point).

Let $t:\Spec E[\epsilon]/\epsilon^2 \ra \cV_{\tau}$ be a non-zero element in the tangent space of  $\cV_{\tau}$ at $z_{\rho}$, via the composition
\begin{equation*}
  t: \Spec E[\epsilon]/\epsilon^2 \lra \cV_{\tau} \lra \widehat{T}_{\tau},
\end{equation*}
one gets a character $\widetilde{\chi}_{\rho}=\widetilde{\chi}_{\rho,1}\otimes \widetilde{\chi}_{\rho,2}: T(F_{\wp})^{\times} \ra (E[\epsilon]/\epsilon^2)^{\times}$, which is in fact an extension of $\chi_{\rho}$ by $\chi_{\rho}$. One key point is that, by applying an adjunction formula in family for the Jacquet-Emerton functor (see \cite[Lem.4.5.12]{Em1} for the $\GL_2(\Q_p)$-case) to the tangent space of $\cV_{\tau}$ at $z_{\rho}$, one gets a non-zero continuous morphism of $\GL_2(F_{\wp})$-representations (where $\overline{B}(F_{\wp})$ denotes the group of lower triangular matrices) (see (\ref{equ: clin-pqgi}))
\begin{equation}\label{equ: clin-pf2}
  \big(\Ind_{\overline{B}(F_{\wp})}^{\GL_2(F_{\wp})} \widetilde{\chi}_{\rho}\delta^{-1}\big)^{\tau-\an} \lra \widehat{H}^1(K^p,E)_{\tau-\an}^{Z_1}[\lambda_{\rho}]
\end{equation}
where $\delta:=\unr(q^{-1}) \otimes \unr(q)$ and we refer to \cite[\S 2]{Sch10} for locally $\tau$-analytic parabolic inductions, and where the right term denotes the generalized $\lambda_{\rho}$-eigenspace of $\widehat{H}^1(K^p,E)_{\tau-\an}^{Z_1}$.

Another key point is that one can describe the character $\widetilde{\chi}_{\rho}$ in term of $\cL_{\tau}$:
\begin{lemma}[$\text{cf. Lem.\ref{thm: clin-aue}}$]\label{lem: clin-htn}
  There exists an additive character $\chi$ of $F_{\wp}^{\times}$ in $E$ such that $\widetilde{\chi}_{\rho}$ (as a $2$-dimensional representation of $T(F_{\wp})$ over $E$) is isomorphic to $\chi_{\rho}\otimes_E \psi(\cL_{\tau},\chi)$ where
  \begin{equation*}
    \psi(\cL_{\tau},\chi)\begin{pmatrix} a & 0 \\ 0 & d \end{pmatrix}=\begin{pmatrix}1 & \log_{\tau,-\cL_{\tau}}(ad^{-1})+\chi (ad) \\ 0 & 1\end{pmatrix},
  \end{equation*}
  and  $\log_{\tau,\cL}$ denotes the additive character of $F_{\wp}^{\times}$ such that $\log_{\tau,\cL}|_{\co_{\wp}^{\times}}=\tau \circ \log$ and $\log_{\tau,\cL}(p)=\cL$.
\end{lemma}
To prove this lemma, one considers the $p$-adic family of Galois representations over $\cV_{\tau}$. In fact, there exist an admissible neighborhood $U$ of $z_{\rho}$ in $\cV_{\tau}$ and a continuous representation $\rho_U: \Gal(\overline{F}/F) \ra \GL_2(\co_U)$ such that the evaluation of $\rho_U$ at any classical point of $U$ (which thus corresponds to certain Hilbert eigenforms $h$) is just the Galois representation associated to $h$. Via the map $t$, one gets a continuous representation $\widetilde{\rho}: \Gal(\overline{F}/F) \ra \GL_2(E[\epsilon]/\epsilon^2)$ which satisfies $\widetilde{\rho}\equiv \rho \pmod{\epsilon}$. By the theory of global triangulation \cite{KPX}, one can obtain an exact sequence (cf. (\ref{equ: clin-2no})):
\begin{equation*}
  0 \ra \cR_{E[\epsilon]/\epsilon^2}\big(\unr(q)\widetilde{\chi}_{\rho,1}\big) \ra D_{\rig}(\widetilde{\rho}_{\wp}) \ra \cR_{E[\epsilon]/\epsilon^2}\Big(\widetilde{\chi}_{\rho,2}\prod_{\sigma\in \Sigma_{\wp}}\sigma^{-1}\Big)\ra 0,
\end{equation*}
where $\widetilde{\rho}_{\wp}:=\widetilde{\rho}|_{\Gal(\overline{F_{\wp}}/F_{\wp})}$. The lemma then follows by applying the formula in \cite[Thm.1.1]{Zhang} (which generalizes Colmez's formula \cite{Colm10} in $\Q_p$-case) to $\widetilde{\rho}_{\wp}$.

Return to the map (\ref{equ: clin-pf2}). We know  $\big(\Ind_{\overline{B}(F_{\wp})}^{\GL_2(F_{\wp})} \widetilde{\chi}_{\rho}\delta^{-1}\big)^{\tau-\an}$ lies in an exact sequence
\begin{equation*}
  0 \ra \big(\Ind_{\overline{B}(F_{\wp})}^{\GL_2(F_{\wp})} \chi_{\rho}\delta^{-1}\big)^{\tau-\an}\ra \big(\Ind_{\overline{B}(F_{\wp})}^{\GL_2(F_{\wp})} \widetilde{\chi}_{\rho}\delta^{-1}\big)^{\tau-\an} \xrightarrow{s}  \big(\Ind_{\overline{B}(F_{\wp})}^{\GL_2(F_{\wp})} \chi_{\rho}\delta^{-1}\big)^{\tau-\an} \ra 0.
\end{equation*}
where $s$ depends on $\cL_{\tau}$ and $\chi$ as in the lemma \ref{lem: clin-htn}. On the other hand, it's known that $\big(\Ind_{\overline{B}(F_{\wp})}^{\GL_2(F_{\wp})} \chi_{\rho}\delta^{-1}\big)^{\tau-\an}$ admits a unique finite dimensional subrepresentation $V(\alpha):=\unr(\alpha)\circ \dett$. Put $\Sigma(\alpha,\cL_{\tau}):=s^{-1}(V(\alpha))/V(\alpha)$ (cf. \cite[\S 4.2]{Sch10}), which turns out to be \emph{independent} of the character $\chi$ in Lem.\ref{lem: clin-htn} and thus depends \emph{only} on $\cL_{\tau}$.
At last, one can prove that (\ref{equ: clin-pf2}) induces actually  a continuous injection of locally $\tau$-analytic representations of $\GL_2(F_{\wp})$
\begin{equation}\label{equ: clin-utau}
\Sigma(\alpha,\cL_{\tau}) \hooklongrightarrow \widehat{\Pi}(\rho)_{\tau-\an}.
\end{equation}

It seems this argument might work for some other groups and some other Shimura varieties. For example, in the $\GL_2/\Q$-case (with Coleman-Mazur eigencurve, reconstructed by Emerton \cite[\S 4]{Em1} using completed cohomology of modular curves), by restricting the map \cite[(4.5.9)]{Em1} to the tangent space at a semi-stable non-crystalline point, one can obtain a map as in (\ref{equ: clin-pf2}). On the other hand, one can prove a similar result as in Lem.\ref{lem: clin-htn} by Kisin's theory in \cite{Ki} and Colmez's formula \cite{Colm10}. Combining them together, one can actually reprove Breuil's result in \cite{Br10} for locally analytic representations and thus obtain directly the equality of Fontaine-Mazur $\cL$-invariant and Breuil's $\cL$-invariant without using Darmon-Orton's $\cL$-invariant (as in Breuil's original proof \cite{Br10}).

We refer to the body of the text for more detailed and more precise statements.

After the results of this paper was firstly announced, Yuancao Zhang informed us that he had proved the existence of $\cL$-invariants in $\widehat{\Pi}(\rho)$ in certain cases by using some arguments as in \cite[\S 5]{BE}, however, the equality between these $\cL$-invariants and Fontaine-Mazur $\cL$-invariants was not proved.
\addtocontents{toc}{\protect\setcounter{tocdepth}{1}}
\subsection*{Acknowledgements} I would like to thank Santosh Nadimpalli, Benjamin Schraen, Yichao Tian, Yuancao Zhang for useful discussions or answering my questions during the preparation of this paper. In particular, I would like to thank Christophe Breuil for a careful reading of a preliminary version, for his remarks on this work, which help improve a lot of the text, and for suggesting the proof of Cor.\ref{cor: clin-ape}.
\addtocontents{toc}{\protect\setcounter{tocdepth}{1}}
\section{Notations and preliminaries}\label{sec: clin-1}Let $F$ be a totally real field of degree $d$ over $\Q$, denote by $\Sigma_{\infty}$ the set of real embeddings of $F$.  For a finite place $\fl$ of $F$, we denote by $F_{\fl}$ the completion of $F$ at $\fl$, $\co_{\fl}$ the ring of integers of $F_{\fl}$ with $\varpi_{\fl}$ a uniformiser of $\co_{\fl}$. Denote by $\bA$ the ring of adeles of $\Q$ and $\bA_F$ the ring of adeles of $F$. For a set $S$ of places of $\Q$ (resp. of $F$), we denote by $\bA^S$ (resp. by $\bA_F^S$) the ring of adeles of $\Q$ (resp. of $F$) outside $S$, $S_F$ the set of places of $F$ above that in $S$, and $\bA_F^S:=\bA_F^{S_F}$.

Let $p$ be a prime number, suppose there exists only one prime $\wp$ of $F$ lying above $p$. Denote by $\Sigma_{\wp}$ the set of $\Q_p$-embeddings of $F_{\wp}$ in $\overline{\Q_p}$; let $\varpi$ be a uniformizer of $\co_{\wp}$, $F_{\wp,0}$ the maximal unramified extension of $\Q_p$ in $F_{\wp}$, $d_0:=[F_{\wp,0}:\Q_p]$, $e:=[F_{\wp}:F_{\wp,0}]$, $q:=p^{d_0}$ and $\us_{\wp}$ a $p$-adic valuation on $\overline{\Q_p}$ normalized by $\us_{\wp}(\varpi)=1$. Let $E$ be a finite extension of $\Q_p$ big enough such that $E$ contains all the $\Q_p$-embeddings of $F$ in $\overline{\Q_p}$, $\co_E$ the ring of integers of $E$ and $\varpi_E$ a uniformizer of $\co_E$.

Let $B$ be a quaternion algebra of center $F$ with $S(B)$ the set (of even cardinality) of places of $F$ where $B$ is ramified, suppose $|S(B)\cap \Sigma_{\infty}|=d-1$ and $S(B)\cap \Sigma_{\wp}=\emptyset$, i.e. there exists $\tau_{\infty}\in \Sigma_{\infty}$ such that $B\otimes_{F,\tau_{\infty}} \R\cong M_2(\R)$, $B\otimes_{F,\sigma} \R \cong \bH$ for any $\sigma\in \Sigma_{\infty}$, $\sigma\neq \tau_{\infty}$, where $\bH$ denotes the Hamilton algebra, and $B\otimes_{\Q} \Q_p\cong M_2(F_{\wp})$. We associate to $B$ a reductive algebraic group $G$ over $\Q$ with $G(R):=(B\otimes_{\Q} R)^{\times}$ for any $\Q$-algebra $R$. Set $\bS:=\Res_{\bC/\R}\bG_m$, and denote by $h$ the morphism
\begin{equation*}
  h: \bS(\R)\cong \bC^{\times} \lra G(\R)\cong \GL_2(\R)\times (\bH^*)^{d-1}, \ a+bi\mapsto \bigg(\begin{pmatrix}a&b\\ -b&a\end{pmatrix}, 1, \cdots, 1\bigg).
\end{equation*}
The space of $G(\R)$-conjugacy classes of $h$ has a structure of complex manifold, and is isomorphic to $\fh^{\pm}:=\bC \setminus \R$ (i.e. $2$ copies of the Poincar\'e's upper half plane). We get a projective system of Riemann surfaces indexed by open compact subgroups of $G(\bA^{\infty})$:
\begin{equation*}
  M_K(\bC):=G(\Q)\setminus \big(\fh^{\pm} \times (G(\bA^{\infty})/K)\big)
\end{equation*}
where $G(\Q)$ acts on $\fh^{\pm}$ via $G(\Q)\hookrightarrow G(\R)$ and the transition map is given by
\begin{equation}\label{equ: clin-sab}
G(\Q)\setminus \big(\fh^{\pm} \times (G(\bA^{\infty})/K_1)\big) \lra G(\Q)\setminus \big(\fh^{\pm} \times (G(\bA^{\infty})/K_2)\big), \ (x,g)\mapsto (x,g),
\end{equation} for $K_1\subseteq K_2$. It's known that $M_K(\bC)$ has a canonical proper smooth model over $F$ (via the embedding $\tau_{\infty}$), denoted by $M_K$, and these $\{M_K\}_{K}$ form a projective system of proper smooth algebraic curves over $F$ (i.e. the transition map (\ref{equ: clin-sab}) admits also an $F$-model).  
One has a natural isomorphism $G(\Q_p) \xlongrightarrow{\sim} \GL_2(F_{\wp})$.
Let $K_{0,\wp}:=\GL_2(\co_{\wp})$, in the following, we fix an open compact subgroup $K^{p}$ of $G(\bA^{\infty,p})$ small enough such that the open compact subgroup $K^pK_{0,\wp}$ of $G(\bA^{\infty})$ is neat (cf. \cite[Def.4.11]{New}). Denote by $S(K^{p})$ the set of finite places $\fl$ of $F$ such that $\fl \nmid p$, that $B$ is split at $\fl$, i.e. $B\otimes_F F_{\fl}\xlongrightarrow{\sim}  \mathrm{M}_2(F_{\fl})$, and that $K^{p}\cap \GL_2(F_{\fl})\cong \GL_2(\co_{\fl})$. Denote by $\cH^p$ the commutative $\co_E$-algebra generated by the double coset operators $[\GL_2(\co_{\fl}) g_{\fl} \GL_2(\co_{\fl})]$ for all $g_{\fl}\in \GL_2(F_{\fl})$ with $\dett(g_{\fl})\in \co_{\fl}$ and for all $\fl \in S(K^p)$. Set
\begin{eqnarray*}T_{\fl}&:=&\bigg[\GL_2(\co_{\fl})\begin{pmatrix} \varpi_{\fl} & 0 \\ 0 & 1\end{pmatrix}\GL_2(\co_{\fl})\bigg], \\
S_{\fl} &:=& \bigg[\GL_2(\co_{\fl})\begin{pmatrix} \varpi_{\fl} & 0 \\ 0 & \varpi_{\fl}\end{pmatrix}\GL_2(\co_{\fl})\bigg],
\end{eqnarray*}
then $\cH^p$ is the polynomial algebra over $\co_E$ generated by $\{T_{\fl}, S_{\fl}\}_{\fl \in S(K^p)}$.

Denote by $Z_0$ the kernel of the norm map $\sN: \Res_{F/\Q} \bG_m \ra \bG_m$ which is a subgroup of $Z=\Res_{F/\Q} \bG_m$. We set $G^c:=G/Z_0$.



Denote by $\Art_{F_{\wp}}: F_{\wp}^{\times}\xrightarrow{\sim} W_{F_\wp}^{\ab}$ the local Artin map normalized by sending uniformizers to geometric Frobenius elements \big(where $W_{F_{\wp}}\subset \Gal(\overline{\Q_p}/F_{\wp})$ denotes the Weil group\big). 
Let $\sigma\in \Sigma_{\wp}$, denote by $\log_{\sigma}$ the composition $\co_{\wp}^{\times} \xrightarrow{\log} \co_{\wp} \xrightarrow{\sigma} E$. For $\cL\in E$, denote by $\log_{\sigma,\cL, \varpi}$ the (additive) character of $F_{\wp}^{\times}$ such that $\log_{\sigma,\cL,\varpi}|_{\co_{\wp}^{\times}}=\log_{\sigma}$ and $\log_{\sigma,\cL,\varpi}(\varpi)=\cL$.  Denote by $\log_{\sigma,\cL}$ the (additive) character of $F_{\wp}^{\times}$ in $E$ satisfying $\log_{\sigma,\cL}|_{\co_{\wp}^{\times}}=\log_\sigma$ and $\log_{\sigma,\cL}(p)=\cL$. Let $\cL(\varpi):=e^{-1}\big(\cL-\log_{\sigma}\big(\frac{p}{\varpi^{e}}\big)\big)$, thus one has
\begin{equation*}
  \log_{\sigma,\cL}=\log_{\sigma, \cL(\varpi), \varpi}.
\end{equation*}
Denote by $\unr(a)$ the unramified character of $F_{\wp}^{\times}$ sending $\varpi$ to $a$.

Let $V$ be an $E$-vector space equipped with an $E$-linear action of $A$ (with $A$ a set of operators), $\chi$ a system of eigenvalues of $A$, denote by $V^{A=\chi}$ the $\chi$-eigenspace, $V[A=\chi]$ the generalized $\chi$-eigenspace, $V^A$ the vector space of $A$-fixed vectors.

Let $S\subseteq \Sigma_\wp$, $k_{\sigma}\in \Z_{\geq 2}$ for all $\sigma\in S$, denote by $W(\ul{k}_S):=\otimes_{\sigma\in S} \big(\Sym^{k_{\sigma}-2} E^2\big)^{\sigma}$ the algebraic representation of $G(\Q_p)\cong\GL_2(F_{\wp})$ with $\GL_2(F_{\wp})$ acting on $\big(\Sym^{k_{\sigma}-2} E^2\big)^{\sigma}$ via $ \GL_2(F_{\wp}) \xrightarrow{\sigma} \GL_2(E)$ for $\sigma\in S$. Let $w\in \Z$, $k_{\sigma}\in \Z_{\geq 2}$, $k_{\sigma}\equiv w \pmod{2}$ for all $\sigma\in \Sigma_{\wp}$, put $W(\ul{k}_{\Sigma_{\wp}},w):=\otimes_{\sigma\in \Sigma_{\wp}} \big(\Sym^{k_{\sigma}-2} E^2\otimes \dett^{\frac{w-k_{\sigma}+2}{2}}\big)^{\sigma}$.

Denote by  $B(F_{\wp})$ \big(resp. $\overline{B}(F_{\wp})$\big) the subgroup of $\GL_2(F_{\wp})$ of upper (resp. lower) triangular matrices, $T(F_{\wp})$ the group of diagonal matrices, $N(F_{\wp})$ the group of unipotent elements in $B(F_{\wp})$, $N_0:=N(F_{\wp})\cap \GL_2(\co_{\wp})$, $Z':=T(F_{\wp})\cap \SL_2(F_{\wp})$,  $K_{1,\wp}:=\{g\in \GL_2(\co_{\wp})\ |\ g\equiv 1 \pmod{2\varpi} \}$, $Z_{1}$ the center of $K_{1,\wp}$, $Z_{1}':=Z'\cap K_{1,\wp}$. Put $\delta:=\unr(q^{-1}) \otimes \unr(q)$ being a character of $T(F_{\wp})$ (which is in fact the modulus character of $B(F_{\wp})$).
\subsection*{Locally $\Q_p$-analytic representations of $\GL_2(F_{\wp})$}
Recall some notions on locally $\Q_p$-analytic representations. Let $V$ be a locally $\Q_p$-analytic representation of $\GL_2(F_{\wp})$ over $E$, i.e. a locally analytic representation of $\GL_2(F_{\wp})$ with $\GL_2(F_{\wp})$ viewed as a $p$-adic $\Q_p$-analytic group, $V$ is naturally equipped with a $\Q_p$-linear action of the Lie algebra $\ug$ of $\GL_2(F_{\wp})$ \big(thus an $E$-linear action of $\ug_{\Sigma_{\wp}}:=\ug\otimes_{\Q_p} E$\big) given by
\begin{equation*}
  \fx\cdot v:=\frac{d}{dt}\exp(t\fx)(v)|_{t=0}.
\end{equation*} Using the isomorphism
\begin{equation}\label{equ: clin2-grwm}
  F_{\wp}\otimes_{\Q_p} E \xlongrightarrow{\sim} \prod_{\sigma\in \Sigma_{\wp}} E, \  a \otimes b \mapsto (\sigma(a)b)_{\sigma: F_{\wp}\ra E},
\end{equation}
one gets a decomposition $\ug_{\Sigma_{\wp}}\xrightarrow{\sim} \prod_{\sigma\in \Sigma_{\wp}} \ug_{\sigma}$ with $\ug_{\sigma}:=\ug \otimes_{F_{\wp},\sigma}E$. Let $J\subseteq \Sigma_{\wp}$, a vector $v\in V$ is called \emph{locally $J$-analytic} if the action of $\ug_{\Sigma_{\wp}}$ on $v$ factors through $\ug_{J}:=\prod_{\sigma\in J} \ug_{\sigma}$ (we put $\ug_{\emptyset}:=\{0\}$), in other words, if $v$ is killed by $\ug_{\Sigma_{\wp}\setminus J}$  (cf. \cite[Def.2.4]{Sch10}); $v$ is called \emph{quasi-$J$-classical} if there exist a finite dimensional representation $U$ of $\ug_J$ and a $\ug_J$-invariant map
\begin{equation*}\label{equ: drp-var}
  U \hooklongrightarrow  V
\end{equation*}
whose image contains $v$, if the $\ug_J$-representation $U$ can moreover give rise to an algebraic representation of $\GL_2(F_{\wp})$, then we say that $v$ is \emph{$J$-classical}. In particular, $v$ is $\Sigma_{\wp} \setminus J$-classical if $v$ is locally $J$-analytic. Note that $v$ is $\emptyset$-analytic is equivalent to that $v$ is a smooth vector for the action of $\GL_2(F_{\wp})$ \big(i.e. $v$ is fixed by certain open compact subgroup of $\GL_2(F_{\wp})$\big) which implies in particular $v$ is $\Sigma_{\wp}$-classical.

Let $V$ be a Banach representation of $\GL_2(F_{\wp})$ over $E$, denote by $V_{\Q_p-\an}$ the $E$-vector subspace generated by the locally $\Q_p$-analytic vectors of $V$, which is stable by $\GL_2(F_{\wp})$ and hence is a locally $\Q_p$-analytic representation of $\GL_2(F_{\wp})$. If $V$ is moreover admissible, by \cite[Thm.7.1]{ST03}, $V_{\Q_p-\an}$ is an admissible locally $\Q_p$-analytic representation of $\GL_2(F_{\wp})$ and  dense in $V$. For $J\subseteq \Sigma_{\wp}$, denote by $V_{J-\an}$ the subrepresentation generated by locally $J$-analytic vectors of $V_{\Q_p-\an}$, put $V_{\infty}:=V_{\emptyset-\an}$.

Let $\chi$ be a continuous (thus locally $\Q_p$-analytic) character of $F_{\wp}^{\times}$ (or any open compact subgroup of $F_{\wp}^{\times}$) over $E$, then $\chi$ induces a natural $\Q_p$-linear map (where $F_{\wp}$ is viewed as the Lie algebra of $F_{\wp}^{\times}$)
\begin{equation*}
  F_{\wp}\lra E,\  \fx \mapsto \frac{d}{dt}\chi\big(\exp(t\fx)\big)|_{t=0},
\end{equation*}
and hence an $E$-linear map $d_{\chi}: F_{\wp}\otimes_{\Q_p} E\cong \prod_{\sigma\in \Sigma_{\wp}} E \ra E$. So there exist $k_{\chi,\sigma}\in E$, called the $\sigma$-weight of $\chi$,  for all $\sigma\in \Sigma_{\wp}$ such that $d_\chi\big((a_{\sigma})_{\sigma\in \Sigma_{\wp}}\big)=\sum_{\sigma\in \Sigma_{\wp}} a_{\sigma}k_{\chi,\sigma}$.

Let $\chi=\chi_1 \otimes \chi_2 $ be a locally $\Q_p$-analytic character of $T(F_{\wp})$ over $E$. Put
\begin{equation}\label{equ: clin-cchi}C(\chi):=\{\sigma\in \Sigma_{\wp}\ |\ k_{\chi_1,\sigma}-k_{\chi_2,\sigma}\in \Z_{\geq 0}\}.\end{equation}
Denote by $\ft$ the Lie algebra of $T(F_{\wp})$, the character $\chi$ induces a character $d\chi$ of $\ft_{\Sigma_{\wp}}:=\ft\otimes_{\Q_p} E$ given by $d\chi: \ft_{\Sigma_{\wp}} \ra E$, $d\chi \begin{pmatrix}
  a_{\sigma} & 0 \\ 0 & d_{\sigma}
\end{pmatrix}=a_{\sigma}k_{\chi_1,\sigma}+d_{\sigma}k_{\chi_2,\sigma}$ for $ \begin{pmatrix}
  a_{\sigma} & 0 \\ 0 & d_{\sigma}
\end{pmatrix}\in \ft_{\sigma}:=\ft \otimes_{F_\wp,\sigma} E$, $\sigma\in \Sigma_{\wp}$.
\addtocontents{toc}{\protect\setcounter{tocdepth}{2}}
\section{Completed cohomology of quaternion Shimura curves}\label{sec: clin-2}
Recall some facts on completed cohomology of quaternion Shimura curves, following \cite{Em1} and \cite{New}.
\subsection{Generalities}\label{sec: clin-2.1}
Let $W$ be a finite dimensional algebraic representation of $G^c$ over $E$, as in \cite[\S 2.1]{Ca2}, one can associate to $W$ a local system $\cV_W$ of $E$-vector spaces over $M_K$. Let $W_0$ be $\co_E$-lattice of $W$, denote by $\cS_{W_0}$ the set (ordered by inclusions) of open compact subgroups of $G(\Q_p)\cong \GL_2(F_{\wp})$ which stabilize $W_0$. For any $K_{p}\in S_{W_0}$, one can associate to $W_0$ \big(resp. to $W_0/\varpi_E^s$ for $s\in \Z_{\geq 1}$\big) a local system $\cV_{W_0}$ \big(resp. $\cV_{W_0/\varpi_E^s}$\big) of $\co_E$-modules \big(resp. of $\co_E/\varpi_E^s$-modules\big) over $M_{K^pK_p}$.  Following Emerton (\cite{Em1}), we put
\begin{eqnarray*}
   H^i_{\et}(K^p,W_0)&:=&\varinjlim_{K_p\in \cS_{W_0}} H^i_{\et} \big( M_{K_pK^p, \overline{\Q}}, \cV_{W_0}\big) \\
&\cong& \varinjlim_{K_p\in \cS_{W_0}} \varprojlim_{s} H^i_{\et} ( M_{K_pK^p, \overline{\Q}}, \cV_{W_0/\varpi_E^s});\\
\widetilde{H}^i_{\et}(K^p,W_0)&:=& \varprojlim_{s}\varinjlim_{K_p\in \cS_{W_0}}H^i_{\et} ( M_{K_pK^p,\overline{\Q}}, \cV_{W_0/\varpi_E^s});\\
H^i_{\et}(K^p,W_0)_E &:=& H^i_{\et}(K^p,W_0) \otimes_{\co_E}E; \\
\widetilde{H}^i_{\et}(K^p,W_0)_E &:=&\widetilde{H}^i_{\et}(K^p,W_0) \otimes_{\co_E}E.
\end{eqnarray*}
All these groups ($\co_E$-modules or $E$-vector spaces) are equipped with a natural topology induced from the discrete topology on the finite group $H^i_{\et} \big( M_{K_pK^p, \overline{\Q}}, \cV_{W_0/\varpi_E^s}\big)$, and equipped with a natural continuous action of $\cH^p\times \Gal(\overline{\Q}/F)$ and of $K_p\in \cS_{W_0}$. Moreover, for any $\fl\in S(K^p)$, the action of $\Gal(\overline{F_{\fl}}/F_{\fl})$ (induced by that of $\Gal(\overline{\Q}/F)$)  is unramified and satisfies the Eichler-Shimura relation:
\begin{equation}\label{equ: clin-llf-}
  \Frob_{\fl}^{-2} -T_{\fl}\Frob_{\fl}^{-1}+\ell^{f_{\fl}}S_{\fl}=0
\end{equation}
where $\Frob_{\fl}$ denotes the arithmetic Frobenius, $\ell$ the prime number lying below $\fl$, $f_{\fl}$ the degree of the maximal unramified extension (of $\Q_{\ell}$) in $F_{\fl}$ over $\Q_{\ell}$ (thus $\ell^{f_{\fl}}= |\co_{\fl}/\varpi_{\fl}|$). Note that $\widetilde{H}^i_{\et}(K^p,W_0)_E$ is an $E$-Banach space with norm defined by the $\co_E$-lattice $\widetilde{H}^i_{\et}(K^p,W_0)$.

Consider the ordered set (by inclusion) $\{W_0\}$ of $\co_E$-lattices of $W$, following \cite[Def.2.2.9]{Em1}, we put
\begin{eqnarray*}
H^i_{\et}(K^p,W)&:=& \varinjlim_{W_0} H^i_{\et}(K^p,W_0)_E, \\
\widetilde{H}^i_{\et}(K^p,W)&:=& \varinjlim_{W_0} \widetilde{H}^i_{\et}(K^p,W_0)_E,
\end{eqnarray*}
where all the transition maps are topological isomorphisms (cf. \cite[Lem.2.2.8]{Em1}). These $E$-vector spaces are moreover equipped with a natural continuous action of $\GL_2(F_{\wp})$.
\begin{theorem}[$\text{cf. \cite[Thm.2.2.11 (i), Thm.2.2.17]{Em1}}$]\label{thm: clin-ecs}(1) The $E$-Banach space $\widetilde{H}^i_{\et}(K^p,W)$ is an admissible Banach representation of $\GL_2(F_{\wp})$. If $W$ is the trivial representation, the representation $\widetilde{H}^i_{\et}(K^p,W)$ is unitary.

(2) One has a natural isomorphism of Banach representations of $\GL_2(F_{\wp})$ invariant under the action of $\cH^p \times \Gal(\overline{F}/F)$:
\begin{equation}\label{equ: clin-wpkw}
  \widetilde{H}^i_{\et}(K^p,W) \xlongrightarrow{\sim} \widetilde{H}^i_{\et}(K^p,E)\otimes_E W.
\end{equation}

(3) One has a natural $\GL_2(F_{\wp})\times \cH^p \times \Gal(\overline{F}/F)$-invariant map
\begin{equation}\label{equ: clin-ehiw}
  H^i_{\et}(K^p,W) \lra \widetilde{H}^i_{\et}(K^p,W).
\end{equation}
%
\end{theorem}
\subsection{Localization at a non-Eisenstein maximal ideal}Let $\rho$ be a $2$-dimensional continuous representation of $\Gal(\overline{F}/F)$ over $E$ such that $\rho$ is unramified at all $\fl\in S(K^p)$. Let $\rho_0$ be a $\Gal(\overline{F}/F)$-invariant lattice of $\rho$, and $\overline{\rho}^{\sss}$ the semisimplification of $\rho_0/\varpi_E$, which is in fact independent of the choice of $\rho_0$. To $\overline{\rho}^{\sss}$, one can associate a maximal ideal of $\cH^p$, denoted by $\fm(\overline{\rho}^{\sss})$, as the kernel of the following morphism
\begin{equation*}
  \cH^p \lra k_E:=\co_E/\varpi_E, \ T_{\fl} \mapsto \tr(\Frob_{\fl}^{-1}),\ S_{\fl} \mapsto \dett(\Frob_{\fl}^{-1})
\end{equation*}
for all $\fl\in S(K^p)$.
\begin{notation}
  For an $\cH^p$-module $M$, denote by $M_{\overline{\rho}^{\sss}}$ the localisation of $M$ at $\fm(\overline{\rho}^{\sss})$.
\end{notation}
Keep the notation in \S \ref{sec: clin-2.1}. As in \cite[\S 5.2, 5.3]{Em4}, one can show that $\widetilde{H}_{\et}^1(K^p,W)_{\overline{\rho}^{\sss}}$ is a direct summand of $\widetilde{H}_{\et}^1(K^p,W)$.
Suppose in the following that  $\rho$ is absolutely irreducible modulo $\varpi_E$ and put $\overline{\rho}:=\overline{\rho}^{\sss}$. 
\begin{proposition}[$\text{\cite[Prop.5.2]{New}}$]\label{prop: clin-trn}
The map (\ref{equ: clin-ehiw}) induces an isomorphism
\begin{equation*}
  H^1_{\et}(K^p,W)_{\overline{\rho}} \xlongrightarrow{\sim} \widetilde{H}^1_{\et}(K^p,W)_{\overline{\rho}, \infty},
\end{equation*}
where $ \widetilde{H}^1_{\et}(K^p,W)_{\overline{\rho}, \infty}$ denotes the smooth vectors (for the action of $\GL_2(F_{\wp})$) in $\widetilde{H}^1_{\et}(K^p,W)_{\overline{\rho}}$.
\end{proposition}
\begin{proposition}[$\text{\cite[Cor.5.8]{New}}$]\label{prop: clin-pfo}
  Let $H$ be a open compact prop-$p$ subgroup of $K_{0,\wp}$, then there exists $r\in \Z_{\geq 1}$ such that
  \begin{equation*}
    \widetilde{H}_{\et}^1(K^p,E)_{\overline{\rho}} \xlongrightarrow{\sim} \cC\Big(H \big/\overline{(Z(\Q)\cap K^{p}H)_{p}}, E\Big)^{\oplus r}
  \end{equation*}
  as representations of $H$, where $\cC\Big(H \big/\overline{(Z(\Q)\cap K^pH)_p}, E\Big)$ denotes the space of continuous functions from $H \big/\overline{(Z(\Q)\cap K^pH)_p}$ to $E$, on which $H$ acts by the right regular action, $\overline{(Z(\Q)\cap K^pH)_p}$ the closure of $(Z(\Q)\cap K^pH)_p$ in $G(\Q_p)$, and $(Z(\Q)\cap K^pH)_p$ the image of $Z(\Q)\cap K^pH$ in $G(\Q_p)$ via the projection $G(\bA^{\infty}) \twoheadrightarrow G(\Q_p)$.
\end{proposition}
Let $\psi$ be a continuous character of $Z_1$ over $E$. We see  $\widetilde{H}^1_{\et}(K^p,E)^{Z_1=\psi}$
is also an admissible Banach representation of $\GL_2(F_{\wp})$  stable under the action $\Gal(\overline{F}/F) \times \cH^p$. Put
\begin{equation}\label{equ: clin-2ls}U_1:=\{g_{\wp}\in K_{1,\wp}\ |\ \dett(g_{\wp})=1\}.\end{equation} Let $H_{\wp}:=Z_1U_1$ which is a open compact subgroup of $K_{1,\wp}$, ($H_{\wp}=K_{1,\wp}$ when $p\neq 2$), we see the center of $H_{\wp}$ is $Z_1$. By Prop.\ref{prop: clin-pfo} applied to $H=H_{\wp}$, one has \Big(note that $\overline{(Z(\Q)\cap K^pH_\wp)_p}$ is a subgroup of $Z_1$\Big)
\begin{corollary}\label{cor: clin-ost}
  Let $\psi$ be a continuous character of $Z_1$ such that $\psi|_{\overline{(Z(\Q)\cap K^{p}H_{\wp})_{p}}}=1$, then one has an isomorphism of $H_{\wp}$-representations
  \begin{equation*}
   \widetilde{H}_{\et}^1(K^{p},E)_{\overline{\rho}}^{Z_1=\psi}\xlongrightarrow{\sim} \cC\big(U_1, E\big)^{\oplus r}
  \end{equation*}
  where $Z_1$ acts on $\cC\big(U_1, E\big)^{\oplus r}$ by the character $\psi$, and $U_1$ by the right regular action.
\end{corollary}

\section{Eigenvarieties}
\subsection{Generalities}\label{sec: clin-3.1}
Consider the admissible locally $\Q_p$-analytic representation
 $ \widetilde{H}^1_{\et}(K^p,E)_{\Q_p-\an}$,
by applying the functor of Jacquet-Emerton (cf. \cite{Em11}), one obtains an essentially admissible locally $\Q_p$-analytic representation $J_B\big(\widetilde{H}^1_{\et}(K^p,E)_{\Q_p-\an}\big)$ of $T(F_{\wp})$
(cf. \cite[\S 6.4]{Em04}). Denote by $\widehat{T}_{\Sigma_{\wp}}$ the rigid space over $E$ parameterizing the locally $\Q_p$-analytic characters of $T(F_{\wp})$.
By definition (of essentially admissible locally $\Q_p$-analytic representations, cf. \emph{loc. cit.}), the action of $T(F_{\wp})$ on $J_B\big(\widetilde{H}^1_{\et}(K^p,E)_{\Q_p-\an}\big)^{\vee}_b$ (where "$b$" signifies the strong topology) can extend to a continuous action of $\co(\widehat{T}_{\Sigma_{\wp}})$ (being a Fr\'echet-Stein algebra) such that $J_B\big(\widetilde{H}^1_{\et}(K^p,E)_{\Q_p-\an}\big)^{\vee}_b$ is a coadmissible $\co(\widehat{T}_{\Sigma_{\wp}})$-module. Thus there exists a coherent sheaf $\cM_0$ on $\widehat{T}_{\Sigma_{\wp}}$ such that \begin{equation*}\cM_0\big(\widehat{T}_{\Sigma_{\wp}}\big)\xlongrightarrow{\sim} J_B\big(\widetilde{H}^1_{\et}(K^p,E)_{\Q_p-\an}\big)^{\vee}_b.\end{equation*}
The action of $\cH^p$ on $J_B\big(\widetilde{H}^1_{\et}(K^p,E)_{\Q_p-\an}\big)$ induces a natural $\co_{\widehat{T}_{\Sigma_{\wp}}}$-linear action of $\cH^p$ on $\cM_0$. Following Emerton, one can construct an \emph{eigenvariety} $\cV(K^p)$ from the triple
%
%
$\Big\{\cM_0, \widehat{T}_{\Sigma_{\wp}}, \cH^p\Big\}$:
\begin{theorem}[$\text{cf. \cite[\S 2.3]{Em1}}$]\label{thm: clin-cjw}
  There exists a rigid analytic space $\cV(K^p)$ over $E$ together with a finite morphism of rigid spaces
  \begin{equation*}
    i: \cV(K^p) \lra \widehat{T}_{\Sigma_{\wp}}
  \end{equation*}
  and a morphism of $E$-algebras with dense image (see Rem.\ref{rem: linv-yden} below)
  \begin{equation}\label{equ: clin-oct}
    \cH^p\otimes_{\co_E} \co\big(\widehat{T}_{\Sigma_{\wp}}\big) \lra \co\big(\cV(K^p)\big)
  \end{equation}
  such that
  \begin{enumerate}
    \item a closed point $z$ of $\cV(K^p)$ is uniquely determined by its image $\chi$ in $\widehat{T}_{\Sigma_{\wp}}(\overline{E})$ and the induced morphism $\lambda: \cH^p\ra \overline{E}$, called a system of eigenvalues of $\cH^p$, so $z$ would be denoted by $(\chi,\lambda)$;
    \item for a finite extension $L$ of $E$, a closed point $(\chi,\lambda)\in \cV(K^p)(L)$ if and only if the corresponding eigenspace
    \begin{equation*}
      J_B\big(\widetilde{H}^1_{\et}(K^{p},E)_{\Q_p-\an}\otimes_E L\big)^{T(F_{\wp})=\chi, \cH^p=\lambda}
    \end{equation*}
    is non zero;
    \item there exists a coherent sheaf over $\cV(K^p)$, denoted by $\cM$, such that $i_* \cM \cong \cM_0$ and that for an $L$-point $z=(\chi, \lambda)$, the special fiber $\cM\big|_z$ is naturally dual to the (finite dimensional) $L$-vector space
        \begin{equation*}
          J_B\big(\widetilde{H}^1_{\et}(K^p,E)_{\Q_p-\an}\otimes_E L\big)^{T(F_{\wp})=\chi, \cH^p=\lambda}.
        \end{equation*}
  \end{enumerate}
\end{theorem}
\begin{remark}\label{rem: linv-yden}
  Indeed, by construction as in \cite[\S 2.3]{Em1}, for any affinoid admissible open $U=\Spm A$ in $\widehat{T}_{\Sigma_{\wp}}$, one has $i^{-1}(U)\cong \Spm B$ where $B$ is the affinoid algebra over $A$ generated by the image of $\cH^p \ra \End_A(\cM_0(U))$, from which we see (\ref{equ: clin-oct}) has a dense image.
\end{remark}
Denote by $\cV(K^p)_{\red}$ the reduced closed rigid subspace of $\cV(K^p)$.
\begin{lemma}
  The image of  $\cH^p$  in $\co\big(\cV(K^p)_{\red}\big)$ via (\ref{equ: clin-oct}) lies in  \begin{equation*}\co\big(\cV(K^p)_{\red}\big)^{0}\\ :=\Big\{f\in \co\big(\cV(K^p)_{\red}\big)\ \Big|\ \|f(x)\|\leq 1,\ \forall x\in \cV(K^p)(\overline{E})\Big\}.\end{equation*}
\end{lemma}
\begin{proof}
  It's sufficient to prove for any closed point $(\chi,\lambda)$ of $\cV(K^p)$, the morphism $\lambda: \cH^p \ra \overline{E}$ factors through $\overline{\co_E}$. But this is clear since $\widetilde{H}^1_{\et}(K^p,E)$ has an $\cH^p$-invariant $\co_E$-lattice (see \S \ref{sec: clin-2.1}).
\end{proof}
Since the rigid space $\widehat{T}_{\Sigma_{\wp}}$ is nested, by \cite[Lem.7.2.11]{BCh}, one has
\begin{proposition}\label{prop: clin-ide}
  The rigid space $\cV(K^p)$ is nested, and $\co\big(\cV(K^p)_{\red}\big)^{0}$ is a compact subset of $\co\big(\cV(K^p)_{\red}\big)$ (where we refer to \emph{loc. cit.} for the topology).
\end{proposition}
It would be conviennent to fix a central character (in the quaternion Shimura curve case), let $w\in \Z$, consider the (essentially admissible) locally $\Q_p$-analytic representation
\begin{equation}\label{equ: linv-tepn}
  J_B\Big(\widetilde{H}^1_{\et}(K^p,E)_{\Q_p-\an}^{Z_1=\sN^{-w}}\Big)\cong J_B\big(\widetilde{H}^1_{\et}(K^p,E)_{\Q_p-\an}\big)^{Z_1=\sN^{-w}}.
\end{equation}
One can construct an eigenvariety, denoted by $\cV(K^p,w)$, in the same way as in Thm.\ref{thm: clin-cjw} which satisfies all the properties in Thm.\ref{thm: clin-cjw} with $\widetilde{H}^1_{\et}(K^p,E)$ replaced by $\widetilde{H}^1_{\et}(K^p,E)^{Z_1=\sN^{-w}}$. Denote by $\widehat{T}_{\Sigma_{\wp}}(w)$ the closed rigid subspace of $\widehat{T}_{\Sigma_{\wp}}$ such that \begin{equation}\label{equ: clin-tapwe}\widehat{T}_{\Sigma_{\wp}}(w)(\overline{E})=\big\{\chi\in \widehat{T}_{\Sigma_{\wp}}\ \big|\ \chi|_{Z_1}=\sN^{-w}\big\},\end{equation} moreover, if we denote by $\cM_0(w)$ the coherent sheaf over $\widehat{T}_{\Sigma_{\wp}}(w)$ associated to  $J_B\Big(\widetilde{H}^1_{\et}(K^p,E)_{\Q_p-\an}^{Z_1=\sN^{-w}}\Big)$, by (\ref {equ: linv-tepn}), one has $\cM_0(w)\cong \cM_0 \otimes_{\co(\widehat{T}_{\Sigma_{\wp}})} \co\big(\widehat{T}_{\Sigma_{\wp}}(w)\big)$ and thus (by the construction in \cite[\S 2.3]{Em1}) $\cV(K^p,w)\cong \cV(K^p) \times_{\widehat{T}_{\Sigma_{\wp}}} \widehat{T}_{\Sigma_{\wp}}(w)$.

\subsection{Classicality and companion points}Let $T(F_{\wp})^+:=\bigg\{\begin{pmatrix}
  a & 0 \\ 0 & d
\end{pmatrix}\in T(F_{\wp})\ \bigg|\ \us_{\wp}(a) \geq \us_{\wp}(d)\bigg\}$,
one can equip $V^{N_0}$ with a continuous action of $T(F_{\wp})^+$ by
\begin{equation*}
  \pi_t(v):=\big|N_0/tN_0t^{-1}\big|^{-1} \sum_{n\in N_0/tN_0t^{-1}} (nt)(v).
\end{equation*}
One has a natural $T(F_{\wp})^+$-invariant injection
\begin{equation}\label{equ: clin-pvJ0n}
  J_B(V) \hooklongrightarrow V^{N_0}
\end{equation}
which induces a bijection (cf. \cite[Prop.3.4.9]{Em11})
\begin{equation}\label{equ clin-wmv0i}
  J_B(V)^{T(F_{\wp})=\chi} \xlongrightarrow{\sim} V^{N_0, T(F_{\wp})^+=\chi}
\end{equation}
for any continuous (locally $\Q_p$-analytic) characters $\chi$ of $T(F_{\wp})$. In fact, by the same argument in \cite[Prop.3.2.12]{Em11}, one can show (\ref{equ: clin-pvJ0n}) induces a bijection between generalized eigenspaces \big(note that the action of $T(F_{\wp})^+$ on $V^{N_0}[T(F_{\wp})^+=\chi]$ extends naturally to an action of $T(F_{\wp})$\big)
\begin{equation}\label{equ: clin-0nvi}
    J_B(V)[T(F_{\wp})=\chi] \xlongrightarrow{\sim} V^{N_0}[T(F_{\wp})^+=\chi].
\end{equation}
\begin{definition}\label{def: clin-ltn}
For an $L$-point $z=(\chi,\lambda)$ of $\cV(K^p)$, $S \subseteq \Sigma_{\wp}$, $z$ is called $S$-classical (resp. quasi-$S$-classical) if there exists a non-zero vector
\[v\in \big(\widetilde{H}^1_{\et}(K^p, E)_{\Q_p-\an}\otimes_E L\big)^{N_0, T(F_{\wp})^+=\chi, \cH^p=\lambda}\] such that $v$ is $S$-classical (resp. quasi-$S$-classical).
We call $z$ classical (quasi-classical) if $z$ is $\Sigma_{\wp}$-classical (resp. quasi-$\Sigma_{\wp}$-classical).
\end{definition}
\begin{definition}Let $z=(\chi_1\otimes \chi_2,\lambda)$ be a closed point in $\cV(K^p)$, for $S\subseteq C(\chi)$ (cf. (\ref{equ: clin-cchi})), put
\begin{equation}\label{equ: clin-s2cd}
  \chi_S^c=\chi_{1,S}^c\otimes \chi_{2,S}^c:=\chi_1 \prod_{\sigma \in S} \sigma^{k_{\chi_2,\sigma}-k_{\chi_1,\sigma}-1} \otimes \chi_2 \prod_{\sigma\in S} \sigma^{k_{\chi_1,\sigma}-k_{\chi_2,\sigma}+1};
\end{equation}
we say that $z$ admits an $S$-companion point if $z_S^c:=(\chi_S^c, \lambda)$ is also a closed point in $\cV(K^p)$. If so, we say the companion point $z_S^c$ is effective if it is moreover quasi-$C(\chi)\setminus S$-classical \big(note that $C(\chi_S^c)=C(\chi)\setminus S$\big).
\end{definition}As in \cite[Lem.6.2.24]{Ding}, one has
\begin{proposition}\label{prop: clin-vka}
Let $z=(\chi, \lambda)$ be an $L$-point in $\cV(K^p)$, $\sigma\in C(\chi)$, suppose there exists a non quasi-$\sigma$-classical vector (see \S \ref{sec: clin-1} for $d\chi$)
\begin{equation*}
  v\in \big(\widetilde{H}^1_{\et}(K^p, E)_{\Q_p-\an} \otimes_E L\big)^{N_0, \ft_{\Sigma_{\wp}}=d\chi}[T(F_{\wp})^+=\chi, \cH^p=\lambda],
\end{equation*}
then $z$ admits a $\sigma$-companion point. Moreover, there exists $S\subseteq C(\chi)$ containing $\sigma$ such that $z$ admits an effective $S$-companion point.
\end{proposition}
\begin{proof}
  We sketch the proof. Let $k_\sigma:=k_{\chi_1,\sigma}-k_{\chi,\sigma}+2\in \Z_{\geq 2}$ (since $\sigma\in C(\chi)$), if $v$ is not quasi-$\sigma$-classical, we deduce that $v_{\sigma}^c=X_{-,\sigma}^{k_{\sigma}-1} \cdot v\neq 0$ \bigg(where $X_{-,\sigma}:=\begin{pmatrix}
    0 & 0 \\ 1 & 0
  \end{pmatrix}\in \ug_{\sigma}$\bigg), moreover, as in \cite[Lem.6.3.15]{Ding}, one can prove $v_{\sigma}^c$ is a generalised $(\chi_\sigma^c,\lambda)$-eigenvector for $T(F_{\wp})\times \cH^p$. From which we deduce $z$ admits a $\sigma$-companion point.  If $v_{\sigma}^c$ is not quasi-$\sigma'$-classical for some $\sigma'\in C(\chi)\setminus \{\sigma\}=C(\chi_{\sigma}^c)$, one can repeat this argument to find companion points of $z_{\sigma}^c$ until one gets $S\subseteq C(\chi)$ and an effective $S$-companion point of $z$.
\end{proof}
\begin{remark}
  One can also deduce this proposition from the adjunction formula in \cite[Thm.4.3]{Br13}.
\end{remark}As in \cite[Prop.6.2.27]{Ding}, one has
\begin{theorem}[Classicality]\label{thm: clin-etn}
Let $z=(\chi=\chi_1\otimes \chi_2, \lambda)$ be an $L$-point in $\cV(K^p)$.
For $\sigma\in C(\chi)$, put $k_{\sigma}:=k_{\chi_1,\sigma}-k_{\chi_2,\sigma}+2\in \Z_{\geq 2}$. Let $S\subseteq C(\chi)$, if
  \begin{equation*}
    \us_{\wp}(q\chi_1(\varpi)) <\inf_{\sigma\in S}\{k_{\sigma}-1\},
  \end{equation*}
then any vector in
\begin{equation*}
\big(\widetilde{H}^1_{\et}(K^p,E)_{\Q_p-\an}\otimes_E L\big)^{N_0, \ft_{\Sigma_{\wp}}=d\chi}[T(F_{\wp})^+=\chi, \cH^p=\lambda]
\end{equation*}
in quasi-$S$-classical, in particular, the point $z$ is quasi-$S$-classical.
\end{theorem}
\begin{proof}We sketch the proof. For $\sigma \in S$, if there exists a non quasi-$\sigma$-classical vector
 \begin{equation*}
v\in \big(\widetilde{H}^1_{\et}(K^p,E)_{\Q_p-\an}\otimes_E L\big)^{N_0, \ft_{\Sigma_{\wp}}=d\chi}[T(F_{\wp})^+=\chi, \cH^p=\lambda],
 \end{equation*}
 by Prop.\ref{prop: clin-vka}, $z$ admits an effective $S'$-companion point $z_{S'}^c$ with $S'\subseteq C(\chi)$ containing $\sigma$. Using \cite[Prop.6.2.23]{Ding}, this point would induce a continuous injection from a locally $\Q_p$-analytic parabolic induction twisted with certain algebraic representation \big(as in \emph{loc. cit.} by replacing $J$, $S$, $C_{\overline{B}}(\chi)$ by $\Sigma_{\wp}$, $S'$, $C(\chi)$ respectively\big) into $\widetilde{H}^1_{\et}(K^p,E)\otimes_E L$. Since $\widetilde{H}^1_{\et}(K^p,E)$ is unitary, one can apply \cite[Prop.5.1]{Br}, and get (as in \cite[Cor.6.2.24]{Ding}) $\us_{\wp}(q\chi_1(\varpi))\geq \sum_{\sigma\in S'} (k_{\sigma}-1)$, a contradiction.
\end{proof}
\begin{corollary}\label{cor: clin-caw-}
 Let $w\in \Z$, $z=(\chi=\chi_1\otimes \chi_2, \lambda)$ be an $L$-point in $\cV(K^p,w)$ with $C(\chi)=\Sigma_{\wp}$, $k_{\sigma}:=k_{\chi_1,\sigma}-k_{\chi_2,\sigma}+2\in 2 \Z_{\geq 1}$ such that $k_{\sigma}\equiv w \pmod{2}$ for all $\sigma \in \Sigma_{\wp}$. There exist thus smooth characters $\psi_1$, $\psi_2$ such that (note that $k_{\chi_1,\sigma}+k_{\chi_2,\sigma}=-w$)
 \begin{equation*}
   \chi_1 \otimes \chi_2=\prod_{\sigma\in \Sigma_{\wp}} \sigma^{-\frac{w-k_{\sigma}+2}{2}} \psi_1 \otimes \prod_{\sigma \in \Sigma_{\wp}} \sigma^{-\frac{w+k_{\sigma}-2}{2}} \psi_2
 \end{equation*}
 Let $S\subseteq \Sigma_{\wp}$, if
 \begin{equation*}
   \us_{\wp}(q \psi_1(\varpi))<\sum_{\sigma\in \Sigma_{\wp}} \frac{w-k_{\sigma}+2}{2} + \inf_{\sigma\in S} \{k_{\sigma}-1\},
 \end{equation*}
 then the point $z$ is $S$-classical.
\end{corollary}
\begin{remark}
  We invite the reader to compare this corollary with conjectures of Breuil in \cite{Br00} and results of Tian-Xiao in \cite{TX}.
\end{remark}
\subsection{Localization at a non-Eisenstein maximal ideal}\label{sec: clin-3.3}
Let $\rho$ be a $2$-dimensional continuous representation of $\Gal(\overline{F}/F)$ over $E$, suppose that $\rho$ is absolutely irreducible modulo $\varpi_E$ and there exists an irreducible algebraic representation $W$ of $G^c$ such that $H^1_{\et}(K^p,W)_{\overline{\rho}}$ ($\rho$ is thus called \emph{modular}). It's known that there exist $w\in \Z$, $k_{\sigma}\in Z_{\geq 2}$, $k_{\sigma}\equiv w \pmod{2}$ for all $\sigma\in \Sigma_{\wp}$ such that $W\cong W(\ul{k}_{\Sigma_{\wp}},w)$. We fix this $w$ in the following. Consider the essentially admissible locally $\Q_p$-analytic representation $J_B\big(\widetilde{H}^1_{\et}(K^p,E)_{\overline{\rho}}^{Z_1=\sN^{-w}}\big)$, whose strong dual gives rise to a coherent sheaf $\cM_0(K^p,w)_{\overline{\rho}}$ over $\widehat{T}_{\Sigma_{\wp}}$. As in Thm.\ref{thm: clin-cjw}, one can obtain an eigenvariety $\cV(K^p,w)_{\overline{\rho}}$ together with a coherent sheaf $\cM(K^p,w)_{\overline{\rho}}$ over $\cV(K^p,w)_{\overline{\rho}}$, which satisfies the properties in Thm.\ref{thm: clin-cjw} with $\widetilde{H}^1_{\et}(K^p,E)$ replaced by $\widetilde{H}^1_{\et}(K^p,E)^{Z_1=\sN^{-w}}_{\overline{\rho}}$. Since $\widetilde{H}^1_{\et}(K^p,E)^{Z_1=\sN^{-w}}_{\overline{\rho}}$ is a direct summand of $\widetilde{H}^1_{\et}(K^p,E)^{Z_1=\sN^{-w}}$, $\cV(K^p,w)_{\overline{\rho}}$ is a closed rigid subspace of $\cV(K^p,w)$ (cf. \cite[Lem.6.2.6]{Ding}). By Prop.\ref{prop: clin-trn}, one can describe the classical vectors of $\widetilde{H}^1_{\et}(K^p,E)^{Z_1=\sN^{-w}}_{\overline{\rho}}$ as follows:
\begin{corollary}\label{prop: clin-ernr}
  With the notation in Cor.\ref{cor: clin-caw-}, suppose moreover $z$ in $\cV(K^p,w)_{\overline{\rho}}$, let $v$ be a vector in
  \begin{equation}\label{equ: clin-weht} \big(\widetilde{H}^1_{\et}(K^p,E)_{\Q_p-\an, \overline{\rho}}\otimes_E L\big)^{N_0, Z_1=\sN^{-w},\ft_{\Sigma_{\wp}}=d\chi}[T(F_{\wp})^+=\chi, \cH^p=\lambda],\end{equation}
  if $v$ is classical, then $v$ lies in (see Rem.\ref{rem: clin-ging} below)
  \begin{equation}\label{equ: clin-kpwg}
\big(H^1_{\et}\big(K^p, W(\ul{k}_{\Sigma_{\wp}},w)\big)_{\overline{\rho}}\otimes_E L\big)^{N_0, Z_1=\psi_1\psi_2}[T(F_{\wp})^+=\psi_1 \otimes \psi_2, \cH^p=\lambda] \otimes_E \chi(\ul{k}_{\Sigma_{\wp}},w),
  \end{equation}
  with $\chi(\ul{k}_{\Sigma_{\wp}},w):=\prod_{\sigma\in \Sigma_{\wp}} \sigma^{-\frac{w-k_{\sigma}+2}{2}}\otimes \prod_{\sigma\in \Sigma_{\wp}} \sigma^{-\frac{w+k_{\sigma}-2}{2}}$ \big(being a character of $T(F_{\wp})$\big).
\end{corollary}
\begin{remark}\label{rem: clin-ging}
  Note that $T(F_{\wp})$ acts on $\big(W(\ul{k}_{\Sigma_{\wp}},w)^{\vee}\big)^{N_0}$ via $\chi(\ul{k}_{\Sigma_{\wp}},w)$, the embedding of the vector space (\ref{equ: clin-kpwg}) into (\ref{equ: clin-weht}) is obtained by taking $N_0$-invariant  vectors of the following $\GL_2(F_{\wp}) \times \cH^p \times \Gal(\overline{F}/F)$-invariant injection (cf. Prop.\ref{prop: clin-trn})
  \begin{multline*}
    H^1_{\et}\big(K^p,W(\ul{k}_{\Sigma_{\wp}},w)\big)_{\overline{\rho}} \otimes_E W(\ul{k}_{\Sigma_{\wp}},w)^{\vee} \otimes_E L\\ \xlongrightarrow{\sim} \big(\widetilde{H}^1_{\et}(K^p,E)_{\Q_p-\an,\overline{\rho}} \otimes_E W(\ul{k}_{\Sigma_{\wp}},w)\big)_{\infty}\otimes_E W(\ul{k}_{\Sigma_{\wp}},w)^{\vee} \otimes_E L \\ \hooklongrightarrow \widetilde{H}^1_{\et}(K^p,E)_{\Q_p-\an, \overline{\rho}} \otimes_E L.
  \end{multline*}
\end{remark}


We study in details the structure of $\cV(K^p,w)_{\overline{\rho}}$. Let \begin{equation}\label{equ: clin-t'b}T':=Z_1'\times \begin{pmatrix}
  \varpi & 0 \\ 0 & 1
\end{pmatrix}^{\Z} \times  \begin{pmatrix}
  \varpi & 0 \\ 0 & \varpi
\end{pmatrix}^{\Z}   \\ \times \bigg\{ \begin{pmatrix}
  z_{1} & 0 \\ 0 & z_{2}
\end{pmatrix}\ \bigg|\ \ z_{i}\in \co_{\wp}^{\times},  \ z_{i}^{q-1}=1\bigg\}.\end{equation}
One has thus a finite morphism of rigid spaces (which is moreover an isomorphism when $p\neq 2$)
\begin{equation}\label{equ: clin-st1}
  \widehat{T}_{\Sigma_{\wp}} \xlongrightarrow{(\pr_1,\pr_2)} \widehat{(T')}_{\Sigma_{\wp}} \times \widehat{(Z_1)}_{\Sigma_{\wp}}, \  \chi\mapsto (\chi|_{T'},\chi|_{Z_1}),
\end{equation}
where $\widehat{(T')}_{\Sigma_{\wp}}$ and $\widehat{(Z_1)}_{\Sigma_{\wp}}$ denote the rigid spaces parameterizing locally $\Q_p$-analytic characters of $T'$ and $Z_1$ respectively.
Note that $\cM_1:=\pr_{1,*}\cM_0(K^p,w)_{\overline{\rho}}$ is in fact a coherent sheaf over $\widehat{(T')}_{\Sigma_{\wp}}$: the support of  $\cM_0(K^p,w)_{\overline{\rho}}$ is contained in $\widehat{T}_{\Sigma_{\wp}}(w)$ \big(as a closed subspace of $\widehat{T}_{\Sigma_{\wp}}$\big), which  is finite over $\widehat{(T')}_{\Sigma_{\wp}}$.

Put $\Pi:=\begin{pmatrix}\varpi & 0 \\ 0 & 1\end{pmatrix}$,  let $R_{\wp}$ be a (finite) set of representatives of $T(F_{\wp})/T'Z_1$ in $T(F_{\wp})$ (note $T(F_{\wp})=T'Z_1$ when $p\neq 2$), let $\cH$ be the $\co_E$-algebra generated by $\cH^p$ and $\begin{pmatrix}
  \varpi & 0 \\ 0 & \varpi
\end{pmatrix}$, $\bigg\{\begin{pmatrix}
  z_{1} & 0 \\ 0 & z_{2}
\end{pmatrix}\ \bigg|\ z_{i}\in \co_{\wp}^{\times},\ z_{i}^{q-1}=1\bigg\}$,  and the elements in $R_{\wp}$. Denote by $\cW_{1,\Sigma_{\wp}}$ the rigid space ove $E$ which parameterizes locally $\Q_p$-analytic characters of $1+2\varpi \co_{\wp}\cong Z_1'$, by the decomposition of groups (\ref{equ: clin-t'b}), one gets a natural projection $\pr: \widehat{(T')}_{\Sigma_{\wp}}\ra \cW_{1,{\Sigma_{\wp}}}\times \bG_m, \  \chi \mapsto (\chi|_{Z_1'}, \chi(\Pi))$. By Cor.\ref{cor: clin-ost} and the argument in the proof of \cite[Prop.4.2.36]{Em11} (e.g. see \cite[(4.2.43)]{Em11}), we see $\cM_1\big(\widehat{(T')}_{\Sigma_{\wp}}\big)$ is a coadmissible $\co(\cW_{1,\Sigma_{\wp}})\{\{X,X^{-1}\}\}$-module with $X$ acting on $\cM_1\big(\widehat{(T')}_{\Sigma_{\wp}}\big)$ by the operator $\Pi$. Denote by $\cM_2$ the associated coherent sheaf over $\cW_{1,\Sigma_{\wp}} \times \bG_m$, which is equipped with an $\co_{\cW_{1,\Sigma_{\wp}}\times \bG_m}$-linear action of $\cH$. One can thus construct $\cV(K^p,w)_{\overline{\rho}}$ from the triple $\big\{\cM_2, \widehat{(T')}_{\Sigma_{\wp}}, \cH\big\}$ as in \cite[\S 2.3]{Em1}.


Let  $\{\Spm A_i\}_{i \in I}$ be an admissible covering of $\cW_{1,\Sigma_{\wp}}$  by increasing affinoid opens,  by Cor.\ref{cor: clin-ost}, \cite[(4.2.43)]{Em11} and the results in \cite[\S 5.A]{Ding}, for any $i\in I$, there exists a Fredholm series $F_i(z)\in 1+zA_i\{\{z\}\}$ (which is hence a global section over $\Spm A_i \times \bG_m$) such that the coherent sheaf $\cM_2|_{\Spm A_i \times \bG_m}$ is supported at $\cZ_i$ where $\cZ_i$ is the closed rigid subspace of $\Spm A_i \times \bG_m$ defined by $F_i(z)$. Moreover, it's known that (cf. \cite[\S 4]{Bu}) $\cZ_i$ admits an admissible covering $\{U_{ij}\cong \Spm A_i[z]/P_j(z)\}$ such that
\begin{itemize}
  \item $P_j(z)\in 1+zA_i[z]$ is a polynomial of degree $d_j$ with leading coefficient being a unit,
  \item there exists $Q_j(z)\in 1+zA_i\{\{z\}\}$ such that $F_i(z)=P_j(z)Q_j(z)$ and that $(P_j(z),Q_j(z))=1$.
\end{itemize}
As in the proof of \cite[Prop.5.A.6]{Ding}, one can show $\cM_2(U_{ij})$ is a finite locally free $A_i$-module of rank $d_j$, equipped with an $A_i$-linear action of $\cH$ such that the characteristic polynomial of $\Pi$ is given by $P_j(z)$, and that $Q_j(\Pi)$ acts on $\cM_2(U_{ij})$ via an invertible operator. Denote by $\cH_{ij}$ the $A_i[z]/P_j(z)$-algebra generated by the image of the natural map
\begin{equation*}
  \cH \lra \End_{A_i[z]/P_j(z)}(\cM_2(U_{ij})),
\end{equation*}which is also the $A_i$-algebra generated by the image of the map $\cH \ra \End_{A_i}(\cM_2(U_{\ij}))$ (since $\Pi\in \cH$). The restriction $\cV(K^p,w)_{\overline{\rho}} \big|_{U_{ij}}$ is thus isomorphic to $\Spm \cH_{ij}$. In particular, we see that the construction of $\cV(K^p,w)_{\overline{\rho}}$ coincides with the construction of eigenvarieties by Coleman-Mazur (formalized by Buzzard in \cite{Bu}). Since $\cW_{1,\Sigma_{\wp}}$ is equidimensional of dimension $d$, by \cite[Prop.6.4.2]{Che}, we have
\begin{proposition}
  The rigid analytic space $\cV(K^p,w)_{\overline{\rho}}$ is equidimensional of dimension $d$.
\end{proposition}
\begin{definition}For a character $\chi$ of $T(F_{\wp})$, we say that $\chi$ is spherically algebraic if $\chi$ is the twist of an algebraic character by an unramified character. We call a closed point $z=(\chi,\lambda)$ of  $\cV(K^p,w)$ semi-stable classical if $z$ is classical and $\chi$ is spherically algebraic.
\end{definition}
Denote by $C(w)$ the set of semi-stable classical points in $\cV(K^p,w)_{\overline{\rho}}$. By the same argument as in the proof of \cite[Prop.6.2.7, Prop.6.4.6]{Che}, the following proposition follows from Cor.\ref{cor: clin-caw-}.
\begin{proposition}\label{thm: clin-adtf}(1) Let $z=(\chi,\lambda)$ be a closed point of $\cV(K^p,w)_{\overline{\rho}}$, suppose moreover $\chi$ spherically algebraic, then the set $C(w)$ accumulates over the point $z$, i.e. for any admissible open $U$ containing $z$, there exists an admissible open $V\subseteq U$, $z\in V(\overline{E})$ such that $C(w)\cap V(\overline{E})$ is Zariski-dense in $V$.

(2) The set $C(w)$ is Zariski-dense in $\cV(K^p,w)_{\overline{\rho}}$.
\end{proposition}

\subsection{Families of Galois representations}
\subsubsection{Families of Galois representations on eigenvarieties}
Keep the notation in \S \ref{sec: clin-3.3}. For  $\fl\in S(K^p)$, denote by $\fa_{\fl}\in \co \big(\cV(K^p,w)_{\overline{\rho},\red}\big)$ \Big(resp. $\ub_{\fl}\in \co \big(\cV(K^p,w)_{\overline{\rho},\red}\big)$\Big) the image of $T_{\fl}\in \cH^p$ \big(resp. $S_{\fl}\in \cH^p$\big) via the natural morphism $\cH^p \ra  \co \big(\cV(K^p,w)_{\overline{\rho},\red}\big)$. Denote by $\cS$ the complement of $S(K^p)$ in the set of finite places of $F$, which is hence a finite set. Denote by $F^{\cS}$ the maximal algebraic extension of $F$ which is unramified outside $\cS$.

For any $z\in C(w)$, by \cite{Ca2} (and Cor.\ref{prop: clin-ernr}), there exists a $2$-dimensional continuous representation $\rho_z$ of $\Gal(\overline{F}/F)$ over $k(z)$, the residue field at $z$, which is unramified outside $S$ and hence a representation of $\Gal(F^{\cS}/F)$, such that (see also (\ref{equ: clin-llf-}))
\begin{equation*}
  \Frob_{\fl}^{-2} -\fa_{\fl,z}\Frob_{\fl}^{-1}+\ell^{f_{\fl}}\ub_{\fl,z}=0
\end{equation*}where $\fa_{\fl,z}, \ub_{\fl,z}\in k(z)$ denote the respective evaluation of $\fa_{\fl}$ and $\ub_{\fl}$ at $z$. In particular, one has $\tr(\Frob_{\fl}^{-1})=\fa_{\fl,z}$. Denote by $\cT_z: \Gal(F^{\cS}/F) \ra k(z)$, $g\mapsto \tr(\rho_z(g))$, which is thus a $2$-dimensional continuous pseudo-character of $\Gal(F^{\cS}/F)$ over $k(z)$. By \cite[Prop.7.1.1]{Che} and Prop.\ref{prop: clin-ide}, one has
\begin{proposition}
  There exists a unique $2$-dimensional continuous pseudo-character $\cT: \Gal(F^{\cS}/F) \ra \co \big(\cV(K^p,w)_{\overline{\rho},\red}\big)$ such that the evaluation of $\cT$ at  $z\in C(w)$ equals to $\cT_z$.
\end{proposition}
Let $z$ be a closed point of $\cV(K^p,w)_{\overline{\rho}}$, denote by $\cT_z:=\cT|_z$, which is a $2$-dimensional continuous pseudo-character of $\Gal(F^{\cS}/F)$ over $k(z)$. By \cite[Thm.1(2)]{Ta}, there exists a unique $2$-dimensional continuous semi-simple representation $\rho_z$ of $\Gal(F^{\cS}/F)$ such that $\tr(\rho_z)=\cT_z$. By Eichler-Shimura relations, one has $\overline{\rho_z}\cong \overline{\rho}$, in particular, $\rho_z$ is absolutely irreducible. By \cite[Lem.5.5]{Bergd}, one has
\begin{proposition}\label{prop: clin-kzc}
  For any closed point $z$ of $\cV(K^p,w)_{\overline{\rho}}$, there exist an admissible open affinoid $U$ containing $z$ in $\cV(K^p,w)_{\overline{\rho},\red}$ and a continuous representation
  \begin{equation*}
    \rho_U: \Gal(F^{\cS}/F) \lra \GL_2(\co_U)
  \end{equation*}
  such that $\rho_U|_{z}\cong \rho_{z'}$ for any $z'\in U(\overline{E})$.
\end{proposition}
In general, by \cite[Lem.7.8.11]{BCh}, one has
\begin{proposition}\label{prop: clin-wpw}
Let $U$ be an open affinoid of $\cV(K^p,w)_{\overline{\rho},\red}$, there exist a rigid space $\widetilde{U}$ over $U$, and an $\co_{\widetilde{U}}$-module $\cM$ locally free of rank $2$ equipped with a continuous $\co_{\widetilde{U}}$-linear action of $\Gal(F^{\cS}/F)$ such that
\begin{enumerate}
  \item the morphism $g: \widetilde{U}\ra U$ factors through a rigid space $U'$ such that $\widetilde{U}$ is a blow-up over $U'$ of $U'\setminus U''$ with $U''$ an Zariski-open Zariski-dense subspace of $U'$ and that $U'$ is finite, dominant over $U$;
  \item for any $z\in \widetilde{U}(\overline{E})$, the $2$-dimensional representation $\cM|_z$ of $\Gal(F^{\cS}/F)$ is isomorphic to $\rho_{g(z)}$.
\end{enumerate}
\end{proposition}
\begin{remark}\label{rem: clin-dsu}
  Let $Z$ be a Zariski-dense subset of closed points in $U$, then $g^{-1}(U)$ is Zariski-dense in $\widetilde{U}$: denote by $g': U'\ra U$ the morphism as in (1), by \cite[Lem.6.2.8]{Che},  we see $(g')^{-1}(Z)$ is Zariski-dense in $U$, so $(g')^{-1}(Z)\cap U''(\overline{E})$ is Zariski-dense in $U''$ and hence Zariski-dense in $\widetilde{U}$.
\end{remark}
\subsubsection{Trianguline representations} Consider the restriction $\rho_{z,\wp}:=\rho_z|_{\Gal(\overline{\Q_p}/F_{\wp})}$. Denote by $F_{\wp,\infty}:=\cup_{n} F_{\wp}(\zeta_{p^n})$ where $\zeta_{p^n}$ is a root of unity primitive of order $p^n$. Set $\Gamma:=\Gal(F_{\wp,\infty}/F_{\wp})$, $H_{\wp}:=\Gal(\overline{\Q_p}/F_{\wp,\infty})$. One has a ring $B_{\rig}^{\dagger}$ (cf. \cite[\S 3.4]{Berger}) which is equipped with an action of $\varphi$ and $\Gal(\overline{\Q_p}/\Q_p)$ such that $B_{\rig,F_{\wp}}^{\dagger}:=(B_{\rig}^{\dagger})^{H_{F_{\wp}}}$ is naturally isomorphic to the Robba ring with coefficients in $F_{\wp}'$ where $F_{\wp}'$ denotes the maximal unramified extension of $\Q_p$ in $F_{\wp,\infty}$ (which is finite over $F_{\wp,0}$). For an $n$-dimensional continuous representation of $\Gal(\overline{\Q_p}/F_{\wp})$ over $E$, $D_{\rig}(V):=(B_{\rig}^{\dagger} \otimes_{\Q_p} V)^{H_{\F_{\wp}}}$ is an \'etale $(\varphi,\Gamma)$-module of rank $n$ over $\cR_E:= B_{\rig,F_{\wp}}^{\dagger} \otimes_{\Q_p} E$ (i.e. an \'etale $(\varphi,\Gamma)$-module over $B_{\rig,F_{\wp}}^{\dagger}$ equipped with an action of $E$ which commutes with that of $\varphi$ and $\Gamma$) (cf. \cite[Prop.3.4]{Berger}). Let $\delta: F_{\wp}^{\times} \ra E^{\times}$ be a continuous character, following \cite[\S 1.4]{Na}, one can associate to $\delta$ a $(\varphi,\Gamma)$-module, denoted by $\cR_E(\delta)$, free of rank $1$ over $\cR_E$. The converse is also true, i.e. for any $(\varphi,\Gamma)$-module $D$ free of rank $1$ over $\cR_E$, there exists a continuous character $\delta: F_{\wp}^{\times} \ra E^{\times}$ such that $D\cong \cR_E(\delta)$.
\begin{definition}[$\text{cf. \cite[Def.4.1]{Colm2}, \cite[Def.1.15]{Na}}$]
  Let $\rho$ be a $2$-dimensional continuous representation of $\Gal(\overline{\Q_p}/F_{\wp})$ over $E$, $\rho$ is called trianguline if there exist continuous characters $\delta_1$, $\delta_2$ of $F_{\wp}^{\times}$ over $E$ such that $D_{\rig}(\rho)$ lies in an exact sequence as follows:
\begin{equation*}
  0 \ra \cR_E(\delta_1) \ra D_{\rig}(V) \ra \cR_E(\delta_2) \ra 0.
  \end{equation*}
  Such an exact sequence is called a triangulation, denoted by $(\rho,\delta_1,\delta_2)$, of $D_{\rig}(\rho)$ (and of $\rho$).
\end{definition}
We refer to \cite{Na} for a classification of $2$-dimensional trianguline representations of $\Gal(\overline{\Q_p}/F_{\wp})$. Note that if $\rho$ is semi-stable, then $\rho$ is trianguline, if $\rho$ is moreover non-crystalline, then the triangulation of $\rho$ is unique.
\begin{definition}[$\text{cf. \cite[Def.4.3.1]{Liu}}$]
  Let $\rho$ be a $2$-dimensional trianguline representation of $\Gal(\overline{\Q_p}/F_{\wp})$ over $E$ with $(\rho, \delta_1,\delta_2)$ a triangulation of $\rho$. For $\sigma\in \Sigma_{\wp}$, we say that $\rho$ is non-$\sigma$-critical if $k_{\delta_1,\sigma}-k_{\delta_2,\sigma}\in \Z_{\geq 1}$. More generally, for $J\subseteq \Sigma_{\wp}$, we say $\rho$ is non-$J$-critical if $\rho$ is non-$\sigma$-critical for all $\sigma\in J$, we say $\rho$ is non-critical if $\rho$ is non-$\Sigma_{\wp}$-critical.
\end{definition}For a closed point $z=(\chi_z=\chi_{z,1}\otimes \chi_{z,2}, \lambda_z)$ of $\cV(K^p,w)_{\overline{\rho}}$, put $k_{z,\sigma}:= k_{\chi_{z,1},\sigma}-k_{\chi_{z,2},\sigma}+2\in \Z_{\geq 2}$, for $\sigma\in C(\chi_z)$. If $z\in C(w)$, we see $k_{z,\sigma}\equiv w \pmod{2}$ for all $\sigma \in \Sigma_{\wp}$ (note that $C(\chi_z)=\Sigma_{\wp}$ in this case).
The following theorem can be easily deduced from the results in \cite{Sa} (together with Cor.\ref{prop: clin-ernr} and the results of \cite{Na} on  triangulations of semi-stable representations, see \cite[Prop.6.2.44]{Ding} for the unitary Shimura curves  case).
\begin{theorem}\label{thm: clin-zgn}
  Let $z=\big(\chi_{z,1}\otimes \chi_{z,2}, \lambda_z\big)\in C(w)$, then $\rho_{z,\wp}$ is semi-stable (hence trianguline) with a triangulation given by
  \begin{equation}\label{equ: clin-dgz}
    0 \ra \cR_{k(z)}(\delta_{z,1}) \ra D_{\rig}(\rho_{z,\wp}) \ra \cR_{k(z)}\big(\delta_{z,2}\big)\ra 0
  \end{equation}
  where
  \begin{equation*}
    \begin{cases}
       \delta_{z,1}=\unr(q)\chi_{z,1} \prod_{\sigma\in \Sigma_{z}} \sigma^{1-k_{z,\sigma}} \\
       \delta_{z,2}=\chi_{z,2}\prod_{\sigma\in \Sigma_{\wp}} \sigma^{-1}\prod_{\sigma\in \Sigma_{z}} \sigma^{k_{z,\sigma}-1}
    \end{cases}
  \end{equation*}
  with $\Sigma_{z}\subseteq \Sigma_{\wp}$ (maybe empty). Put $\psi_{z,1}:=\chi_{z,1}\prod_{\sigma\in \Sigma_{\wp}}\sigma^{\frac{w-k_{z,\sigma}+2}{2}}$ (being an unramified character of $F_{\wp}^{\times}$), for $S\subseteq \Sigma_{\wp}$, if one has
  \begin{equation*}
    \us_{\wp}\big(q\psi_{z,1}(\varpi)\big) < \sum_{\sigma\in \Sigma_{\wp}} \frac{w-k_{z,\sigma}+2}{2}+\inf_{\sigma\in S}\{k_{z,\sigma}-1\},
  \end{equation*}then $\Sigma_{z}\cap S=\emptyset$, in particular, in this case the triangulation (\ref{equ: clin-dgz}) is non-$S$-critical.
\end{theorem}
Denote by $C(w)_0$ the subset of $C(w)$ of points $z$ such that
\begin{equation}\label{equ: clin-man}
\us_{\wp}\big(q\psi_{z,1}(\varpi)\big) < \sum_{\sigma\in \Sigma_{\wp}} \frac{w-k_{z,\sigma}+2}{2}+\inf_{\sigma\in \Sigma_{\wp}}\{k_{z,\sigma}-1\}.
\end{equation}
As in \cite[Prop.6.2.7, Prop.6.4.6]{Che}, one can prove $C(w)_0$ is Zariski-dense in $\cV(K^p,w)_{\overline{\rho}}$, and is an accumulation subset (cf. \cite[\S 3.3.1]{BCh}).
By the theory of global triangulation, one has
\begin{theorem}\label{thm: clin-iz2d}
  Let  $z=(\chi_z=\chi_{z,1}\otimes \chi_{z,2}, \lambda)$ be a closed point of $\cV(K^p,w)_{\overline{\rho}}$, then the representation $\rho_{z,\wp}$ is trianguline with a triangulation given by
  \begin{equation*}
    0 \ra \cR_{k(z)}(\delta_{z,1}) \ra D_{\rig}(\rho_{z,\wp}) \ra \cR_{k(z)}\big(\delta_{z,2}\big)\ra 0
  \end{equation*}
  where
  \begin{equation*}
    \begin{cases}
       \delta_{z,1}=\unr(q)\chi_{z,1}\prod_{\sigma\in \Sigma_{z}} \sigma^{1-k_{z,\sigma}} \\
       \delta_{z,2}=\chi_{z,2}\prod_{\sigma\in \Sigma_{\wp}}\sigma^{-1}\prod_{\sigma\in \Sigma_{z}} \sigma^{k_{z,\sigma}-1}
    \end{cases}
  \end{equation*}
  with $\Sigma_{z}$ a subset (maybe empty) of $C(\chi_z)\cap \Sigma_{\wp}$.
  \end{theorem}
\begin{proof}
  Since $\cV(K^p,w)_{\overline{\rho},\red}$ is nested, and $C(w)_0$ is Zariski-dense, there exists an irreducible affinoid neighborhood of $z$ such that $C(w)_0\cap U(\overline{E})$ is Zariski-dense in $U$ (e.g. see \cite[Lem.7.2.9]{BCh}). Denote by $g: \widetilde{U}\ra U$ the rigid space as in Prop. \ref{prop: clin-wpw}, thus $g^{-1}\big(C(w)_0\cap U(\overline{E})\big)$ is Zariski-dense in $\widetilde{U}$. The theorem then follows from \cite[Thm.6.3.13]{KPX} and \cite[Ex.6.3.14]{KPX} (see also \cite[Thm.4.4.2]{Liu}).
\end{proof}
\begin{corollary}\label{cor: clin-aat}Keep the notation in Thm.\ref{thm: clin-iz2d}, suppose moreover
  \begin{equation}\label{equ: clin-iz-}\unr(q^{-1}) \chi_{z,1}^{-1}\chi_{z,2}\neq \prod_{\sigma\in \Sigma_{\wp}} \sigma^{n_{\sigma}} \text{ for all $\ul{n}_{\Sigma_{\wp}}\in \Z^{d}$},\end{equation} let $S\subseteq C(\chi_z)$, if $\Sigma_{z}\cap S=\emptyset$, then $z$ does not have $S'$-companion point for any $S'\subseteq S$, $S'\neq \emptyset$. As a result, the point $z$ is quasi-$S$-classical.
\end{corollary}
\begin{proof}
The second part follows from the first part and Prop.\ref{prop: clin-vka}. We prove the first part. Let $S'\subseteq S$, $S'\neq \emptyset$, suppose $z$ admits an $S'$-companion point $z_{S'}^c$, by applying Thm.\ref{thm: clin-iz2d} to the point $z_{S'}^c$, one can get a triangulation $\big(\rho_{z_{S'}^c,\wp}, \delta_{z_{S'}^c ,1}, \delta_{z_{S'}^c 2}\big)$ for $\rho_{z_{S'}^c,\wp}\cong \rho_{z,\wp}$
Note that $S'\cap C((\chi_z)_{S'}^c)=\emptyset$, so $S'\cap \Sigma_{z_{S'}^c,\wp}=\emptyset$. By the hypothesis (\ref{equ: clin-iz-}) and \cite[Thm.3.7]{Na}, one can check the triangulations $\big(\rho_{z_{S'}^c,\wp}, \delta_{z_{S'}^c ,1}, \delta_{z_{S'}^c ,2}\big)$ and $\big(\rho_{z,\wp}, \delta_{z,1}, \delta_{z,2}\big)$ are the same. As a result, one sees $S'\subseteq \Sigma_{z}$, a contradiction.
\end{proof}
\begin{corollary}\label{cor: clin-siz}Keep the notation in Thm.\ref{thm: clin-iz2d}, suppose $\chi_z$ is spherically algebraic \big(thus $\chi_{i,z}=\psi_{i,z}\prod_{\sigma\in \Sigma_{\wp}} \sigma^{k_{\chi_i,\sigma}}$ with $\psi_{i,z}$ an unramified character of $F_{\wp}^{\times}$\big) and satisfies
\begin{equation*}
  \psi_{1,z}(p)^{-1}\psi_{2,z}(p) q^{-e} \neq 1
\end{equation*}
\big(note this condition is slightly stronger than the hypothesis (\ref{equ: clin-iz-})\big), then there exists an open affinoid neighborhood $U$ of $z$ in $\cV(K^p,w)_{\overline{\rho},\red}$ containing $z$ such that for any closed point $z'=(\chi_{z'},\lambda')\in U(\overline{E})$, $\Sigma_{z'}=\emptyset$ and $z'$ does not have companion point.
\end{corollary}
\begin{proof}
 As in the proof of \cite[Prop.6.2.7]{Che}, one can prove $C(w)_0$ accumulates over $z$. Thus one can choose an open affinoid neighborhood $U_0$ of $z$ such that
   \begin{enumerate}\item $C(w)_0\cap U_0(\overline{E})$ is Zariski-dense in $U_0$,
   \item $\unr(q^{-1}) \chi_{z',1}^{-1} \chi_{z',2} \neq \prod_{\sigma\in \Sigma_{\wp}} \sigma^{n_{\sigma}}$ for any $\ul{n}_{\Sigma_{\wp}}\in \Z^{d}$, $z'\in U_0(\overline{E})$ (see Lem.\ref{lem: clin-cte}(1) below).\end{enumerate}
  By \cite[Thm.6.3.9]{KPX}, $Z_{U_0}:=\{z'\in U_0(\overline{E})\ |\ \Sigma_{z'}\neq \emptyset \}$ is a Zariski-closed subset of $U_0$ and $z\notin Z_{U_0}$. So there exists an open affinoid $U$ of $U_0$ containing $z$ such that $Z_{U}$ (defined in the same way as $Z_{U_0}$ by replacing $U_0$ by $U$) is empty. The corollary follows.
\end{proof}
\begin{lemma}\label{lem: clin-cte}
  Let $\chi_1\otimes \chi_2$ be a spherically algebraic character of $T(F_{\wp})$ (which can be seen as a closed point of $\widehat{T}_{\Sigma_{\wp}}$), and let $\psi_i:=\chi_i\prod_{\sigma\in \Sigma_{\wp}} \sigma^{-k_{\chi_1,\sigma}}$.

  (1) Suppose $\psi_{1}(p)^{-1}\psi_{2}(p) q^{-e} \neq 1$, then there exists an admissible neighborhood $U$ of $\chi_1\otimes \chi_2$ in $\widehat{T}_{\Sigma_{\wp}}$ such that $\unr(q^{-1}) (\chi_{1}')^{-1} \chi_{2}' \neq \prod_{\sigma\in \Sigma_{\wp}} \sigma^{n_{\sigma}}$ for any $\ul{n}_{\Sigma_{\wp}}\in \Z^{d}$ and $\chi_1'\otimes \chi_2'\in U(\overline{E})$.

  (2) Suppose $\psi_1(\varpi)^{-1} \psi_2(\varpi)^{-1}q^{-1} \neq 1$, then there exists an admissible neighborhood $U$ of $\chi_1\otimes \chi_2$ in $\widehat{T}_{\Sigma_{\wp}}$ such that $\unr(q^{-1}) (\chi_{1}')^{-1} \chi_{2}' \neq \prod_{\sigma\in \Sigma_{\wp}} \sigma^{n_{\sigma}}$ for any $\ul{n}_{\Sigma_{\wp}}\in \Z_{\geq k_{\chi_2,\sigma}-k_{\chi_1,\sigma}}^{d}$ and $\chi_1'\otimes \chi_2'\in U(\overline{E})$.
\end{lemma}
\begin{proof}
  Denote by $\widehat{Z}_{\Sigma_{\wp}}$ the rigid space parameterizing locally $\Q_p$-analytic characters of $F_{\wp}^{\times}$. One has a morphism of rigid spaces:
  \begin{equation}\label{equ: clin-pwz}
    \widehat{T}_{\Sigma_{\wp}} \lra \widehat{Z}_{\Sigma_{\wp}}, \  (\chi_1')^{-1}\otimes \chi_2'\mapsto \chi_1'\chi_2'.
  \end{equation}
Let $\psi_0:=\unr(q^{-1}) \psi_1^{-1}\psi_2$, we claim that \begin{itemize}\item if $\psi_0(p)\neq 1$ (resp. $\psi_0(\varpi)\neq 1$), then there exists an admissible open $U_0$ of $\widehat{Z}_{\Sigma_{\wp}}$ containing $\psi_0$ (where $\psi_0$ is seen as a closed point of $\widehat{Z}_{\Sigma_{\wp}}$\big) such that $\prod_{\sigma\in \Sigma_{\wp}} \sigma^{n_{\sigma}}\notin U(\overline{E})$ for any $\ul{n}_{\Sigma_{\wp}}\in \Z^{d}$ \big(resp. $\prod_{\sigma\in \Sigma_{\wp}} \sigma^{n_{\sigma}} \notin U(\overline{E})$ for any $\ul{n}_{\Sigma_{\wp}}\in \Z_{\geq 0}^d$\big).
\end{itemize}
 Assuming this claim, and  let $U$ be the preimage of the admissible open $\big(\prod_{\sigma\in \Sigma_{\wp}} \sigma^{-k_{\chi_1,\sigma}+k_{\chi_2,\sigma}}\big) U_0$ \big(of $\widehat{Z}_{\Sigma_{\wp}}$\big) in $\widehat{T}_{\Sigma_{\wp}}$ via (\ref{equ: clin-pwz}). Thus $U$ satisfies the property in the lemma (1) (resp. (2)).
 
We prove the claim. Consider the projection $\widehat{Z}_{\Sigma_{\wp}}\ra \cW_{\Sigma_{\wp}} \times \bG_m$, $\chi\mapsto \big(\chi|_{\co_{\wp}^{\times}}, \chi(p)\big)$ \Big(resp. the isomorphism $\widehat{Z}_{\Sigma_{\wp}} \xrightarrow{\sim} \cW_{\Sigma_{\wp}} \times \bG_m$, $\chi \mapsto \big(\chi|_{\co_{\wp}^{\times}}, \chi(\varpi)\big)$\Big), set $a:=\psi_0(p)$ (resp. $a:=\psi_0(\varpi)$), which is the image of $\psi_0$ in $\bG_m$. We discuss in the following two cases:

If $\us_{\wp}(a)\neq 0$, then choose $n\in \Z_{\geq 1}$ such that $\us_{\wp}(a)\notin p^n\Z$; let $U_1$ be an admissible open in $\cW_{\Sigma_{\wp}}$ containing the trivial character such that if  $\prod_{\sigma \in \Sigma_{\wp}} \sigma^{n_{\sigma}}\big|_{\co_{\wp}^{\times}}\in U_1(\overline{E})$ then $p^n|n_{\sigma}$ for all $\sigma$, $U_2$ be an admissible open in $\bG_m$ containing $a$ such that $\us_{\wp}(a')=\us_{\wp}(a)$ for all $a'\in U(\overline{E})$, one easily check the admissible open $U_0:=U_1\times U_2$ satisfies the property in the claim.

If $\us_{\wp}(a)=0$, since $a\neq 1$ by hypothesis, let $U_2$ be an admissible open in $\bG_m$ such that $a\in U_2(E)$, $1\notin U_2(E)$ and for all $a'\in U_2(\overline{E})$, $\us_{\wp}(a')=0$; put $U_0:=\cW_{\Sigma_{\wp}} \times U_2$, we see if $\chi'=\prod_{\sigma\in \Sigma_L} \sigma^{n_{\sigma}}\in U_0(\overline{E})$ for $\ul{n}_{\Sigma_{\wp}}\in \Z^d$ \big(resp. for $\ul{n}_{\Sigma_{\wp}}\in \Z_{\geq 0}^d$\big),  thus $\us_{\wp}(\chi'(p))=0$ (resp. $\us_{\wp}(\chi'(\varpi))=0$), thus $\sum_{\sigma\in \Sigma_L}n_{\sigma}=0$, so $\chi'(p)=1\notin U_2(\overline{E})$ \big(resp. $n_{\sigma}=0$ for all $\sigma\in \Sigma_{\wp}$ hence $\chi'(\varpi)=1\notin U_2(\overline{E})$\big), a contradiction, so $U_0$ satisfies the property in the claim.
\end{proof}

\subsubsection{\'Etaleness of eigenvarieties at non-critical classical points}\label{sec: clin-3.4.3}
Let $z=\big(\chi_z=\chi_{z,1} \otimes \chi_{z,2}, \lambda_z\big)$ be a semi-stable classical point of $\cV(K^p,w)_{\overline{\rho}}$, for $\sigma\in \Sigma_{\wp}$, we say that $z$ is \emph{non-critical} if
\begin{enumerate} \item the triangulation $(\rho_{z,\wp},\delta_{z,1},\delta_{z,2})$ (cf. Thm.\ref{thm: clin-iz2d}) is non-critical (i.e. $\Sigma_{z}=\emptyset$),\\
\item $\unr(q^{-1}) \chi_{z,1}^{-1} \chi_{z,2} \neq \prod_{\sigma\in \Sigma_{\wp}} \sigma^{n_{\sigma}}$ for any $\ul{n}_{\Sigma_{\wp}}\in \Z^{d}$.
\end{enumerate}
Let $\psi_{z,1}$, $\psi_{z,2}$ be unramified characters of $F_{\wp}^{\times}$ such that $\chi_{z,i}=\psi_{z,i}\prod_{\sigma\in \Sigma_{\wp}}\sigma^{k_{\chi_{z,i},\sigma}}$, then the condition (2) is equivalent to $\psi_{z,1}(q\varpi)\neq \psi_{z,2}(\varpi)$. If one considers the Galois representation $\rho_{z,{\wp}}$ (which is semi-stable), this condition means the eigenvalues of $\varphi^{d_0}$ on $D_{\st}(\rho_{z,\wp})$ are different.

Consider the natural morphism
\begin{equation*}
  \kappa: \cV(K^p,w)_{\overline{\rho}} \lra \widehat{T}_{\Sigma_{\wp}} \lra \cW_{1,\Sigma_{\wp}},
\end{equation*}
where the last map is induced by the inclusion $Z_1' \ra T(F_{\wp})$ (see also (\ref{equ: clin-st1})). This section is devoted to prove the following result.
\begin{theorem}\label{thm: clin-elt}Let $z$ be a non-critical semi-stable classical point of the rigid space $ \cV(K^p,w)_{\overline{\rho}}$, then  $\cV(K^p,w)_{\overline{\rho}}$ is \'etale over $\cW_{1,\Sigma_{\wp}}$ at $z$.
\end{theorem}
The theorem follows by the same argument as in the proof of \cite[Thm.4.8]{Che11}.
Let $z=(\chi_z=\chi_{z,1}\otimes \chi_{z,2},\lambda_z):\Spec \overline{E} \ra \cV(K^p,w)_{\overline{\rho}}$ be a non-critical semi-stable classical point of $C(w)$, by the construction of $ \cV(K^p,w)_{\overline{\rho}}$ as in \S \ref{sec: clin-3.3}, one can find a connected affinoid neighborhood $U$ of $\kappa(z)$ in $\cW_{1,\Sigma_{\wp}}$ and a finite locally free $\co(U)$-module $M$ equipped with an $\co(U)$-linear action of $\cH$ such that (see also the proof of \cite[Thm.4.8]{Che11})
\begin{enumerate}
  \item the affinoid spectrum $V$ of $\Ima\big(\co(U) \otimes_{\co_E} \cH \ra \End_{\co(U)}(M)\big)$ is an affinoid neighborhood of $z$ in $ \cV(K^p,w)_{\overline{\rho}}$ \big(thus one has $\cM(K^p,w)_{\overline{\rho}}(\co(V))\cong M$ as $\co(V)$-module\big);

  \item for each continuous character $\chi\in \cW_{1,\Sigma_{\wp}}(\overline{E})$, there is a $T(F_{\wp})\times \cH^p$-invariant isomorphism
  \begin{equation*}
    M\otimes_{\co(U), \chi} \overline{E}\cong
     \bigoplus_{(\chi_{z'},\lambda_{z'})\in \kappa^{-1}(\chi)} \Big(J_B\big(\widetilde{H}^1_{\et}(K^p, E)_{\Q_p-\an}\big) \otimes_E \overline{E}\Big)^{Z_1=\sN^{-w},Z_1'=\chi}[T(F_{\wp})=\chi_{z'},\cH^p=\lambda_{z'}]^{\vee};
  \end{equation*}
  \item $\kappa^{-1}(\kappa(z))^{\red}=\{z\}$ and the natural surjection $\co(V) \ra k(z)$ has a section.
\end{enumerate}
Let $Z_0\subseteq U$ be the set of closed points $\chi$ such that any point in $\kappa^{-1}(\chi)\cap V(\overline{E})$ is classical, thus $Z_0$ is Zariski-dense in $U$ \big(by Thm.\ref{thm: clin-etn}, note that $\us_{\wp}(\chi_{z',1}(\varpi))$ is bounded for $z'\in V(\overline{E})$\big), and $Z:=\kappa^{-1}(Z_0) \cup \{z\}$ is Zariski-dense in $V$. Up to shrinking $Z_0$, one can assume that for any $z'\in Z$,
\begin{enumerate}[label=(\alph*)]
\item $\Sigma_{z'}=\emptyset$  (since $z$ is supposed to be non-critical, for $z'\neq z$, this would follow from Thm.\ref{thm: clin-zgn}),
 \item $\unr(q^{-1}) \chi_{z',1}^{-1} \chi_{z',2} \neq \prod_{\sigma\in \Sigma_{\wp}} \sigma^{n_{\sigma}}$ for any $\ul{n}_{\Sigma_{\wp}}\in \Z^{d}$ \big(e.g. by Lem.\ref{lem: clin-cte} (2), since by shrinking $Z_0$, one can assume $k_{\chi_{z',2},\sigma}-k_{\chi_{z',1},\sigma}\geq k_{\chi_{z,2},\sigma}-k_{\chi_{z,1},\sigma}$ for all $\sigma\in \Sigma_{\wp}$, and $z'\in Z_0$\big).
\end{enumerate}Let $z'=(\chi_{z'},\lambda_{z'})\in Z$, denote by $k(z')$ the residue field at $z'$. One has an isomorphism (cf. (\ref{equ: clin-0nvi}))
\begin{multline*}
  J_B\Big(\widetilde{H}^1_{\et}(K^p,E)_{\Q_p-\an}\otimes_E k(z')\Big)^{Z_1=\sN^{-w}, Z_1'=\kappa(z')}[T(F_{\wp})=\chi_{z'},\cH^p=\lambda_{z'}] \\ \xlongrightarrow{\sim} \Big(\widetilde{H}^1_{\et}(K^p,E)_{\Q_p-\an}\otimes_E k(z')\Big)^{N_0,Z_1=\sN^{-w},Z_1'=\kappa(z')}[T(F_{\wp})=\chi_{z'}, \cH^p=\lambda_{z'}].
\end{multline*}
\begin{lemma}\label{lem: clin-adb}
Keep the above notation, any vector in $$\Big(\widetilde{H}^1_{\et}(K^p,E)_{\Q_p-\an}\otimes_E k(z')\Big)^{N_0,Z_1=\sN^{-w},Z_1'=\kappa(z')}[T(F_{\wp})=\chi_{z'}, \cH^p=\lambda_{z'}]$$ is classical.
\end{lemma}
\begin{proof}
  Suppose there exists a non-classical vector $v$, thus there exists $\sigma\in \Sigma_{\wp}$, such that $v$ is non-$\sigma$-classical. By Prop.\ref{prop: clin-vka}, one can prove $z'$ admits a $\sigma$-companion point, which would lead to a contradiction by the same argument as in the proof of Cor.\ref{cor: clin-aat}.
\end{proof}
Keep the above notation (so $z'\in Z$), put $\psi_{z'}:=\chi_{z'}\big(\prod_{\sigma\in \Sigma_{\wp}} \sigma^{-k_{\chi_{1,z'},\sigma}} \otimes \prod_{\sigma\in \Sigma_{\wp}} \sigma^{-k_{\chi_{2,z'},\sigma}}\big)$ (being a smooth character of $T(F_{\wp})$), $T_0:=Z_1Z_1'$, by Lem.\ref{lem: clin-adb}, Cor.\ref{prop: clin-ernr}, one has an isomorphism of $k(z')$-vector spaces
\begin{multline*}
   \Big(\widetilde{H}^1_{\et}(K^p,E)_{\Q_p-\an}\otimes_E k(z')\Big)^{N_0,Z_1=\sN^{-w},Z_1'=\kappa(z')}[T(F_{\wp})=\chi_{z'}, \cH^p=\lambda_{z'}]
   \\ \xlongrightarrow{\sim} \Big(H^1_{\et}\big(K^p,W(\ul{k}_{\Sigma_{\wp}},w)\big)\otimes_E k(z')\Big)^{N_0, T_0=\psi_{z'}}[T(F_{\wp})=\psi_{z'},\cH^p=\lambda_{z'}],
\end{multline*}
where $k_{\sigma}:=k_{\chi_{1,z'},\sigma}-k_{\chi_{2,z'},\sigma}+2$ for all $\sigma\in \Sigma_{\wp}$. So
\begin{multline*}
 J_B\Big(\widetilde{H}^1_{\et}(K^p,E)_{\Q_p-\an}\otimes_E k(z')\Big)^{Z_1=\sN^{-w}, Z_1'=\kappa(z')}[T(F_{\wp})=\chi_{z'},\cH^p=\lambda_{z'}] \\\xlongrightarrow{\sim} J_B\Big(H^1_{\et}\big(K^p,W(\ul{k}_{\Sigma_{\wp}},w)\big)\otimes_E k(z')\Big)^{T_0=\psi_{z'}}[T(F_{\wp})=\psi_{z'},\cH^p=\lambda_{z'}].
\end{multline*}
 Denote by $\delta(z')$ the dimension of the above vector space over $k(z')$. Set
\begin{equation*}
  H^1_{\et}\big(W(\ul{k}_{\Sigma_{\wp}},w)\big):=\varinjlim_{(K^p)'}H^1_{\et}\big((K^p)',W(\ul{k}_{\Sigma_{\wp}},w)\big)\otimes_E \overline{E}
\end{equation*}
where $(K^p)'$ runs over open compact subgroups of $K^p$, this is a smooth admissible representation of $G(\bA^{\infty})$ equipped with a continuous action of $\Gal(\overline{F}/F)$. One has a decomposition of $G(\bA^{\infty}) \times \Gal(\overline{F}/F)$-representations
\begin{equation*}
  H^1_{\et}\big(W(\ul{k}_{\Sigma_{\wp}},w)\big) \cong \bigoplus_{\pi} \rho(\pi) \otimes \pi
\end{equation*}
where $\pi$ runs over  irreducible smooth admissible representations of $G(\bA^{\infty})$. It's known that if $\rho(\pi)\neq 0$, then $\dim_{\overline{E}}\rho(\pi)=2$ (e.g. see \cite[\S 2.2.4]{Ca2}). A necessary condition for $\rho(\pi)$ to be non-zero is that there exists an admissible representation $\pi_{\infty}$ of $G(\R)$ such that $\pi_{\infty}\otimes \pi$ is an automorphic representation of $G(\bA)$ (we fix an isomorphism $\overline{E}\xrightarrow{\sim} \bC$).  Note that one has
\begin{equation*}
 H^1_{\et}\big(K^p, W(\ul{k}_{\Sigma_{\wp}},w)\big)\otimes_{E} \overline{E} \xlongrightarrow{\sim} H^1_{\et}\big(W(\ul{k}_{\Sigma_{\wp}},w)\big)^{K^p}\cong  \bigoplus_{\pi} \rho(\pi) \otimes \pi^{K^p}.
\end{equation*}
For an irreducible smooth admissible representation $\pi$ of $G(\bA^{\infty})$, $\pi$ admits thus a decomposition $\pi\cong \otimes_{\fl} \pi_{\fl}$ with $\pi_{\fl}$ an irreducible smooth admissible representation of $(B\otimes_F F_{\fl})^{\times}$, where $\fl$ runs over the finite places of $F$. Recall (e.g. see \cite[Thm.VI.1.1(4)]{HT})
\begin{proposition}
Let $\pi_1$, $\pi_2$ be two automorphic representations of $G(\bA)$, if $\pi_{1,\fl}\cong \pi_{2,\fl}$ for all but finitely many places $\fl$ of $F$, then $\pi_1\cong \pi_2$.
\end{proposition}
By this proposition (and the above discussions), there exists a unique irreducible smooth admissible representation $\pi_{z'}$ of $G(\bA^{\infty})$ such that the action of $\cH^p$ on $(\pi_{z'})^{K^p}$ is given by $\lambda_{z'}$ and that $\rho(\pi_{z'})\neq 0$ \big(one has in fact $\rho_{z'}\cong \rho(\pi_{z'})$\big). Thus
\begin{equation*}
  H^1_{\et}\big(K^p,W(\ul{k}_{\Sigma_{\wp}},w)\big)[\cH^p=\lambda_{z'}]\cong \rho_{z'} \otimes\pi_{z',\wp} \otimes \Big(\Big(\bigotimes_{\fl \neq \wp} \pi_{z',\fl}^{(K^p)_{\fl}}\Big)[\cH^p=\lambda_{z'}]\Big).
\end{equation*}
We have the following facts:
\begin{itemize}
  \item $\dim_{\overline{E}} J_B(\pi_{z',\wp})[T(F_{\wp})=\psi_{z'}]=1$ (by classical Jacquet module theory and the condition (b)),
  \item $\dim_{\overline{E}} \pi_{z',\fl}^{(K^p)_{\fl}}=1$, for all $\fl\in S(K^p)$,
\end{itemize}
from which we deduce
\begin{multline*}
  J_B\big(H^1_{\et}\big(K^p, W(\ul{k}_{\Sigma_{\wp}},w)\big)\otimes_E k(z')\big)^{T(F_{\wp})=\psi_{z'}, \cH^p=\lambda_{z'}}\\ \xlongrightarrow{\sim}J_B\big(H^1_{\et}\big(K^p, W(\ul{k}_{\Sigma_{\wp}},w)\big)\otimes_E k(z')\big)^{T_0=\psi_{z'}}[T(F_{\wp})=\psi_{z'},\cH^p=\lambda_{z'}].
\end{multline*}
Denote by $S'$ the complement of $S(K^p) \cup \{\wp\}$ in the set of finite places of $F$ (thus $S'$ is a finite set), we also  deduce (compare with \cite[(4.21)]{Che11})
\begin{equation*}
  \delta(z')=2\sum_{\fl \in S'} \dim_{\overline{E}}\big(\pi_{z',\fl}^{(K^p)_{\fl}}\big).
\end{equation*}
By the same argument as in the proof of \cite[Thm.4.8]{Che11}, one can prove $\delta(z')\geq \delta(z)$ for all $z'\in Z$, and then deduce that $\co(V)\cong \co(U)\otimes_E k(z)$. The theorem follows.
\begin{remark}\label{rem: clin-rat}
  Keep the above notation, if $z$ is moreover an $E$-point of $ \cV(K^p,w)_{\overline{\rho}}$ (in practice, one can always enlarge $E$ if necessary),  thus one has $\co(V) \cong \co(U)$. So the action of $T(F_{\wp})$ on $M$ is given by the character $T(F_{\wp})\ra \co(V)^{\times}\cong \co(U)^{\times}$ induced by the natural morphism $V\ra \widehat{T}_{\Sigma_{\wp}}$.
\end{remark}
\section{$\cL$-invariants and local-global compatibility}
\subsection{Fontaine-Mazur $\cL$-invariants}\label{sec: clin-ene}
Recall Fontaine-Mazur $\cL$-invariants for $2$-dimensional semi-stable non-crystalline representations of $\Gal(\overline{\Q_p}/F_{\wp})$.

Let $k_{1,\sigma}, k_{2,\sigma}\in \Z$, $k_{1,\sigma}<k_{2,\sigma}$ for all $\sigma\in \Sigma_{\wp}$; let $\rho$ be a $2$-dimensional semi-stable non-crystalline representation of $\Gal(\overline{\Q_p}/F_{\wp})$ over $E$ of Hodge-Tate weights $(-k_{2,\sigma}, -k_{1,\sigma})_{\sigma\in \Sigma_{\wp}}$. By Fontaine's theory (cf. \cite{Fon94}, \cite{FO}), one can associate to $\rho$ a filtered $(\varphi,N)$-module $(D_0, D)$ where $D_0:=D_{\st}(\rho):= (B_{\st}\otimes_{\Q_p} \rho)^{\Gal(\overline{\Q_p}/F_{\wp})}$ is a free $F_{\wp,0}\otimes_{\Q_p} E$-module of rank $2$ equipped with a bijective \big($F_{\wp,0}$-semi-linear and $E$-linear\big) endomorphism $\varphi$ and a nilpotent $F_{\wp,0}\otimes_{\Q_p} E$-linear operator $N$ such that $N\varphi=p\varphi N$,  and that $D:=D_0\otimes_{F_{\wp,0}}F_{\wp} \cong D_{\dR}(\rho):=(B_{\dR}\otimes_{\Q_p} \rho)^{\Gal(\overline{\Q_p}/F_{\wp})}$ is a free $F_{\wp}\otimes_{\Q_p} E$-module of rank $2$ equipped with a decreasing exhaustive separated filtration by $F_{\wp} \otimes_{\Q_p} E$-submodules.

Using the isomorphism
\begin{equation*}
  F_{\wp,0}\otimes_{\Q_p} E \xlongrightarrow{\sim} \prod_{\sigma_0: F_{\wp,0}\ra E} E, \  a \otimes b \mapsto \big(\sigma_0(a)b\big)_{\sigma_0: F_{\wp,0}\ra E},
\end{equation*}
one can decompose $D_0$ as $D_0\xrightarrow{\sim} \prod_{\sigma_0: F_{\wp,0}\ra E} D_{\sigma_0}$. Each $D_{\sigma_0}$ is an $E$-vector space of rank $2$ equipped with an $E$-linear action of $\varphi^{d_0}$ and $N$, moreover, the operator $\varphi$ (on $D_0$) induces a bijection: $D_{\sigma_0}\xrightarrow{\sim} D_{\sigma_0\circ \varphi^{-1}}$. It's known that $\Ker(N)$ is a free $F_{\wp,0}\otimes_{\Q_p} E$-module of rank $1$, and thus admits a decomposition $\Ker(N)\xrightarrow{\sim}  \prod_{\sigma_0: F_{\wp,0}\ra E} \Ker(N)_{\sigma_0}$. Let $e_{0,\sigma_0}\in D_{\sigma_0}$ such that $Ee_{0,\sigma_0}=\Ker(N)_{\sigma_0}$. In fact, one can choose $e_{0,\sigma_0}$ such that \begin{equation}\label{equ: clin-e0e}\varphi(e_{0,\sigma_0})=e_{0,\sigma_0\circ \varphi^{-1}}.\end{equation} Since $\Ker(N)_{\sigma_0}$ is stable by $\varphi^{d_0}$, there exists $\alpha\in E^{\times}$ such that $\varphi^{d_0}(e_{0,\sigma_0})=\alpha e_{0,\sigma_0}$ \big(by (\ref{equ: clin-e0e}), we see $\alpha$ is independent of $\sigma_0$\big). Since $N\varphi=p\varphi N$, there exists a unique $e_{1,\sigma_0}\in D_{\sigma_0}$ such that $Ne_{1,\sigma_0}=e_{0,\sigma_0}$ and $\varphi^{d_0}(e_{1,\sigma_0})=q\alpha e_{1,\sigma_0}$ (thus $D_{\sigma_0}=E e_{0,\sigma_0} \oplus E e_{1,\sigma_0}$).

Using the isomorphism
\begin{equation*}
  F_{\wp}\otimes_{\Q_p} E \xlongrightarrow{\sim} \prod_{\sigma\in \Sigma_{\wp}} E, \  a \otimes b \mapsto (\sigma(a)b)_{\sigma\in \Sigma_{\wp}},
\end{equation*}
one can decompose $D$ as $D\xrightarrow{\sim} \prod_{\sigma\in \Sigma_{\wp}} D_{\sigma}$. One has $D_{\sigma_0}\otimes_{F_{\wp,0}} F_{\wp} \cong \prod_{\substack{\sigma\in \Sigma_{\wp} \\ \sigma|_{F_{\wp,0}}=\sigma_0}} D_{\sigma}$ for any $\sigma_0: F_{\wp,0}\ra E$. For $\sigma\in \Sigma_{\wp}$, $i=0,1$, let $e_{i,\sigma}\in D_{\sigma}$, such that
\begin{equation*}
  e_{i,\sigma_0} \otimes 1 = \big(e_{i,\sigma}\big)_{\substack{\sigma\in \Sigma_{\wp} \\ \sigma|_{F_{\wp,0}}=\sigma_0}}.
\end{equation*}
Since $\rho$ is of Hodge-Tate weights $(-k_{2,\sigma}, -k_{1,\sigma})_{\sigma\in \Sigma_{\wp}}$, for all $\sigma\in \Sigma_{\wp}$, there exists $(a_{\sigma},b_{\sigma})\in E\times E \setminus \{(0,0)\}$ such that
\begin{equation*}
  \Fil^{i} D_{\sigma}=\begin{cases}
     D_{\sigma} & i\leq k_{1,\sigma} \\
     E\big(a_{\sigma}e_{1,\sigma}+ b_{\sigma}e_{0,\sigma}\big) & k_{1,\sigma}<i \leq k_{2,\sigma} \\
     0 & i > k_{2,\sigma}
  \end{cases}.
\end{equation*}
We suppose $\rho$ satisfies the following hypothesis.
\begin{hypothesis}\label{hyp: clin-aq0}
  For all $\sigma\in \Sigma_{\wp}$, $a_{\sigma}\neq 0$.
\end{hypothesis}
\begin{remark}
This hypothesis is automatically satisfied when $F_{\wp}=\Q_p$ by the weak admissibility of $(D_0,D)$.
\end{remark}
Pose $\cL_{\sigma}:=b_{\sigma}/a_{\sigma}$, for  $\sigma\in \Sigma_{\wp}$. One sees easily that $\cL_{\sigma}$ is independent of the choice of $e_{0,\sigma}$. An important fact is that one can recover $\big(D_{\st}(\rho),D_{\dR}(\rho)\big)$ (and hence $\rho$)  by the data:
\begin{equation*}
  \Big\{(-k_{2,\sigma}, -k_{1,\sigma}\big)_{\sigma\in \Sigma_{\wp}};\  \alpha,q\alpha;\ \{\cL_{\sigma}\}_{\sigma\in \Sigma_{\wp}}\Big\}.
\end{equation*}
Note that by the hypothesis \ref{hyp: clin-aq0}, $\rho$ admits a unique triangulation given by
\begin{equation*}
0 \ra \cR_E\Big(\unr(\alpha)\prod_{\sigma\in \Sigma_{\wp}} \sigma^{-k_{1,\sigma}}\Big) \ra D_{\rig}(\rho) \ra \cR_E\Big(\unr(q\alpha) \prod_{\sigma\in \Sigma_{\wp}} \sigma^{-k_{2,\sigma}}\Big) \ra 0,
\end{equation*}
in particular, $\rho$ is non-critical.
\subsection{$\cL$-invariants and locally $\Q_p$-analytic representations}\label{sec: clin-4.2} Keep the above notation, following \cite{Sch10}, one can associate to $\rho$ a locally $\Q_p$-analytic representation of $\GL_2(F_{\wp})$. We recall the construction (note that the log maps that we use are slightly different from those in \cite{Sch10}) and introduce some notations.  Let $w\in \Z$, $k_{\sigma}\in \Z_{\geq 2}$ for all $\sigma\in \Sigma_{\wp}$ such that $k_{\sigma}\equiv w \pmod{2}$, suppose $\rho$ is of Hodge-Tate weights $\big(-\frac{w+k_{\sigma}}{2}, -\frac{w-k_{\sigma}+2}{2}\big)_{\sigma\in \Sigma_{\wp}}$. Put
\begin{equation*}
  \chi(\ul{k}_{\Sigma_{\wp}},w; \alpha):= \unr(\alpha) \prod_{\sigma\in \Sigma_{\wp}} \sigma^{-\frac{w-k_{\sigma}+2}{2}} \otimes \unr(\alpha) \prod_{\sigma \in \Sigma_{\wp}}\sigma^{-\frac{w+k_{\sigma}-2}{2}} ,
\end{equation*}
which is a continuous character of $T(F_{\wp})$ over $E$. Consider the parabolic induction $$\Big(\Ind_{\overline{B}(F_{\wp})}^{\GL_2(F_{\wp})} \chi(\ul{k}_{\Sigma_{\wp}},w;\alpha)\Big)^{\Q_p-\an},$$ we have the following facts
\begin{itemize}
  \item the  unique finite dimensional subrepresentation of $\Big(\Ind_{\overline{B}(F_{\wp})}^{\GL_2(F_{\wp})} \chi(\ul{k}_{\Sigma_{\wp}},w;\alpha)\Big)^{\Q_p-\an}$ is $V(\ul{k}_{\Sigma_{\wp}},w;\alpha):=\big(\unr(\alpha)\circ \dett\big)\otimes_E W(\ul{k}_{\Sigma_{\wp}},w)^{\vee}$;
      \item the maximal locally algebraic subrepresentation of the quotient
      \begin{equation*}
        \Sigma\big(\ul{k}_{\Sigma_{\wp}},w;\alpha\big):=\Big(\Ind_{\overline{B}(F_{\wp})}^{\GL_2(F_{\wp})} \chi(\ul{k}_{\Sigma_{\wp}},w;\alpha)\Big)^{\Q_p-\an}\Big/V(\ul{k}_{\Sigma_{\wp}},w;\alpha)
      \end{equation*}
      is $\St(\ul{k}_{\Sigma_{\wp}},w;\alpha):=\St \otimes_E V(\ul{k}_{\Sigma_{\wp}},w;\alpha)$, which is also the socle of $\Sigma\big(\ul{k}_{\Sigma_{\wp}},w;\alpha\big)$ (where $\St$ denotes the Steinberg representation).
\end{itemize}

Let $\psi(\ul{\cL}_{\Sigma_{\wp}})$ be the following $(d+1)$-dimensional representation of $T(F_{\wp})$ over $E$
\begin{equation*}
    \psi(\ul{\cL}_{\Sigma_{\wp}})\begin{pmatrix}
      a & 0 \\ 0 & d
    \end{pmatrix} =\begin{pmatrix}
    1& \log_{\sigma_1, -\cL_{\sigma_1}}(ad^{-1}) & \log_{\sigma_2, -\cL_{\sigma_2}}(ad^{-1}) & \cdots &\log_{\sigma_{d}, -\cL_{\sigma_{d}}}(ad^{-1}) \\ 0 & 1 & 0 & \cdots & 0 \\ 0 & 0 & 1 &\cdots & 0\\
    \vdots & \vdots &\vdots &\ddots & \vdots \\
    0 & 0 & 0 & \cdots & 1
  \end{pmatrix}.
\end{equation*}
One gets thus an exact sequence of locally $\Q_p$-analytic representations of $\GL_2(F_{\wp})$:
\begin{multline*}
  0 \lra \Big(\Ind_{\overline{B}(F_{\wp})}^{\GL_2(F_{\wp})} \chi(\ul{k}_{\Sigma_{\wp}},w;\alpha)\Big)^{\Q_p-\an} \lra \Big(\Ind_{\overline{B}(F_{\wp})}^{\GL_2(F_{\wp})} \chi(\ul{k}_{\Sigma_{\wp}},w;\alpha)\otimes_E \psi(\ul{\cL}_{\Sigma_{\wp}})\Big)^{\Q_p-\an}\\  \xlongrightarrow{s}\Big(\Big(\Ind_{\overline{B}(F_{\wp})}^{\GL_2(F_{\wp})} \chi(\ul{k}_{\Sigma_{\wp}},w;\alpha)\Big)^{\Q_p-\an}\Big)^{\oplus d} \lra 0.
\end{multline*}
Following Schraen \cite[\S 4.2]{Sch10}, put
\begin{equation}\label{equ: clin-akp}\Sigma\big(\ul{k}_{\Sigma_{\wp}},w;\alpha;\ul{\cL}_{\Sigma_{\wp}}\big):=s^{-1}\big(V(\ul{k}_{\Sigma_{\wp}},w;\alpha)^{\oplus d}\big)/V(\ul{k}_{\Sigma_{\wp}},w;\alpha).
\end{equation}
\begin{remark}\label{rem: clin-tagia} (1) By \cite[Prop.4.13]{Sch10}, $\Sigma\big(\ul{k'}_{\Sigma_{\wp}},w'; \alpha';\ul{\cL'}_{\Sigma_{\wp}}\big) \cong \Sigma\big(\ul{k}_{\Sigma_{\wp}},w;\alpha;\ul{\cL}_{\Sigma_{\wp}}\big)$ if and only if $\ul{k'}_{\Sigma_{\wp}}=\ul{k}_{\Sigma_{\wp}}$, $w'=w$, $\alpha'=\alpha$ and $\ul{\cL'}_{\Sigma_{\wp}}=\ul{\cL}_{\Sigma_{\wp}}$.

 (2) For $\sigma\in \Sigma_{\wp}$, denote by $\psi(\cL_{\sigma})$ the following $2$-dimensional representation of $T(F_{\wp})$:
\begin{equation*}
    \psi(\cL_{\sigma})\begin{pmatrix}
      a & 0 \\ 0 & d
    \end{pmatrix} =\begin{pmatrix}
    1& \log_{\sigma, -\cL_{\sigma}}(ad^{-1})  \\ 0 & 1
  \end{pmatrix}.
\end{equation*}
One has thus an exact sequence
\begin{multline*}
  0 \lra \Big(\Ind_{\overline{B}(F_{\wp})}^{\GL_2(F_{\wp})} \chi(\ul{k}_{\Sigma_{\wp}},w;\alpha)\Big)^{\Q_p-\an} \lra \Big(\Ind_{\overline{B}(F_{\wp})}^{\GL_2(F_{\wp})} \chi(\ul{k}_{\Sigma_{\wp}},w;\alpha)\otimes_E \psi(\cL_{\sigma})\Big)^{\Q_p-\an}\\  \xlongrightarrow{s_{\sigma}}\Big(\Ind_{\overline{B}(F_{\wp})}^{\GL_2(F_{\wp})} \chi(\ul{k}_{\Sigma_{\wp}},w;\alpha)\Big)^{\Q_p-\an} \lra 0.
\end{multline*}
Put $\Sigma\big(\ul{k}_{\Sigma_{\wp}},w;\alpha;\cL_\sigma\big):=s_{\sigma}^{-1}\big(V(\ul{k}_{\Sigma_{\wp}},w;\alpha)\big)/V(\ul{k}_{\Sigma_{\wp}},w;\alpha)$. One has an isomorphism of locally $\Q_p$-analytic representations of $\GL_2(F_{\wp})$:
\begin{multline}\label{equ: clin-al1h}
  \Sigma\big(\ul{k}_{\Sigma_{\wp}},w;\alpha;\cL_{\sigma_1}\big)\oplus_{\Sigma(\ul{k}_{\Sigma_{\wp}},w;\alpha)} \Sigma\big(\ul{k}_{\Sigma_{\wp}},w;\alpha;\cL_{\sigma_2}\big)\oplus_{\Sigma(\ul{k}_{\Sigma_{\wp}},w;\alpha)}\\ \cdots \oplus_{\Sigma(\ul{k}_{\Sigma_{\wp}},w;\alpha)}\Sigma\big(\ul{k}_{\Sigma_{\wp}},w;\alpha;\cL_{\sigma_{d}}\big) \xlongrightarrow{\sim} \Sigma\big(\ul{k}_{\Sigma_{\wp}},w;\alpha;\ul{\cL}_{\Sigma_{\wp}}\big).
\end{multline}
\end{remark}
Let $\chi_i$ be a locally $\sigma_i$-analytic (additive)  character of $F_{\wp}^{\times}$ in $E$, replacing the term $\log_{\sigma_i, -\cL_{\sigma_i}}(ad^{-1})$ by  $\log_{\sigma_i, -\cL_{\sigma_i}}(ad^{-1}) + \chi_i \circ \dett $, one can construct a representation $\Sigma'\big(\ul{k}_{\Sigma_{\wp}},w;\alpha;\ul{\cL}_{\Sigma_{\wp}}\big)$ exactly the same way as $\Sigma\big(\ul{k}_{\Sigma_{\wp}},w;\alpha;\ul{\cL}_{\Sigma_{\wp}}\big)$. By cohomology arguments as in \cite[\S 4.3]{Sch10}, one can actually  prove
\begin{lemma}\label{lem: linv-nmfyc}One has an isomorphism of locally $\Q_p$-analytic representations of $\GL_2(F_{\wp})$:
\begin{equation}\label{equ: clin-agia}
\Sigma'\big(\ul{k}_{\Sigma_{\wp}},w;\alpha;\ul{\cL}_{\Sigma_{\wp}}\big)\xlongrightarrow{\sim}\Sigma\big(\ul{k}_{\Sigma_{\wp}},w;\alpha;\ul{\cL}_{\Sigma_{\wp}}
\big).\end{equation}
\end{lemma}
\begin{proof}We use $\Ext^1$ to denote the extensions in the category of admissible locally $\Q_p$-analytic representations.
By the same argument as in \cite[\S 4.3]{Sch10}, replacing $\overline{G}$, $\overline{T}$ by $\GL_2(F_{\wp})$, $T(F_{\wp})$ respectively, one has (see in particular \cite[Lem.4.8]{Sch10} and the discussion which follows)
  \begin{equation*}\Ext^1_{\GL_2(F_{\wp})}\Big(V(\ul{k}_{\Sigma_{\wp}},w;\alpha),  \big(\Ind_{\overline{B}(F_{\wp})}^{\GL_2(F_{\wp})} \chi(\ul{k}_{\Sigma_{\wp}},w;\alpha)\big)^{\Q_p-\an}\Big)\cong \Hom_{\Q_p-\an}(T(F_{\wp}),E),\end{equation*}
  which is hence of dimension $2(d+1)$ over $E$.

On the other hand, one can prove
\begin{equation}\label{equ: linv-2pwv1}\Ext^1_{\GL_2(F_{\wp})}\Big(V(\ul{k}_{\Sigma_{\wp}},w;\alpha),V(\ul{k}_{\Sigma_{\wp}},w;\alpha)\Big)\cong \Hom_{\Q_p-\an}(F_{\wp}^{\times},E).\end{equation}
Indeed, put $V:=V(\ul{k}_{\Sigma_{\wp}},w;\alpha)$ for simplicity, then by \cite[Prop.3.5]{Sch10}, one has
\begin{equation*}
  \Ext^1_{\GL_2(F_{\wp})}(V,V)\cong H^1_{\an}\big(\GL_2(F_{\wp}), V\otimes_E V^{\vee}\big);
\end{equation*}
for any finite dimensional algebraic representation $W$ of $\Res_{L/\Q_p} \GL_2$ over $E$, by \cite[Thm.3]{CW}, one has 
\begin{equation*}
  H^1_{\an}(\GL_2(F_{\wp}), W)\cong H^1(\ug \otimes_{\Q_p} E, W).
\end{equation*}
Using K\"unneth formula \big(with respect to the decomposition $\ug \cong \fs \times \fz$, where $\fs$ denotes the Lie algebra of $\SL_2(F_{\wp})$ and $\fz$ the Lie algebra of the center $Z(F_{\wp})$ of $\GL_2(F_{\wp})$\big) and the first Whitehead lemma (cf. \cite[Cor.7.8.10]{Wei}), one can show $H^1(\ug\otimes_{\Q_p} E,W)=0$, if $W$ is irreducible non-trivial; and $H^1(\ug\otimes_{\Q_p} E,E)\cong H^1(\fz\otimes_{\Q_p} E, E)\cong H^1_{\an}(F_{\wp}^{\times}, E)\cong \Hom_{\Q_p-\an}(F_{\wp}^{\times},E)$. Since the trivial representation has multiplicity one in $V\otimes_E V^{\vee}$, one gets the isomorphism in (\ref{equ: linv-2pwv1}).

From the exact sequence
  \begin{equation*}
    0 \ra V(\ul{k}_{\Sigma_{\wp}},w;\alpha) \ra \big(\Ind_{\overline{B}(F_{\wp})}^{\GL_2(F_{\wp})} \chi(\ul{k}_{\Sigma_{\wp}},w;\alpha)\big)^{\Q_p-\an} \ra \Sigma(\ul{k}_{\Sigma_{\wp}},w;\alpha) \ra 0
\end{equation*}
one gets\begin{multline*}
  0 \lra \Ext^{1}_{\GL_2(F_{\wp})}\big(V(\ul{k}_{\Sigma_{\wp}},w;\alpha) , V(\ul{k}_{\Sigma_{\wp}},w;\alpha) \big) \\ \lra \Ext^1_{\GL_2(F_{\wp})}\Big(V(\ul{k}_{\Sigma_{\wp}},w;\alpha),\big(\Ind_{\overline{B}(F_{\wp})}^{\GL_2(F_{\wp})} \chi(\ul{k}_{\Sigma_{\wp}},w;\alpha)\big)^{\Q_p-\an}\Big) \\
  \xlongrightarrow{j} \Ext^1_{\GL_2(F_{\wp})}\big(V(\ul{k}_{\Sigma_{\wp}},w;\alpha),\Sigma(\ul{k}_{\Sigma_{\wp}},w;\alpha)\big).
\end{multline*} So $\dim_E \Ima(j)=d+1$. This, combined with the discussion above \cite[Prop.4.10]{Sch10}, shows that the natural injection (cf. \cite[Prop.3.5]{Sch10}, where $\PGL_2:=\GL_2/Z$)
\begin{equation*}
  \Ext_{\PGL_2(F_{\wp})}^1\big(V(\ul{k}_{\Sigma_{\wp}},w;\alpha),\Sigma(\ul{k}_{\Sigma_{\wp}},w;\alpha)\big) \hooklongrightarrow  \Ext^1_{\GL_2(F_{\wp})}\big(V(\ul{k}_{\Sigma_{\wp}},w;\alpha),\Sigma(\ul{k}_{\Sigma_{\wp}},w;\alpha)\big)
\end{equation*}
induces a bijection between $\Ext_{\PGL_2(F_{\wp})}^1\big(V(\ul{k}_{\Sigma_{\wp}},w;\alpha),\Sigma(\ul{k}_{\Sigma_{\wp}},w;\alpha)\big)$ and $\Ima(j)$, from which the isomorphism (\ref{equ: clin-agia}) follows.
\end{proof}
\subsection{Local-global compatibility}Let $w\in \Z$, $k_{\sigma}\in \Z_{\geq 2}$, $k_{\sigma}\equiv w \pmod{2}$ for all $\sigma\in \Sigma_{\wp}$. Let $\rho$ be a $2$-dimensional continuous representation of $\Gal(\overline{F}/F)$ over $E$ such that
\begin{enumerate}  \item $\rho$ is absolutely irreducible modulo $\varpi_E$;
  \item $\rho_{\wp}:=\rho|_{\Gal(\overline{\Q_p}/F_{\wp})}$ is semi-stable non-crystalline of Hodge-Tate weights $\big(-\frac{w+k_{\sigma}}{2}, -\frac{w-k_{\sigma}+2}{2}\big)_{\sigma\in \Sigma_{\wp}}$ satisfying the Hypothesis \ref{hyp: clin-aq0} with $\{\cL_{\sigma}\}_{\sigma\in \Sigma_{\wp}}$ the associated Fontaine-Mazur $\cL$-invariants and $\{\alpha,q\alpha\}$ the eigenvalues of $\varphi^{d_{0}}$ over $D_{\st}(\rho_{\wp})$
\item $\Hom_{\Gal(\overline{F}/F)}\Big(\rho, H^1_{\et}\big(K^{p}, W(\ul{k}_{\Sigma_{\wp}},w)\big)\Big)\neq 0$.
\end{enumerate}
Denote by $\lambda_{\rho}$ the system of eigenvalues of $\cH^p$ associated to $\rho$ (via the Eichler-Shimura relations), put
\begin{equation*}
  \widehat{\Pi}(\rho):=\Hom_{\Gal(\overline{F}/F)}\Big(\rho, \widetilde{H}^1_{\et}(K^p, E)^{\cH^p=\lambda_{\rho}}\Big).
\end{equation*}
Note that one has
\begin{equation*}
  \widehat{\Pi}(\rho)\xlongrightarrow{\sim} \Hom_{\Gal(\overline{F}/F)}\Big(\rho, \widetilde{H}^1_{\et}(K^p, E)_{\overline{\rho}}^{\cH^p=\lambda_{\rho}}\Big).
\end{equation*}
One can deduce from the isomorphism (cf. Thm.\ref{thm: clin-ecs}(2))
\begin{equation*}
  \widetilde{H}^1_{\et}\big(K^p,W(\ul{k}_{\Sigma_{\wp}},w)\big)_{\Q_p-\an} \xlongrightarrow{\sim} \widetilde{H}^1_{\et}(K^p, E)_{\Q_p-\an} \otimes_E W(\ul{k}_{\Sigma_{\wp}},w)
\end{equation*}
a natural injection (cf. Prop.\ref{prop: clin-trn} and \cite[Prop.4.2.4]{Em04})
\begin{equation}\label{equ: clin-apwgi}
  H^1_{\et}\big(K^p, W(\ul{k}_{\Sigma_{\wp}},w)\big)_{\overline{\rho},\Q_p-\an} \otimes_E W(\ul{k}_{\Sigma_{\wp}},w)^{\vee} \hooklongrightarrow \widetilde{H}^1_{\et}(K^p,  E)_{\overline{\rho},\Q_p-\an},
\end{equation}
thus $\widehat{\Pi}(\rho)$ is non-zero \big(by the condition (3)\big). Moreover, by Thm.\ref{thm: clin-ecs}(1), $\widehat{\Pi}(\rho)$ is a unitary admissible Banach representation of $\GL_2(F_{\wp})$ over $E$. In fact, $\widehat{\Pi}(\rho)$ is supposed to be the right representation of $\GL_2(F_{\wp})$ corresponding to $\rho_{\wp}$ in the $p$-adic Langlands program (cf. \cite{Br0}). By the local-global compatibility in the classical local Langlands correspondence for $\ell=p$, and Prop.\ref{prop: clin-trn} (see also \cite[Thm.5.3]{New}), one can show that there exists $r\in \Z_{\geq 1}$, such that (cf. \S \ref{sec: clin-4.2})
\begin{equation}\label{equ: clin-wrs}
 \St(\ul{k}_{\Sigma_{\wp}},w;\alpha)^{\oplus r}\xlongrightarrow{\sim} \widehat{\Pi}(\rho)_{\lalg},
\end{equation}where $\widehat{\Pi}(\rho)_{\lalg}$ denotes the locally algebraic vectors of $\widehat{\Pi}(\rho)$.We can now announce the main result of this article.
\begin{theorem}\label{thm: clin-sio}Keep the above notation and hypothesis,
  the natural restriction map
  \begin{equation*}
    \Hom_{\GL_2(F_{\wp})}\Big(\Sigma\big(\ul{k}_{\Sigma_{\wp}},w;\alpha;\cL_{\tau}\big),\widehat{\Pi}(\rho)_{\Q_p-\an}\Big) \lra \Hom_{\GL_2(F_{\wp})}\Big( \St(\ul{k}_{\Sigma_{\wp}},w;\alpha), \widehat{\Pi}(\rho)_{\Q_p-\an}\Big)
  \end{equation*}is bijective. In particular, one has a continuous injection of $\GL_2(F_{\wp})$-representations
  \begin{equation*}
 \Sigma\big(\ul{k}_{\Sigma_{\wp}},w;\alpha;\cL_{\tau}\big)^{\oplus r} \hooklongrightarrow  \widehat{\Pi}(\rho)_{\Q_p-\an}.
  \end{equation*}
  which induces an isomorphism between the locally algebraic subrepresentations.
\end{theorem}Such a result is called local-global compatibility, since  the $\Pi(\rho_{\wp})$ are constructed by the local parameters (i.e. parameters of $\rho_{\wp}$) while $\widehat{\Pi}(\rho)_{\Q_p-\an}$ is a global object. By the isomorphism (\ref{equ: clin-al1h}), the theorem \ref{thm: clin-sio} would follow from the following proposition.
\begin{proposition}\label{prop: clin-igla}
For any $\tau\in \Sigma_{\wp}$, the restriction map
  \begin{equation*}
    \Hom_{\GL_2(F_{\wp})}\Big(\Sigma\big(\ul{k}_{\Sigma_{\wp}},w;\alpha;\cL_{\tau}\big),\widehat{\Pi}(\rho)_{\Q_p-\an}\Big) \lra \Hom_{\GL_2(F_{\wp})}\Big( \St(\ul{k}_{\Sigma_{\wp}},w;\alpha), \widehat{\Pi}(\rho)_{\Q_p-\an}\Big)
  \end{equation*}is bijective.
\end{proposition}
Before proving this proposition, we give a corollary on the uniqueness of $\cL$-invariants (suggested by Breuil):
\begin{corollary}\label{cor: clin-ape}Keep the notation in Prop.\ref{prop: clin-igla}, let $\cL'_{\tau}\in E$, if there exists a continuous injection of $\GL_2(F_{\wp})$-representations
  \begin{equation*}
   i: \Sigma\big(\ul{k}_{\Sigma_{\wp}},w;\alpha;\cL_{\tau}'\big)\hooklongrightarrow \widehat{\Pi}(\rho)_{\Q_p-\an},
  \end{equation*}
  then $\cL'_{\tau}=\cL_{\tau}$.
\end{corollary}
\begin{proof}By Prop.\ref{prop: clin-igla}, the restriction on $i$ to $\St\big(\ul{k}_{\Sigma_{\wp}},w;\alpha\big)$ gives rise to a continuous injection
\begin{equation*}
  j:  \Sigma\big(\ul{k}_{\Sigma_{\wp}},w;\alpha;\cL_{\tau}\big)\hooklongrightarrow \widehat{\Pi}(\rho)_{\Q_p-\an}.
\end{equation*}
Suppose $\cL'_{\tau}\neq \cL_{\tau}$, one can thus  deduce from $i$ and $j$ an injection
\begin{equation}\label{equ: clin-apwap}
\Sigma\big(\ul{k}_{\Sigma_{\wp}},w;\alpha;\cL_{\tau}'\big) \oplus_{\Sigma(\ul{k}_{\Sigma_{\wp}},w;\alpha)}\Sigma\big(\ul{k}_{\Sigma_{\wp}},w;\alpha;\cL_{\tau}\big)\hooklongrightarrow \widehat{\Pi}(\rho)_{\Q_p-\an}.
\end{equation}
 Put for simplicity \begin{equation*}V:=\Sigma\big(\ul{k}_{\Sigma_{\wp}},w;\alpha;\cL_{\tau}'\big) \oplus_{\Sigma(\ul{k}_{\Sigma_{\wp}},w;\alpha)}\Sigma\big(\ul{k}_{\Sigma_{\wp}},w;\alpha;\cL_{\tau}\big),\end{equation*}
the key point is the locally algebraic subrepresentation $V_{\lalg}$ contains an extension of $V(\ul{k}_{\Sigma_{\wp}},w;\alpha)$ by $\St\big(\ul{k}_{\Sigma_{\wp}},w;\alpha\big)$, which would contradict to (\ref{equ: clin-wrs}):

Denote by  $\psi(\cL_{\tau}',\cL_{\tau})$ the following $3$-dimensional representation of $T(F_{\wp})$:
\begin{equation*}
  \psi(\cL_{\tau}',\cL_{\tau})\begin{pmatrix} a & 0 \\ 0 & d \end{pmatrix}=\begin{pmatrix} 1 & \log_{\tau,-\cL_{\tau}'}(ad^{-1}) & \log_{\tau, -\cL_{\tau}}(ad^{-1}) \\ 0 & 1 & 0 \\ 0 & 0 & 1 \end{pmatrix},
\end{equation*}
thus one has an exact sequence
\begin{multline*}
  0 \lra \Big(\Ind_{\overline{B}(F_{\wp})}^{\GL_2(F_{\wp})} \chi(\ul{k}_{\Sigma_{\wp}},w;\alpha)\Big)^{\Q_p-\an} \lra \Big(\Ind_{\overline{B}(F_{\wp})}^{\GL_2(F_{\wp})} \chi(\ul{k}_{\Sigma_{\wp}},w;\alpha)\otimes_E\psi(\cL_{\tau}',\cL_{\tau}))\Big)^{\Q_p-\an}\\  \xlongrightarrow{s'}\Big(\Big(\Ind_{\overline{B}(F_{\wp})}^{\GL_2(F_{\wp})} \chi(\ul{k}_{\Sigma_{\wp}},w;\alpha)\Big)^{\Q_p-\an}\Big)^{\oplus 2} \lra 0.
\end{multline*}
It's straightforward to see \begin{equation*}V\xlongrightarrow{\sim} (s')^{-1}(V(\ul{k}_{\Sigma_{\wp}},w;\alpha)^{\oplus 2})/V(\ul{k}_{\Sigma_{\wp}},w;\alpha).\end{equation*}
On the other hand, $\psi(\cL_{\tau}',\cL_{\tau})$ admits a smooth subrepresentation
\begin{equation*}
  \psi_0(\cL_{\tau}',\cL_{\tau})\begin{pmatrix} a & 0 \\ 0 & d \end{pmatrix}=\begin{pmatrix}1 & (\cL'_\tau-\cL_{\tau})\us_{\wp}(ad^{-1}) \\ 0 & 1\end{pmatrix},
\end{equation*}
one has thus an exact sequence of smooth representations of $\GL_2(F_{\wp})$
\begin{equation*}
  0 \lra \Big(\Ind_{\overline{B}(F_{\wp})}^{\GL_2(F_{\wp})} \chi_{\alpha}\Big)^{\infty} \lra \Big(\Ind_{\overline{B}(F_{\wp})}^{\GL_2(F_{\wp})} \chi_{\alpha}\otimes_E \psi_0(\cL_{\tau}',\cL_{\tau})\Big)^{\infty}  \xlongrightarrow{s''}\Big(\Ind_{\overline{B}(F_{\wp})}^{\GL_2(F_{\wp})} \chi_{\alpha}\Big)^{\infty} \lra 0,
\end{equation*}
where $\chi_{\alpha}:=\unr(\alpha) \otimes \unr(\alpha)$. Note that $\unr(\alpha)\circ \dett$ is the socle of $\big(\Ind_{\overline{B}(F_{\wp})}^{\GL_2(F_{\wp})} \chi_{\alpha}\big)^{\infty}$, and one can check
$$V':=\big((s'')^{-1}\big(\unr(\alpha)\circ \dett\big)/\unr(\alpha)\circ \dett\big) \otimes_E W(\ul{k}_{\Sigma_{\wp}},w)^{\vee}$$
is a locally algebraic subrepresentation of $V$, which is an extension (non-split) of $V(\ul{k}_{\Sigma_{\wp}},w;\alpha)$ by $\St\big(\ul{k}_{\Sigma_{\wp}},w;\alpha\big)$. We deduce from (\ref{equ: clin-apwap}) an injection $V'\hookrightarrow \widehat{\Pi}(\rho)_{\lalg}$, a contradiction with (\ref{equ: clin-wrs}).
\end{proof}
\begin{remark}
By  Thm.\ref{thm: clin-sio} and Cor.\ref{cor: clin-ape}, we see that the local Galois representation $\rho_{\wp}$ can be determined by $\widehat{\Pi}(\rho)$.
\end{remark}
 The following lemma has a straightforward proof that is omitted.
\begin{lemma}\label{lem: clin-nfo}
  Let $V$ be an admissible locally $\Q_p$-analytic representation of $\GL_2(F_{\wp})$ over $E$, there exists a natural bijection
  \begin{equation*}
    \Hom_{\GL_2(F_{\wp})}\big(V, \widehat{\Pi}(\rho)_{\Q_p-\an}\big) \xlongrightarrow{\sim}\\ \Hom_{\Gal(\overline{F}/F)}\Big(\rho, \Hom_{\GL_2(F_{\wp})}\big(V, \widetilde{H}^1_{\et}(K^p,E)_{\Q_p-\an}^{\cH^p=\lambda_{\rho}}\big)\Big).
  \end{equation*}
\end{lemma}
The Prop.\ref{prop: clin-igla} thus follows from
\begin{proposition}\label{prop: clin-awdl}With the notation in Prop.\ref{prop: clin-igla}, the restriction map
  \begin{multline}\label{equ: clin-gpwp}
    \Hom_{\GL_2(F_{\wp})}\Big(\Sigma\big(\ul{k}_{\Sigma_{\wp}},w;\alpha;\cL_{\tau}\big),\widetilde{H}^1_{\et}(K^p,E)_{\Q_p-\an}^{\cH^p=\lambda_{\rho}}\Big) \\
    \lra \Hom_{\GL_2(F_{\wp})}\Big(\St\big(\ul{k}_{\Sigma_{\wp}},w;\alpha\big), \widetilde{H}^1_{\et}(K^p,E)_{\Q_p-\an}^{\cH^p=\lambda_{\rho}}\Big)
  \end{multline}is bijective.
\end{proposition}
The rest of this paper is devoted to the proof of Prop.\ref{prop: clin-awdl}. Given an injection (whose existence follows from (\ref{equ: clin-wrs})) $\St\big(\ul{k}_{\Sigma_{\wp}},w;\alpha\big) \hookrightarrow \widetilde{H}^1_{\et}(K^p,E)_{\Q_p-\an}^{\cH^p=\lambda_{\rho}}$, by applying the Jacquet-Emerton functor and Thm.\ref{thm: clin-cjw}, one gets a closed $E$-point (associate to $\rho$) in $\cV(K^p,w)_{\overline{\rho}}$ given by $$z:=\big(\chi:= \chi(\ul{k}_{\Sigma_{\wp}},w;\alpha) \delta,\lambda_{\rho}\big).$$
\begin{lemma}\label{lem: clin-pvle}
The  restriction map (\ref{equ: clin-gpwp})
is injective.
\end{lemma}
\begin{proof}
The proof is the same as in \cite[Prop.6.3.9]{Ding}. Let $f$ be in the kernel of (\ref{equ: clin-gpwp}), suppose $f\neq 0$, thus $f$ would induce an injection $$V_{\wp}\hooklongrightarrow\widetilde{H}^1_{\et}(K^p,E)_{\Q_p-\an}^{\cH^p=\lambda_{\rho}}$$ with $V_{\wp}$  an irreducible constituent of $\Sigma\big(\ul{k}_{\Sigma_{\wp}},w;\alpha;\cL_{\tau}\big)$ different from $\St\big(\ul{k}_{\Sigma_{\wp}},w;\alpha\big)$ (since $f$ lies in the kernel of (\ref{equ: clin-gpwp})) and from $V(\ul{k}_{\Sigma_{\wp}},w;\alpha)$ (by (\ref{equ: clin-wrs})), from which, by applying the Jacquet-Emerton functor, one would get a companion point of $z$, which contradicts to the fact that $z$ is non-critical (thus does not admit companion points, cf. Cor.\ref{cor: clin-aat}).
\end{proof}
In the following, we prove the surjectivity of (\ref{equ: clin-gpwp}), which is the key of this paper.
By assumption, we know the point $z$ is non-critical, thus one may find an open neighborhood $\cU$ of $z$ in $\cV(K^p,w)_{\overline{\rho}}$ such that (cf. Cor.\ref{cor: clin-siz} and \cite[Lem.6.3.12]{Ding})
  \begin{enumerate}
    \item $\cU$ is strictly quasi-Stein (\cite[Def.2.1.17(iv)]{Em04});
    \item for any $z'\in \cU(\overline{E})$, $z'$ does not have companion points.
  \end{enumerate}
  Denote by $\cM:=\cM(K^p,w)_{\overline{\rho}}$ for simplicity, the natural restriction (with dense image since $\cU$ is strictly quasi-Stein)
  \begin{equation*}
  J_B\Big(\widetilde{H}^1_{\et}(K^p,E)^{Z_1=\sN^{-w}}_{\overline{\rho}, \Q_p-\an}\Big)^{\vee}_b \cong \cM(\cV(K^p,w)_{\overline{\rho}})\lra \cM(\cU)
  \end{equation*}
  induces a continuous injection of locally $\Q_p$-analytic representations of $T$ (invariant under $\cH^p$)
  \begin{equation}\label{equ: clin-wbi}
    \cM(\cU)_{b}^{\vee} \hooklongrightarrow   J_B\Big(\widetilde{H}^1_{\et}(K^p,E)^{Z_1=\sN^{-w}}_{\overline{\rho}, \Q_p-\an}\Big)
  \end{equation}
where $\cM(\cU)_b^{\vee}$ denotes the strict dual of $\cM(\cU)$.  As in \cite[Lem.6.3.13, 6.3.14]{Ding}, one can show \big(see \cite[Lem.4.5.12]{Em1} and the proof of \cite[Thm.4.5.7]{Em1} for $\GL_2(\Q_p)$-case\big)
  \begin{itemize}
    \item $\cM(\cU)_b^{\vee}$ is an allowable subrepresentation of $J_B\big(\widetilde{H}_{\et}^1(K^p,E)_{\Q_p-\an}^{Z_1=\sN^{-w}}\big)$ (cf. \cite[Def.0.11]{Em2}) (this follows from the fact $\cU$ is strictly quasi-Stein);
    \item the map (\ref{equ: clin-wbi}) is balanced (cf. \cite[Def.0.8]{Em2}) (this follows from the fact any closed point in $\cU$ does not have companion points);
    \item $\cM(\cU)$ is a torsion free $\co(\cW_{1,\Sigma_{\wp}})$-module (cf. \S \ref{sec: clin-3.3}).
  \end{itemize}Denote by $\overline{\ft}\subset \ft$ the Lie algebra of $Z_1'\subset T(F_{\wp})$, by \cite[Cor.5.3.31]{Ding} \Big(note that \cite[Cor.5.3.31]{Ding} still holds with $\ft$ replaced by $\overline{\ft}$, and note that $\cM(\cU)^{\vee}_b$ is a divisible $\text{U}(\overline{\ft}_{\Sigma_{\wp}})$-module since $\cM(\cU)$ is a torsion free $\text{U}(\overline{\ft}_{\Sigma_{\wp}})$-module with $\text{U}(\overline{\ft}_{\Sigma_{\wp}})\hookrightarrow \co(\cW_{1,\Sigma_{\wp}})$, e.g. see \cite[\S 5.1.3, Prop.5.1.12]{Ding}\Big), the map (\ref{equ: clin-wbi}) induces a continuous $\GL_2(F_{\wp})\times\cH^p$-invariant map \big(see \cite[(4.5.9)]{Em1} for $\GL_2(\Q_p)$-case\big)
  \begin{equation}\label{equ: clin-p2w}
    \big(\Ind_{\overline{B}(F_{\wp})}^{\GL_2(F_{\wp})} \cM(\cU)^{\vee}_b[\delta^{-1}]\big)^{\Q_p-\an} \lra \widetilde{H}^1_{\et}(K^p,E\big)^{Z_1=\sN^{-w}}_{\overline{\rho}, \Q_p-\an},
  \end{equation}
  where $ \cM(\cU)^{\vee}_b[\delta^{-1}]$ denotes the twist of $\cM(\cU)^{\vee}_b$ by $\delta^{-1}$. We would deduce Prop.\ref{prop: clin-awdl} from this map.

  For $\sigma\in \Sigma_{\wp}$, denote by $\widehat{T}_{\sigma}$ (resp. $\cW_{1,\sigma}$) the rigid space over $E$ parameterizing locally $\sigma$-analytic characters of $T(F_{\wp})$ (resp. $1+2\varpi\co_{\wp}$), which is hence a closed subspace of $\widehat{T}_{\Sigma_{\wp}}$ (resp. $\cW_{1,\Sigma_{\wp}}$) (e.g. see \cite[\S 5.1.4]{Ding}). 

The character $\chi$ induces a closed embedding
\begin{equation*}
  \chi: \cW_{1,\tau}\hooklongrightarrow \cW_{1,\Sigma_{\wp}}, \  \chi'\mapsto \chi|_{Z_{1}'} \chi'.
\end{equation*}
Recall that we have a natural morphism $\kappa: \cV(K^p,w)_{\overline{\rho}} \ra \cW_{1,\Sigma_{\wp}}$ which is \'etale at $z$ (cf. Thm.\ref{thm: clin-elt}). Put $\cV(K^p,w)_{\overline{\rho},\tau}:=\cV(K^p,w)_{\overline{\rho}}\times_{\cW_{1,\Sigma_{\wp}},\chi} \cW_{1,\tau}$, one has thus a Cartesian diagram
\begin{equation*}
  \begin{CD}
    \cV(K^p,w)_{\overline{\rho},\tau} @>>>\cW_{1,\tau} \\
    @VVV  @V \chi VV\\
    \cV(K^p,w)_{\overline{\rho}}  @>>> \cW_{1,\Sigma_{\wp}}
  \end{CD}
\end{equation*}
 Denote still by $z$ the preimage of $z$ in $\cV(K^p,w)_{\overline{\rho},\tau}$, $\kappa$ the natural morphism $\cV(K^p,w)_{\overline{\rho},\tau} \ra \cW_{1,\tau}$, thus $\kappa$ is \'etale at $z$. By results in \S \ref{sec: clin-3.4.3}, one can choose an open affinoid $V$ of $\cV(K^p,w)$ containing $z$ such that $V$ is \'etale over $\cW_{1,\tau}$, and that any point in $V$ does not have companion points. Denote by $V_{\tau}$  the preimage of $V$ in $\cV(K^p,w)_{\overline{\rho},\tau}$.  We see $V_{\tau}$ is in fact a smooth curve. By Prop.\ref{prop: clin-kzc} and shrinking $V$ (and hence $V_{\tau}$) if necessary, one gets a continuous representation
 $$\rho_{V_{\tau}}: \Gal(\overline{F}/F) \ra \GL_2(\co(V))\ra \GL_2(\co(V_{\tau})).$$
Denote by $\chi_{V_{\tau}}=\chi_{V_{\tau},1}\otimes \chi_{V_{\tau},2}: T(F_{\wp})\ra \co(V_{\tau})^{\times}$ the character induced by the natural morphism $V_{\tau} \ra \cV(K^p,w)_{\overline{\rho},\tau} \ra \cV(K^p,w)_{\overline{\rho}}\ra \widehat{T}_{\Sigma_{\wp}}$. By \cite[Thm.6.3.9]{KPX} applied to the smooth affinoid curve $V_{\tau}$, together with the fact that any point in $V_{\tau}$ does not have companion points, one has
\begin{lemma}\label{lem: clin-foia}
  There exists an exact sequence of $(\varphi,\Gamma)$-modules over $\cR_{\co(V_{\tau})}:=B_{\rig,F_{\wp}}^{\dagger} \widehat{\otimes}_{\Q_p} \co(V_{\tau})$ \big(see for example \cite[Thm.2.2.17]{KPX} for $D_{\rig}(\rho_{V_{\tau},\wp})$, and \cite[Const.6.2.4]{KPX} for $(\varphi,\Gamma)$-modules of rank $1$ associate to continuous character of $F_{\wp}^{\times}$ with values in $\co(V_{\tau})^{\times}$\big):
  \begin{equation}\label{equ: clin-grqp}
    0 \ra  \cR_{\co(V_{\tau})}\big(\unr(q)\chi_{V_{\tau},1}\big) \ra D_{\rig}(\rho_{V_{\tau},\wp}) \ra \cR_{\co(V_{\tau})}\big(\chi_{V_{\tau},2}\prod_{\sigma\in \Sigma_{\wp}}\sigma^{-1}\big) \ra 0.
  \end{equation}
\end{lemma}
Let $t_{\tau}: \Spec E[\epsilon]/\epsilon^2 \ra \cW_{1,\tau}$  be a \emph{non-zero} element in the tangent space of $\cW_{1,\tau}$ at the identity point (corresponding to the trivial character), since $V_{\tau}$ is \'etale over $\cW_{1,\tau}$, $t_{\tau}$ gives rise to a non-zero element, still denoted by $t_{\tau}$, in the tangent space of $\cV(K^p,w)_{\overline{\rho},\tau}$ at the point $z$. The following composition
\begin{equation}\label{equ: clin-amgi}
  t_{\tau}: \Spec E[\epsilon]/\epsilon^2 \lra \cV(K^p,w)_{\overline{\rho},\tau} \lra \cV(K^p,w)_{\overline{\rho}} \lra \widehat{T}_{\Sigma_{\wp}}
\end{equation}
gives rise to a character $\widetilde{\chi}_{\tau}: T(F_{\wp}) \ra (E[\epsilon]/\epsilon^2)^{\times}$. We have in fact $\widetilde{\chi}_{\tau}=t_\tau\circ \chi_{V_{\tau}}$ \big($t_\tau: \co(V_{\tau}) \ra E[\epsilon]/\epsilon^2$\big). We know $\widetilde{\chi}_{\tau}\equiv \chi \pmod{\epsilon}$. Since the image of (\ref{equ: clin-amgi}) lies in $\widehat{T}_{\Sigma_{\wp}}(w)$ (cf. (\ref{equ: clin-tapwe})), we see $\widetilde{\chi}_{\tau}|_{Z_1}=\sN^{-w}$ and thus $(\widetilde{\chi}_{\tau}\chi^{-1})|_{Z_1}=1$.
\begin{lemma}
   There exist $\gamma$, $\eta\in E$, $\mu\in E^{\times}$ such that $$\psi_{\tau}:=\widetilde{\chi}_{\tau}\chi^{-1}=\unr(1+\gamma\epsilon) (1-\mu\epsilon \log_{\tau,0,\varpi}) \otimes \unr(1+\eta \epsilon) (1+\mu\epsilon\log_{\tau,0,\varpi}).$$
\end{lemma}
\begin{proof}
  The lemma is straightforward. Note that $\mu\neq 0$ since $t_{\tau}$ (as an element in the tangent space) is non-zero.
\end{proof}
By multiplying $\epsilon$ by constants, we assume $\mu=1$ and thus $$\psi_{\tau}=\unr(1+\gamma\epsilon) (1-\epsilon\log_{\tau,0,\varpi})\otimes \unr(1+\eta\epsilon) (1+\epsilon\log_{\tau,0,\varpi}).$$

The following lemma, which describes the character $\widetilde{\chi}_{\tau}$ in terms of the $\cL$-invariants, is one of the key points in the proof of Prop.\ref{prop: clin-awdl}.
\begin{lemma}\label{thm: clin-aue}$(\eta-\gamma)/2=-e^{-1}\big(\cL_{\tau}+\log_{\tau}\big(\frac{p}{\varpi^{e}}\big)\big)=(-\cL_{\tau})(\varpi)$ (cf. \S \ref{sec: clin-1}).
\end{lemma}
\begin{proof}
  Denote by $\widetilde{\rho}_{z,\wp}:=t_{\tau}\circ  \rho_{V_{\tau},\wp}:\Gal(\overline{\Q_p}/F_{\wp}) \ra \GL_2(\co(V_{\tau}))\ra \GL_2(E[\epsilon]/\epsilon)$, from (\ref{equ: clin-grqp}), one gets an exact sequence of $(\varphi,\Gamma)$-modules over $\cR_{E[\epsilon]/\epsilon^2}$:
  \begin{equation}\label{equ: clin-2no}
0 \ra \cR_{E[\epsilon]/\epsilon^2} \big(\unr(q)\widetilde{\chi}_{\tau,1} \big) \ra D_{\rig}(\widetilde{\rho}_{z,\wp}) \ra \cR_{E[\epsilon]/\epsilon^2} \Big(\widetilde{\chi}_{\tau,2} \prod_{\sigma\in \Sigma_{\wp}} \sigma^{-1}\Big) \ra 0.
  \end{equation}
  For $\sigma\in \Sigma_{\wp}$, denote by  $\varepsilon_{\sigma,\varpi}$ the character of $F_{\wp}^{\times}$ with $\varepsilon_{\sigma,\varpi}|_{\co_{\wp}^{\times}}=\sigma|_{\co_{\wp}^{\times}}$ and $\varepsilon_{\sigma,\varpi}(\varpi)=1$. Recall $\chi=\unr(q^{-1}\alpha)\prod_{\sigma\in \Sigma_{\wp}}\sigma^{-\frac{w-k_{\sigma}+2}{2}} \otimes \unr(q\alpha) \prod_{\sigma\in \Sigma_{\wp}}\sigma^{-\frac{w+k_{\sigma}-2}{2}}$. We have thus
  \begin{multline*}
    \delta_1:=\unr(q)\widetilde{\chi}_{\tau,1}=\unr\big(\alpha(1+\gamma\epsilon)\big)(1-\epsilon\log_{\tau,0,\varpi})\prod_{\sigma\in \Sigma_{\wp}}\sigma^{-{\frac{w-k_{\sigma}+2}{2}}}\\ =\unr\Big(\alpha(1+\gamma\epsilon)\prod_{\sigma\in \Sigma_{\wp}}\sigma(\varpi)^{-\frac{w-k_{\sigma}+2}{2}}\Big) (1-\epsilon\log_{\tau,0,\varpi}) \prod_{\sigma\in \Sigma_{\wp}}\varepsilon_{\sigma,\varpi}^{-\frac{w-k_{\sigma}+2}{2}},
  \end{multline*}
  \begin{equation*}
    \delta_2:=\widetilde{\chi}_{\tau,2} \prod_{\sigma\in \Sigma_{\wp}} \sigma^{-1}=\unr\Big(q\alpha(1+\eta\epsilon)\prod_{\sigma\in \Sigma_{\wp}}\sigma(\varpi)^{-\frac{w+k_{\sigma}-1}{2}}\Big) (1+\epsilon\log_{\tau,0,\varpi})\prod_{\sigma\in \Sigma_{\wp}} \varepsilon_{\sigma,\varpi}^{-\frac{w+k_{\sigma}}{2}}.
  \end{equation*}
Let $\chi_0:=(1-\epsilon\log_{\tau,0,\varpi}) \prod_{\sigma\in \Sigma_{\wp}}\varepsilon_{\sigma,\varpi}^{-\frac{w-k_{\sigma}+2}{2}}$, one can view $\chi_0$ as a character of $\Gal(\overline{\Q_p}/F_{\wp})$ over $E[\epsilon]/[\epsilon^2]$ via the local Artin map  $\Art_{F_{\wp}}$. Denote by $\widetilde{\rho}:=\widetilde{\rho}_{z,\wp}\otimes_{E[\epsilon]/\epsilon^2}\chi_0^{-1}$, \begin{eqnarray*}\delta_1'&:=&\delta_1\chi_0^{-1}=\unr\Big(\alpha(1+\gamma\epsilon)\prod_{\sigma\in \Sigma_{\wp}}\sigma(\varpi)^{-\frac{w-k_{\sigma}+2}{2}}\Big) \\
  \delta_2'&:=& \delta_2\chi_0^{-1}=\unr\Big(q\alpha(1+\eta\epsilon)\prod_{\sigma\in \Sigma_{\wp}}\sigma(\varpi)^{-\frac{w+k_{\sigma}-1}{2}}\Big) (1+2\epsilon\log_{\tau,0,\varpi})\prod_{\sigma\in \Sigma_{\wp}}\varepsilon_{\sigma,\varpi}^{1-k_{\sigma}}
  \end{eqnarray*}
    Thus one has
  \begin{equation}\label{equ: linv-1drgd}
  0 \ra \cR_{E[\epsilon]/\epsilon^2} (\delta_1') \ra  D_{\rig}(\widetilde{\rho}) \ra \cR_{E[\epsilon]/\epsilon^2}(\delta_2')\ra 0.
  \end{equation}
  Denote by $\widetilde{\alpha}:=\alpha(1+\gamma \epsilon)\prod_{\sigma\in \Sigma_{\wp}} \sigma(\varpi)^{-\frac{w-k_{\sigma}+2}{2}}$, by (\ref{equ: linv-1drgd}),  $(B_{\cris}\otimes_{\Q_p} \widetilde{\rho})^{\Gal(\overline{\Q_p}/F_{\wp}), \varphi^{d_0}=\widetilde{\alpha}}$ is free of rank $1$ over $F_{\wp,0}\otimes_{\Q_p}E[\epsilon]/\epsilon^2$, and that the $E$-representation $\overline{\widetilde{\rho}} \pmod{\epsilon}$ of $\Gal(\overline{\Q_p}/F_{\wp})$ is semi-stable non-crystalline of Hodge-Tate weights $(1-k_{\sigma},0)_{\sigma\in \Sigma_{\wp}}$, and has  the same $\cL$-invariants as $\rho_{\wp}$. By applying the formula in \cite[Thm.1.1]{Zhang}, one gets
  \begin{equation*}
    \frac{\gamma}{d_0}+\big(-\frac{\gamma+\eta}{2d_0}\big)-\frac{1}{d} \log_{\tau}\big(\frac{p}{\varpi^{e}}\big)-\frac{1}{d} \cL_{\tau}=0.
  \end{equation*}
  In fact, with the notation of \emph{loc. cit.}, one has
  \begin{equation*}
    \log_{\tau,0,\varpi}=-\frac{1}{d} \log_{\tau}\big(\frac{p}{\varpi^{e}}\big) \psi_1 +1_{\tau}\psi_2
  \end{equation*}
   where $1_{\tau}\in F_{\wp}\otimes_{\Q_p} E\cong \prod_{\sigma\in \Sigma_{\wp}} E$ such that $(1_{\tau})_{\sigma}=0$ if $\sigma\neq \tau$ and $(1_{\tau})_{\tau}=1$, and where we view $\log_{\tau,0,\varpi}$ as an additive character of $\Gal(\overline{\Q_p}/F_{\wp})$ via the local Artin map $\Art_{F_{\wp}}$. Thus one can apply the formula  in \cite[Thm.1.1]{Zhang} to
   \begin{equation*}\{V, \alpha,\delta,\kappa\}=\Big\{\widetilde{\rho}, \widetilde{\alpha},\Big(-\frac{\gamma+\eta}{d_0}-\frac{2}{d} \log_{\tau}\big(\frac{p}{\varpi^{e}}\big) \Big) \epsilon,2_{\tau} \epsilon\Big\}.
   \end{equation*}
   The lemma follows.
\end{proof}
The following lemma follows directly from Lem.\ref{thm: clin-aue}.
\begin{lemma}
  As representations of $T(F_{\wp})$ (of dimension $2$) over $E$, one has
  \begin{equation*}
    \widetilde{\chi}_{\tau}\delta^{-1} \cong \chi(\ul{k}_{\Sigma_{\wp}},w;\alpha) \otimes_E \psi(\cL_{\tau})',
  \end{equation*}
  where $\psi(\cL_{\tau})'\begin{pmatrix}
    a & 0 \\ 0 & d
  \end{pmatrix}=\begin{pmatrix}
    1 & \log_{\tau,-\cL_{\tau}}(a d^{-1}) +\frac{\gamma+\eta}{2}\us_{\wp}(ad) \\ 0 & 1
  \end{pmatrix}$.
\end{lemma}
The parabolic induction $\big(\Ind_{\overline{B}(F_{\wp})}^{\GL_2(F_{\wp})} \widetilde{\chi}_{\tau}\delta^{-1}\big)^{\Q_p-\an}$ lies thus in an exact sequence as follows
  \begin{multline}\label{equ: clin-pw2p}
  0 \lra \Big(\Ind_{\overline{B}(F_{\wp})}^{\GL_2(F_{\wp})} \chi(\ul{k}_{\Sigma_{\wp}},w;\alpha)\Big)^{\Q_p-\an} \lra \Big(\Ind_{\overline{B}(F_{\wp})}^{\GL_2(F_{\wp})} \widetilde{\chi}_{\tau}\delta^{-1}\Big)^{\Q_p-\an}\\  \xlongrightarrow{s_{\tau}'}\Big(\Ind_{\overline{B}(F_{\wp})}^{\GL_2(F_{\wp})} \chi(\ul{k}_{\Sigma_{\wp}},w;\alpha)\Big)^{\Q_p-\an} \lra 0.
\end{multline}
By Lem.\ref{lem: linv-nmfyc}, one has
\begin{lemma}\label{lem: clin-pcssf}
  One has an isomorphism of locally $\Q_p$-analytic representations of $\GL_2(F_{\wp})$:
  \begin{equation*}
    (s_{\tau}')^{-1}\big(V\big(\ul{k}_{\Sigma_{\wp}},w;\alpha\big)\big)/V\big(\ul{k}_{\Sigma_{\wp}},w;\alpha\big) \cong \Sigma\big(\ul{k}_{\Sigma_{\wp}},w;\alpha;\cL_{\tau}\big).
  \end{equation*}
\end{lemma}
Consider the composition  $t_{\tau}: \Spec E[\epsilon]/\epsilon^2\xrightarrow{t_{\tau}} \cV(K^p,w)_{\overline{\rho},\tau}\hookrightarrow \cV(K^p,w)_{\overline{\rho}}$, and $(t_{\tau}^* \cM)^{\vee}$ being a finite dimensional $E$-vector space equipped with a natural action of $T(F_{\wp})\times \cH^p$. We claim there exists $n\in \Z_{\geq 1}$ such that \big(as $T(F_{\wp})$-representations\big)\begin{equation}\label{equ: clin-mcee}(t_{\tau}^*\cM)^{\vee}\cong \widetilde{\chi}_{\tau}^{\oplus n}.\end{equation}In fact, as in \S \ref{sec: clin-3.4.3}, there exists open affinoids $V'$ of $\cV(K^p,w)_{\overline{\rho}}$ and $U$ of $\cW_{1,\Sigma_{\wp}}$ such that $V'$ lies over $U$, $\co(V')\cong \co(U)$, and that $\cM(V')$ is a locally free $\co(U)$-module. The group $T(F_{\wp})$ acts on $\cM(V')$  via the character $T(F_{\wp}) \ra \co(V')^{\times} \cong \co(U)^{\times}$ \big(with the first map induced by the natural morphism $V'\ra \widehat{T}_{\Sigma_{\wp}}$\big), the claim follows. We also see that $\cH^p$ acts on $\cM(V')$ via the natural morphism $\cH^p \ra \co(V') \cong \co(U)$. Thus $\cH^p$ acts on $(t_{\tau}^*\cM)^{\vee}$ via $\cH^p \ra \co(V') \ra E[\epsilon]/\epsilon^2$, in particular $(t_{\tau}^*\cM)^{\vee}$ is a generalized $\lambda_{\rho}$-eigenspace for $\cH^p$ (one can view $t_{\tau}$ as a thickening of the point $z$).

Since $\cU$ is strictly quasi-Stein, the restriction map $\cM(\cU)\ra t_{\tau}^*\cM$ is surjective, so we have injections
 \begin{equation*}
   (z^*\cM)^{\vee} \hooklongrightarrow (t_{\tau}^*\cM)^{\vee} \hooklongrightarrow \cM(\cU)^{\vee}_b.
 \end{equation*}Firstly note  that a non-zero map $f$ in the right term of (\ref{equ: clin-gpwp}) corresponds to a non-zero vector $v\in (z^*\cM)^{\vee}$ in a natural way:
    \begin{multline}\label{equ: clin-pjpk}
    (z^* \cM)^{\vee} \xlongrightarrow{\sim} J_B\big(\widetilde{H}_{\et}^1(K^p,E)_{\Q_p-\an}^{Z_1=\sN^{-w}}\big)^{T(F_{\wp})=\chi, \cH^p=\lambda_{\rho}}
    \\ \xlongrightarrow{\sim} J_B\big(H_{\et}^1\big(K^p, W(\ul{k}_{\Sigma_{\wp}},w)\big) \otimes_E W(\ul{k}_{\Sigma_{\wp}},w)^{\vee}\big)^{T(F_{\wp})=\chi, \cH^p=\lambda_{\rho}}
    \\ \xlongrightarrow{\sim} J_B\big(H_{\et}^1\big(K^p, W(\ul{k}_{\Sigma_{\wp}},w)\big)\big)^{T(F_{\wp})=\psi, \cH^p=\lambda_{\rho}} \otimes_E \chi(\ul{k}_{\Sigma_{\wp}},w)
    \\ \xlongrightarrow{\sim} \Hom_{\GL_2(F_{\wp})}\Big(\big(\Ind_{\overline{B}(F_{\wp})}^{\GL_2(F_{\wp})} \psi\delta^{-1}\big)^{\infty}, H_{\et}^1\big(K^p, W(\ul{k}_{\Sigma_{\wp}},w)\big)^{\cH^p=\lambda_{\rho}}\Big)
    \\ \xlongrightarrow{\sim} \Hom_{\GL_2(F_{\wp})}\Big(\St\big(\ul{k}_{\Sigma_{\wp}},w;\alpha\big), \widetilde{H}^1_{\et}(K^p,E)_{\Q_p-\an}^{\cH^p=\lambda_{\rho}}\Big),
  \end{multline}where the first isomorphism follows from Thm.\ref{thm: clin-cjw}, the second from the fact that any vector in the second term is classical (see also Cor.\ref{prop: clin-ernr}), $\psi:=\chi \chi(\ul{k}_{\Sigma_{\wp}},w)^{-1}$ (cf. Cor.\ref{prop: clin-ernr}), the fourth from the adjunction formula for the classical Jacquet functor, and the last isomorphism follows from (\ref{equ: clin-apwgi}) and (\ref{equ: clin-wrs}) (and \cite[Cor.5.1.6]{Ding}).

  By the isomorphism (\ref{equ: clin-mcee}), there exists $\widetilde{v}\in (t_{\tau}^*\cM)^{\vee}$ such that $(E[\epsilon]/\epsilon^2)\cdot \widetilde{v} \cong \widetilde{\chi}$ and that $v\in \big(E[\epsilon]/\epsilon^2\big)\cdot \widetilde{v}$. By multiplying $\widetilde{t}$ by scalars in $E$, one can assume $v=\epsilon\widetilde{v}$. The  $T(F_{\wp})$-invariant map, by mapping a basis to $\widetilde{v}$,
 \begin{equation*}
   \widetilde{\chi}\hooklongrightarrow \cM(\cU)^{\vee}_b[\cH^p=\lambda_{\rho}]
 \end{equation*}
 induces a $\GL_2(F_{\wp})$-invariant map denoted by $\widetilde{v}$
 \begin{multline}\label{equ: clin-pqgi}
   \widetilde{v}: \big(\Ind_{\overline{B}(F_{\wp})}^{\GL_2(F_{\wp})} \widetilde{\chi}_{\tau}\delta^{-1}\big)^{\Q_p-\an} \hooklongrightarrow  \big(\Ind_{\overline{B}(F_{\wp})}^{\GL_2(F_{\wp})} \cM(\cU)^{\vee}_b[\delta^{-1}]\big)^{\Q_p-\an}[\cH^p=\lambda_{\rho}]\\ \xlongrightarrow{\text{(\ref{equ: clin-p2w})}} \widetilde{H}^1_{\et}(K^p,E\big)^{Z_1=\sN^{-w}}_{\overline{\rho}, \Q_p-\an}[\cH^p=\lambda_{\rho}].
 \end{multline}
 Similarly, the $T(F_{\wp})$-invariant map $\chi\hookrightarrow (\cM(\cU)^{\vee}_b)^{\cH^p=\lambda_{\rho}}$, by mapping a basis to $v$, induces a $\GL_2(F_{\wp})$-invariant map, denoted by $v$
 \begin{equation*}
   v: \big(\Ind_{\overline{B}(F_{\wp})}^{\GL_2(F_{\wp})}\chi\delta^{-1}\big)^{\Q_p-\an} \lra  \widetilde{H}^1_{\et}(K^p,E\big)^{Z_1=\sN^{-w}, \cH^p=\lambda_{\rho}}_{\overline{\rho}, \Q_p-\an}.
 \end{equation*}
 It's straightforward to see that the following diagram commutes
 \begin{equation}\label{equ: clin-pagip}
   \begin{CD}
    \big(\Ind_{\overline{B}(F_{\wp})}^{\GL_2(F_{\wp})} \chi\delta^{-1}\big)^{\Q_p-\an} @> v >> \widetilde{H}^1_{\et}(K^p,E)^{Z_1=\sN^{-w},\cH^p=\lambda_{\rho}}_{\overline{\rho}, \Q_p-\an} \\
    @VVV @VVV \\
      \big(\Ind_{\overline{B}(F_{\wp})}^{\GL_2(F_{\wp})} \widetilde{\chi}_{\tau}\delta^{-1}\big)^{\Q_p-\an} @> \widetilde{v} >>\widetilde{H}^1_{\et}(K^p,E)^{Z_1=\sN^{-w}}_{\overline{\rho}, \Q_p-\an}[\cH^p=\lambda_{\rho}]
   \end{CD}
 \end{equation}
 where the left arrow is induced by $\chi\hookrightarrow \widetilde{\chi}$, and the right arrow is the natural injection.

 By the same argument as in the proof of Lem.\ref{lem: clin-pvle}, one can prove the map $v$ factors through an injection
 \begin{equation*}
   v:\Sigma\big(\ul{k}_{\Sigma_{\wp}},w;\alpha\big) \hooklongrightarrow  \widetilde{H}^1_{\et}(K^p,E\big)^{Z_1=\sN^{-w}, \cH^p=\lambda_{\rho}}_{\overline{\rho}, \Q_p-\an}.
 \end{equation*}
 Moreover, we claim the restriction $f_v:=v|_{\St(\ul{k}_{\Sigma_{\wp}},w;\alpha)}$ is equal to $f$. In fact, by taking Jacquet-Emerton functor, one sees both the maps $f$ and $f_v$ give rise to the same eigenvector $$v\in  J_B\big(\widetilde{H}_{\et}^1(K^p,E)_{\Q_p-\an}^{Z_1=\sN^{-w}}\big)^{T(F_{\wp})=\chi, \cH^p=\lambda_{\rho}}\cong (z^*\cM)^{\vee},$$ from which the claim follows.

By the commutative diagram (\ref{equ: clin-pagip}) and Lem.\ref{lem: clin-pcssf}, we see $\widetilde{v}$ induces a continuous $\GL_2(F_{\wp})$-invariant injection
\begin{equation}\label{equ: clin-walge}
\Sigma\big(\ul{k}_{\Sigma_{\wp}},w;\alpha;\cL_{\tau}\big) \hooklongrightarrow\widetilde{H}^1_{\et}(K^p,E)^{Z_1=\sN^{-w}}_{\overline{\rho}, \Q_p-\an}[\cH^p=\lambda_{\rho}],
\end{equation}
whose restriction to $\St(\ul{k}_{\Sigma_{\wp}},w;\alpha)$ equals to $f$ (by the above discussion and commutative diagram (\ref{equ: clin-pagip})\big). It's sufficient to prove the map (\ref{equ: clin-walge}) factors through $\widetilde{H}^1_{\et}(K^p,E)^{Z_1=\sN^{-w},\cH^p=\lambda_{\rho}}_{\overline{\rho}, \Q_p-\an}$.

By the same argument as in the proof of Lem.\ref{lem: clin-pvle}, one can prove the following restriction map is injective:
 \begin{multline*}
    \Hom_{\GL_2(F_{\wp})}\Big(\Sigma\big(\ul{k}_{\Sigma_{\wp}},w;\alpha;\cL_{\tau}\big),\widetilde{H}^1_{\et}(K^p,E)_{\Q_p-\an}[\cH^p=\lambda_{\rho}]\Big) \\
    \lra \Hom_{\GL_2(F_{\wp})}\Big(\St\big(\ul{k}_{\Sigma_{\wp}},w;\alpha\big), \widetilde{H}^1_{\et}(K^p,E)_{\Q_p-\an}[\cH^p=\lambda_{\rho}]\Big).
  \end{multline*}
  For any $X\in \cH^p$, we know the restriction of the map $(X-\lambda_{\rho}(X))\widetilde{v}$ to $\St\big(\ul{k}_{\Sigma_{\wp}},w;\alpha\big)$ is zero (since the image $f$ lies in the $\lambda_{\rho}$-eigenspace), hence $(X-\lambda_{\rho}(X))\widetilde{v}=0$ \Big(here $X\widetilde{v}$ signifies the composition
$\Sigma\big(\ul{k}_{\Sigma_{\wp}},w;\alpha;\cL_{\tau}\big) \xrightarrow{\text{(\ref{equ: clin-walge})}}\widetilde{H}^1_{\et}(K^p,E)_{\Q_p-\an}[\cH^p=\lambda_{\rho}] \xrightarrow{X}\widetilde{H}^1_{\et}(K^p,E)_{\Q_p-\an}[\cH^p=\lambda_{\rho}]\Big)$, in other words, $\Ima(\widetilde{v})\in  \widetilde{H}^1_{\et}(K^p,E)_{\Q_p-\an}^{\cH^p=\lambda_{\rho}}$. This concludes the proof of Prop.\ref{prop: clin-awdl}.


\begin{thebibliography}{99}
\bibitem{BCh}
Bellaïche J., Chenevier G., {\em Families of Galois representations and Selmer groups}, Ast\'erisque 324, (2009).
\bibitem{Bergd}
Bergdall J., {\em On the variation of $(\varphi, \Gamma)$-modules over $p$-adic famillies of automorphic forms}, thesis, (2013), Brandeis University.
\bibitem{Berger}
Berger L., {\em Repr\'esentations p-adiques et \'equations diff\'erentielles}, Inventiones mathematicae 148(2), (2002), 219-284.


\bibitem{Br04}
Breuil C., {\em Invariant $\cL$ et s\'erie sp\'eciale p-adique}, Ann. Scient. de l'E.N.S. 37, 2004, 559-610.

\bibitem{Br080}
Breuil C., {\em Introduction g\'en\'erale aux volumes d'Ast\'erisque sur le programme de Langlands p-adique pour $\GL_2(\Q_p)$}, Ast\'erisque 319, (2008), 1-12.
\bibitem{Br10}
Breuil C., {\em S\'erie sp\'eciale p-adique et cohomologie \'etale compl\'et\'ee}, Ast\'erisque 331, (2010), 65-115.

\bibitem{Br0}
Breuil C., {\em The emerging $p$-adic Langlands programme}, Proceedings of I.C.M.2010, Vol II, 203-230.

\bibitem{BE}
Breuil C.,  Emerton, M., {\em Repr\'esentations $p$-adique ordinaires de $\GL_2(\Q_p)$ et compatibilit\'e local-global}, Ast\'erisque 331, (2010), 255-315.
\bibitem{Br00}
Breuil C., {\em Conjectures de classicit\'e sur les formes de Hilbert surconvergentes de pente finie}, unpublished note, march 2010.

\bibitem{Br}
Breuil C., {\em Remarks on some locally $\Q_p$ analytic representations of $\GL_2(F)$ in the crystalline case}, London Math. Soc. Lecture Note Series 393, (2012), 212-238.

\bibitem{Br13}
Breuil C., {\em Vers le socle localement analytique pour $\GL_n$ II}, to appear in Math. Annalen, (2013).


\bibitem{Bu}
Buzzard K., {\em Eigenvarieties}, London mathmatical society lecture note series 320, (2007), P.59.


\bibitem{Ca}
Carayol H., {\em Sur la mauvaise r\'eduction des courbes de Shimura}, Composito Mathematica 59, (1986) 151-230.

\bibitem{Ca2}
Carayol H., {\em Sur les repr\'esentations $l$-adiques associ\'ees aux formes modulaires de Hilbert}, Ann.scient. Ec. Norm. Sup. $4^e$ s\'erie, t.19, (1986), 409-468.

\bibitem{CW}
Casselman W., Wigner D., {\em Continuous cohomology and a conjecture of Serre's}, Inventiones mathematicae, 25(3), (1974), 199-211.


\bibitem{Che}
Chenevier G., {\em Familles $p$-adiques de formes automorphes pour $\GL_n$}, J. reine angew. Math, 570, (2004), 143-217.

\bibitem{Che11}
Chenevier G., {\em On the infinite fern of Galois representations of unitary type}, Ann. Sci. \'Ec. Norm. Sup\'er.(4), 44(6), (2011), 963-1019.
\bibitem{Colm2}
Colmez P., {\em Repr\'esentations triangulines de dimension $2$}, Ast\'erisque 319, (2008), 213-258.
\bibitem{Colm10}
Colmez P., {\em Invariants $\cL$ et d\'eriv\'ees de valeurs propres de Frobenius}, Ast\'erisque 331, (2010), 13-28.



\bibitem{Ding}
Ding Y., {\em Formes modulaires $p$-adiques sur les courbes de Shimura unitares et compatibilit\'e local-global}, thesis, preliminary version available at: \texttt{http://www.math.u-psud.fr/\~{}ding/fpc.pdf}.



\bibitem{Em04}
Emerton M., {\em Locally analytic vectors in representations of locally p-adic analytic groups}, to appear in Memoirs of the Amer. Math. Soc., (2004).


\bibitem{Em1}
Emerton M., {\em On the interpolation of systems of eigenvalues attached to automorphic Hecke eigenforms}, Invent. Math. 164, (2006), 1-84.

\bibitem{Em11}
Emerton M., {\em Jacquet Modules of locally analytic representations of $p$-adic reductive groups I. constructions and first properties}, Ann. Sci. E.N.S.39, no. 5, (2006), 775-839.

\bibitem{Em2}
Emerton M., {\em Jacquet modules of locally analytic representations of p-adic reductive groups II. The relation to parabolic induction}, J. Institut Math. Jussieu, (2007).

\bibitem{Em4}
Emerton M., {\em Local-global compatibiligty in the $p$-adic Langlands programme for $\GL_2/\Q$}, preprint, (2010).

\bibitem{Fon94}
Fontaine J.-M., {\em Le corps des p\'eriodes p-adiques}, Ast\'erisque 223, (1994), 59-102.

\bibitem{FO}
Fontaine J.-M., Ouyang Y., {\em Theory of $p$-adic Galois representations}, preprint (2008).


\bibitem{HT}
Harris M., Taylor R., {\em On the geometry and cohomology of some simple Shimura varieties}, Princeton university press, (2001).


\bibitem{KPX}
Kedlaya K., Pottharst J., Xiao L., {\em Cohomology of arithmetic families of $(\varphi,\Gamma)$-modules}, to appear in J. of the Amer. Math. Soc., (2012).

\bibitem{Ki}
Kisin M.,  {\em Overconvergent modular forms and the Fontaine-Mazur conjecture}, Inventiones mathematicae, 153(2), 373-454, (2003).
\bibitem{Liu}
Liu R., {\em Triangulation of refined families}, Ann Arbor, vol. 1001, (2012), p. 48-109.


\bibitem{Na}
Nakamura K., {\em Classification of two-dimensional split trianguline representations of $p$-adic field}, Compos. Math. 145, no.
4, (2009), 865-914.

\bibitem{New}
Newton J., {\em Completed cohomology of Shimura curves and a p-adic JacquetLanglands correspondence}, Mathematische Annalen, 355(2), (2013), 729-763.
\bibitem{Sa}
Saito T., {\em Hilbert modular forms and $p$-adic Hodge theory}, Compositio Mathematica 145, (2009) 1081-1113.

\bibitem{ST01B}
Schneider P.,  Teitelbaum J., {\em Locally analytic distributions and $p$-adic representation theory, with applications to $\GL_2$}, Journal of the American Mathematical Society 15.2, (2002), 443-468.
\bibitem{ST04hz}
Schneider P., Teitelbaum J., {\em Continuous and locally analytic representation theory}, (2002), available at:
\\
\texttt{http://www.math.uni-muenster.de/u/pschnei/publ/lectnotes/hangzhou.dvi}.
\bibitem{ST03}
Schneider P., Teitelbaum J., {\em Algebras of $p$-adic distributions and admissible representations}, Invent. math. 153, (2003), 145-196.


\bibitem{Sch10}
Schraen B., {\em Repr\'esentations p-adiques de $\GL_2(L)$ et cat\'egories d\'eriv\'ees}, Israel J. Math.
176, (2012), 307-361.

\bibitem{Sha}
Shah S., {\em Interpolating periods}, arXiv preprint arXiv:1305.2872, (2013).

\bibitem{Ta}
Taylor R., {\em Galois representations associated to Siegel modular forms of low weight}, Duke Math. J., 63(2), (1991), 281-332,.
\bibitem{TX}
Tian Y., Xiao L., {\em $p$-adic cohomology and classicality of overconvergent Hilbert modular forms}, preprint.
\bibitem{Wei}
Weibel C. A.,  \emph{An introduction to homological algebra}, Cambridge university press, (No. 38), (1995).
\bibitem{Zhang}
Zhang Y., {\em $\cL$-invariants and logarithm derivatives of eigenvalues of Frobenius},  Science China Mathematics, Volume 57,  Issue 8, (2014), 1587-1604.
\end{thebibliography}
\end{document}